\documentclass[a4paper, 12pt]{article}

\usepackage{amsmath} 
\usepackage{theorem}
\usepackage{amssymb}
\usepackage{amsfonts}
\usepackage{latexsym}
\usepackage{amscd}
\usepackage{diagramb}
\input GrCalc4.sty
\usepackage{graphicx}
\usepackage{enumerate}

\usepackage{pxfonts}

\usepackage{stmaryrd}

\DeclareMathAlphabet{\mathpzc}{OT1}{pzc}{m}{it}

\usepackage{xspace,colortbl}
\usepackage{color}

\definecolor{verde}{rgb}{0.,0.7,0.}
\definecolor{indigo}{rgb}{.18, .34, .78}
\definecolor{indigo1}{rgb}{.18, .24, .78}
\definecolor{indigo2}{rgb}{.18, .14, .78}
\definecolor{indigo3}{rgb}{.18, 0., .78}
\definecolor{rojo}{rgb}{1,0,0}

\definecolor{negro}{rgb}{0,0,0}
\definecolor{grey}{rgb}{0.5,0.5,0.5}

\definecolor{lila}{rgb}{.46, .16, .78}
\definecolor{lila1}{rgb}{.46, .16, .86}
\definecolor{lila2}{rgb}{.56, .16, .86}
	\definecolor{lila3}{rgb}{.63, .16, .78}
\definecolor{lila4}{rgb}{.7, .16, .78}
\definecolor{lila5}{rgb}{.78, .26, .78}
\definecolor{lila6}{rgb}{.6, 0., .78}

\theoremstyle{plain}
\theoremheaderfont{\normalfont\bfseries}

\newtheorem{thm}{Theorem}[section]

\newtheorem{lma}[thm]{Lemma}
\newtheorem{cor}[thm]{Corollary}
\newtheorem{defn}[thm]{Definition}

\newtheorem{prop}[thm]{Proposition}

\theorembodyfont{\rmfamily}

\newtheorem{rem}[thm]{Remark}
\newtheorem{ex}[thm]{Example}
\newcommand{\qed}{\hfill\quad\fbox{\rule[0mm]{0,0cm}{0,0mm}}  \par\bigskip}

\newcommand{\x}{\mbox{-}}

\newcommand{\s}{\hspace{0,06cm}}
\newcommand{\Cat}{\operatorname {Cat}}
\newcommand{\bEM}{{\rm bEM}}
\newcommand{\EM}{{\rm EM}}

\newcommand{\Mnd}{{\rm Mnd}}
\newcommand{\Comnd}{{\rm Comnd}}
\newcommand{\Bimnd}{{\rm Bimnd}}
\newcommand{\QB}{{\rm QB}}
\newcommand{\CQB}{{\rm CQB}}
\newcommand{\Tensor}{{\rm Tensor}}
\newcommand{\Fun}{{\rm Fun}}

\newcommand{\Del}{\boxtimes}
\newcommand{\comp}{\circ}
\newcommand{\iso}{\cong}

\newcommand{\C}{{\mathcal C}}
\newcommand{\Tau}{{\mathcal T}}
\newcommand{\M}{{\mathcal M}}
\newcommand{\D}{{\mathcal D}}
\newcommand{\F}{{\mathcal F}}

\newcommand{\N}{{\mathcal N}}
\newcommand{\A}{{\mathcal A}}
\newcommand{\B}{{\mathcal B}}

\newcommand{\X}{{\mathcal X}}
\newcommand{\Ll}{{\mathcal L}}

\newcommand{\U}{{\mathcal U}}
\newcommand{\YD}{{\mathcal YD}}

\newcommand{\crta}{\overline}

\newcommand{\Id}{\operatorname {Id}}
\newcommand{\id}{\operatorname {id}}

\newcommand{\Epsilon}{\varepsilon}

\def\K{{\mathcal K}}  

\def\u#1{\underline{#1}}

\newcommand{\Mod}{\operatorname{Mod}}
\newcommand{\Bimod}{\operatorname{Bimod}}

\newcommand{\cref}[1]{C.~\ref{c:#1}}

\newcommand{\lelabel}[1]{\label{le:#1}}
\newcommand{\leref}[1]{Lemma~\ref{le:#1}}
\newcommand{\eqlabel}[1]{\label{eq:#1}}
\newcommand{\equref}[1]{(\ref{eq:#1})}
\newcommand{\thlabel}[1]{\label{th:#1}}
\newcommand{\thref}[1]{Theorem~\ref{th:#1}}
\newcommand{\delabel}[1]{\label{de:#1}}
\newcommand{\deref}[1]{Definition~\ref{de:#1}}
\newcommand{\prlabel}[1]{\label{pr:#1}}
\newcommand{\prref}[1]{Proposition~\ref{pr:#1}}
\newcommand{\colabel}[1]{\label{co:#1}}
\newcommand{\coref}[1]{Corollary~\ref{co:#1}}

\newcommand{\rmlabel}[1]{\label{rm:#1}}
\newcommand{\rmref}[1]{Remark~\ref{rm:#1}}
\newcommand{\selabel}[1]{\label{se:#1}}
\newcommand{\seref}[1]{Section~\ref{se:#1}}
\newcommand{\sslabel}[1]{\label{ss:#1}}
\newcommand{\ssref}[1]{Subsection~\ref{ss:#1}}

\begin{document}

\title{ A bicategorical approach \\ to actions of monoidal categories} 
\author{Bojana Femi\'c \vspace{6pt} \\
{\small Facultad de Ingenier\'ia, \vspace{-2pt}}\\
{\small  Universidad de la Rep\'ublica} \vspace{-2pt}\\
{\small  Julio Herrera y Reissig 565,} \vspace{-2pt}\\
{\small  11 300 Montevideo, Uruguay}}

\date{}

\maketitle

\begin{abstract}
We characterize in terms of bicategories actions of monoidal categories to representation categories of algebras. 
For that purpose we introduce cocycles in any 2-category $\K$ and the category of Tambara modules over a monad $B$ in $\K$.  
We show that in an appropriate setting the above action of categories is given by a 2-cocycle in the Eilenberg-Moore category 
for the monad $B$. Furthermore, we introduce (co)quasi-bimonads in $\K$ and their respective 2-categories. We show that 
the categories of Tambara (co)modules over a (co)quasi-bimonad in $\K$ are monoidal, and how the 2-cocycles in the 
Eilenberg-Moore category corresponding to their actions are related to the Sweedler's and Hausser-Nill 2-cocycles in $\K$. 
We define (strong) Yetter-Drinfel`d modules in $\K$ as 1-endocells of the 2-category $\Bimnd(\K)$ of bimonads in $\K$, 
which we introduced in a previous paper. We prove that the monoidal category of Tamabra strong Yetter-Drinfel`d modules in $\K$ 
acts on the category of relative modules in $\K$. 
Finally, we show how the above-mentioned results on actions of categories come from pseudofunctors between appropriate bicategories. 
Our results are 2-categorical generalizations of several results known in the literature. 

\bigbreak

\medskip

\end{abstract}

\section{Introduction}

The purpose of this paper is to give a bicategorical interpretation of the actions of monoidal categories to representation categories of algebras by 
stating a general result in bicategorical terms. The main examples fitting this framework, and that we started with, are the following. On the one hand, 
the action of the category of comodules of a (coquasi-)bialgebra $H$ to the category ${}_A\M$ of modules over an $H$-module algebra $A$, 
which is provided by a crossed product on $A\# H$, 
from \cite{Sch} and \cite{Balan} in the coquasi case, and similarly on the other hand, the action of the category of modules of a (quasi-)bialgebra $H$ 
to the category ${}_A\M$, where $A$ is an $H$-comodule algebra. 

Crossed products were given a 2-categorical interpretation in \cite{LS} leading to the notions of a wreath, mixed wreath and their co- versions. 
In \cite{Femic5} studying the latter we formulated Sweedler's 2-cocycle, a 2-cocycle of the type of Hausser and Nill introduced in \cite{HN} 
and the dual versions of the two, all in a 2-categorical setting. The former comes packed in the data of a wreath, while the latter is packed in the data of a mixed wreath.  
In the same work we also introduced bimonads 
in 2-categories (not as opmonoidal monads as in \cite{McC, Moe}, but rather following the approach of \cite{Wisb} which suited better our purposes). 
All this presents a sufficient tool to express the above-mentioned actions of categories in bicategorical terms. (In this paper we do not study the dual statements, 
where ${}_A\M$ would be substituted by ${}^C\M$ being $C$ an $H$-(co)module coalgebra. They correspond to the cowreath and mixed cowreath 
constructions in the setting we explained above.)

Let $\K$ denote a 2-category and $\Mnd(\K)$ the 2-category of monads in $\K$, \cite{St1}. In \cite{Femic5} we introduced modules over monads and comodules over comonads in $\K$, as certain 1-cells. 
In the present paper we introduce {\em Tambara modules} over a monad $B$ in $\K$ as $B$-modules which are simultaneously objects in $\Mnd(\K)(B)$, the strict monoidal category 
of 1-endocells in $\Mnd(\K)$ over the 0-cell $B$, with a suitable compatibility condition. Moreover, we introduce quasi-bimonads and coquasi-bimonads in $\K$ 
and their respective 2-categories. We prove in \thref{quasi-bim monoidal} 
that the category of Tambara modules over a quasi-bimonad is non-strict monoidal. In a similar fashion we define Tambara comodules over a comonad $F$ (which are simultaneously objects in $\Mnd(\K)(B)$!) 
and by sort of duality we have that the category of Tambara comodules over a coquasi-bimonad is monoidal. The converse statements - if the category of Tambara (co)modules is monoidal that 
$F$ is a (co)quasi-bimonad - hold true under the assumption that the 1-cells of $\K$ posses elements. This is fulfilled for example in 
the 2-category $\Tensor$ of tensor categories, whose objects are tensor categories and given two such objects $\C$ and $\D$, the category $\Tensor(\C,\D):=\C\x\D\x\Bimod$ is the category of 
$C\x\D$-bimodule categories. A (co)quasi-bimonad in $\Tensor$ is a structure 
involving coring categories which we introduced in \cite{Femic3}.

The first main result of the present paper is \thref{main}. To formulate the statement we introduce the notion of 2- and 3-cocycles in any 2-category 
$\K$ as certain 2-cells. We consider a monad $B$ in $\K$ and a monoidal category $\C$ such that there is a quasi-monoidal functor $\F: \C\to\Mnd(\K)(B)$
that factors through a faithful quasi-monoidal 
functor $\U: \C\to\Mnd(\K)(B)$. We prove that there is an action of $\C$ on the category of $B$-modules in $\K$ if and only 
if there is an invertible normalized 2-cocycle over the 0-cell $B$ in the Eilenberg-Moore category $\EM^M(\K)$ for monads in $\K$. Specifying 
this result to the above-mentioned monoidal categories of Tambara (co)modules over a (co)quasi-bimonad $F$ in $\K$, we characterize in 
\prref{Sch case} and \prref{Martin case} the corresponding actions of categories. 
They yield the existence of what we call Sweedler's, respectively Hausser-Nill, Hopf datum in $\K$. 
Our result on the action of categories in the case of quasi-bimonads in $\K$ 
is a 2-categorical generalization of 
what Hausser and Nill proved in \cite[Section 9]{HN} in the setting of modules over a commutative ring. 

\medskip

When we introduced bimonads in 2-categories in \cite{Femic5} and the corresponding 2-category $\Bimnd(\K)$ of bimonads, we observed that the 
endomorphism 1-cells in $\Bimnd(\K)$ are a 2-categorical version of Yetter-Drinfel`d modules. In the present paper we call them {\em strong 
Yetter-Drinfel`d modules in $\K$}. We also consider such Yetter-Drinfel`d modules over a 0-cell {\em i.e.} bimonad $F$ which simultaneously are 
Tambara modules over a monad $B$. They form a monoidal category. We prove in \thref{YD action} that this monoidal category acts on the category of {\em relative 
$(F,B)$-modules in $\K$}, which we also introduce in the present paper. 

On the other hand, we introduce the 2-category $\tau\x\Bimnd(\K)$ of $\tau$-bimonads in $\K$ and show that there is an embedding 2-functor 
$\tau\x\Bimnd(\K)\to\Bimnd(\K)$. 
Fixing a 0-cell in $\tau\x\Bimnd(\K)$ yields a monoidal embedding of categories $\tau\x\Bimnd(\K)(F)\to\Bimnd(\K)(F)$. When $\K$ is induced by 
the category of vector spaces over a filed $k$, the subcategory in this embedding is the category of Yetter-Drinfel`d modules over $k$. 
For this reason the objects of $\tau\x\Bimnd(\K)(F)$ we call {\em classical  Yetter-Drinfel`d modules in $\K$}. In particular, 
classical Yetter-Drinfel`d modules in $\K$ (including those over a field) are strong. We prove that the corresponding 
monoidal category of ``Tambara classical Yetter-Drinfel`d modules'' is a monoidal subcategory of the category of ``Tambara strong Yetter-Drinfel`d modules''. 
Consequently, the former category acts on the category of relative $(F,B)$-modules in $\K$. This generalizes \cite[Theorem 2.3]{RL} to the 2-categorical setting.

On the other hand, it is well-known that a category action $\C\times\M\to\M$ of a monoidal category $\C$ corresponds (bijectively) to a monoidal functor $\C\to\Fun(\M, \M)$ 
to the monoidal category of endofunctors on $\M$. Moreover, given a bi/2-category $\K$ and a 0-cell $\A$ in $\K$, one has that 1-endocells on $\A$ form 
a (strict) monoidal category $\K(\A,\A)$. Then giving a pseudofunctor $\Tau: \K\to\Cat$, to the 2-category of categories, 
it induces a monoidal functor $\Tau_\A: \K(\A,\A)\to 
\Fun(\Tau(\A), \Tau(\A))$. We introduce new bi/2-categories and construct pseudofunctors from these to $\Cat$. The rest of our main results are 
\thref{pseudo bimonad}, \thref{pseudo bicat} and \prref{pseudo Martin}, in which we recover all previously obtained results on actions of monoidal categories. 

\medskip

As for the organization of the paper, we start by introducing 2- and 3-cocycles in $\K$, Sweedler's and Hausser-Nill Hopf datum in $\K$, the category of Tambara modules and 
show when it is monoidal. In Section 3 we define quasi-bimonads and coquasi-bimonads and their respective 2-categories. We prove in \thref{quasi-bim monoidal} that the 
category of Tambara modules over a quasi-bimonad in $\K$ is non-strict monoidal. In Section 4 we prove \thref{main} which studies the action of a monoidal category, from which there is a 
forgetful functor to $\Mnd(\K)(B)$ for a monad $B$ in $\K$, to the category of left $B$-modules in $\K$. Section 5 and 6 are dedicated to specializing the latter Theorem to 
the categories of Tambara modules over a (co)quasi-bimonad $F$ so that $B$ is an $F$-(co)module monad. \seref{Yetter} section studies Yetter-Drinfel`d modules in $\K$. Here is where we prove the 
action of the monoidal category of Tambara strong Yetter-Drinfel`d modules to the category of relative $(F,B)$-modules in $\K$. In the last section we construct pseudofunctors from 
certain bi/2-categories to $\Cat$ recovering the previously obtained results on actions of categories.

\section{Notation, terminology and preliminary results}

We assume the reader has basic knowledge of 2-categories, (co)monads, (co)wreaths, Eilenberg-Moore categories for 2-categories and actions of 
monoidal categories. For reference we recommend \cite{Be, Bo, Neu, St1, LS, EGNObook}. 
Throughout $\K$ will denote a 2-category, the horizontal composition of 2-cells we will denote by $\times$ and the vertical one by $\comp$. 
The Eilenberg-Moore category with respect to monads we will denote by $\EM^M(\K)$ and the one for comonads by $\EM^C(\K)$. 

In \cite[Definition 2.3]{Femic5} we defined modules over a monad 
and comodules over a comonad. 
As a matter of fact, such definition of modules appears in \cite{LS} under the name ``generalized t-algebra with domain $\X$ (in the definition below $t=T, \X=\B$). 
We recall here left modules and right comodules, the rest of the structures is defined analogously. 

\begin{defn} 
Let $(\A, T, \mu, \eta)$ be a monad and $(\A, D, \Delta, \Epsilon)$ a comonad in $\K$.
\begin{enumerate}[(a)]
\item A 1-cell $F:\B\to\A$ in $\K$ is called a {\em left $T$-module} if there is a 2-cell $\nu: TF\to F$ such that $\nu(\mu\times\Id_F)=\nu(\Id_T\times\nu)$
and $\nu(\eta\times\Id_F)=\Id_F$ holds.
\item A 1-cell $F:\A\to\B$ in $\K$ is called a {\em right $D$-comodule} if there is a 2-cell $\rho: F\to FD$ such that $(\Id_F\times\Delta)\rho=(\rho\times\Id_D)\rho$
and $(\Id_F\times\Epsilon)\rho=\Id_F$ holds.
\end{enumerate}
\end{defn}

As the composition of composable 1-cells in $\K$ gives a monoidal structure on such 1-cells and the corresponding 2-cells, we will use freely string diagram notation in our computations. 
Multiplication and unit of a monad, commultiplication and counit of a comonad, left action and coaction and right coaction we write respectively: 
$$
\gbeg{2}{1}
\gmu \gnl
\gend \qquad 
\gbeg{1}{1}
\gu{1} \gnl
\gend  \qquad 
\gbeg{2}{1}
\gcmu \gnl
\gend  \qquad 
\gbeg{1}{1}
\gcu{1} \gnl
\gend  \qquad 
\gbeg{2}{1}
\glm \gnl
\gend  \qquad 
\gbeg{1}{1}
\glcm \gnl
\gend  \qquad 
\gbeg{1}{1}
\grcm \gnl
\gend\quad .
$$

Let $\Mnd(\K)$ denote the 2-category of monads and $\Comnd(\K)$ of comonads. 
We will use 1-cells $(X, \tau)$ in the four 2-categories: $\Mnd(\K), \Mnd(\K^{op}), \Comnd(\K)$ and $\Comnd(\K^{op})$, where 
$\K^{op}$ denotes the 2-category which differs from $\K$ in that the 1-cells appear in reversed order (the diagrams there are left-right symmetric to those for $\K$), 
and even we will deal with pairs $(B, \tau)$ which will simultaneously be 1-cells in all the four 2-categories. For this reason, 
to simplify the formulations we will often say for the distributive law $\tau$ that it is: ``left monadic, right monadic, left comonadic, right comonadic'', respectively. 
Although it has nothing to do with (co)monadic adjunctions, this terminology we find more concise than to say that $\tau$ is a distributive law 
``with respect to the (co)monadic structure in the left/right coordinate'', for example. Similarly, for a 2-cell $\zeta$ in $\K$ which is a 2-cell in some of, or in various of the above four 
2-categories, we will often say ``$\tau$ is natural with respect to $\zeta$''. 

At last let us say that we are going to use the term ``Yang-Baxter equation'' in a more general form. Namely, we will use it for an identity between 2-cells acting on three a priori 
different 1-cells in $\K$, resembling the Rademeister move III.

\bigskip

Now that we fixed notation and terminology, we proceed to some new definitions and first results.

\begin{defn} \delabel{2-coc in K}
Let $\rho$ be an operator that to any pair of composable 1-cells $Y,Z$ 
assigns a 2-cell $\rho_{Y,Z}: YZ\to YZ$ in $\K$, such that $(\zeta\times\xi)\comp\rho_{Y,Z}=\rho_{Y',Z'}\comp(\zeta\times\xi)$ for every pair of 2-cells 
$\zeta: Y\to Y', \xi: Z\to Z'$. Then $\rho$ is called a {\em 2-cocycle in $\K$} if the following holds: 
\begin{equation} \eqlabel{inv 2-coc}
(\Id_X\times\rho_{Y,Z})\comp \rho_{X,YZ}=(\rho_{X,Y}\times\Id_Z)\comp \rho_{XY,Z}.
\end{equation}
for every triple of composable 1-cells $X: \C\to\D, Y: \B\to\C, Z: \A\to \B$ in $K$. Moreover, $\rho$ is {\em a normalized 2-cocycle} if 
$\rho_{Id_{\A}, X}=\rho_{X,Id_{\A}}=id_X$ for every 1-cell $X: \A\to\A$. 
\end{defn}

Differently stated, a 2-cocycle is a family 
$$\rho=(\rho_{Y,Z}: YZ\to YZ \hspace{0,2cm}\vert\hspace{0,2cm} Y,Z \in\K)$$
of natural morphisms in the product category $\K(\B, \C)\times\K(\A, \B)$, one for each object $(Y,Z)\in\K(\B, \C)\times\K(\A, \B)$, 
satisfying condition \equref{inv 2-coc}.

If a 2-cocycle $\rho$ is invertible, being his inverse the operator $\rho^{-1}_{V,W}: VW\to VW$, then observe that it fulfills the identity: 
\begin{equation} \eqlabel{2-coc}
\rho_{X,YZ}^{-1}\comp(\Id_X\times\rho_{Y,Z}^{-1})=
\rho_{XY,Z}^{-1}\comp(\rho_{X,Y}^{-1}\times\Id_Z)
\end{equation}
The above definition can be generalized to any $n$-cocycle in the obvious way, we write out the 3-cocycle condition, the rest is not of our interest in this moment. 
Given three composable 1-cells $X,Y,Z$ an operator $\rho: XYZ\to XYZ$ natural in the three components is a {\em 3-cocycle in $\K$} if: 
\begin{equation} \eqlabel{3-coc}
(\Id_X\times\rho_{Y,Z,W})\comp \rho_{X,YZ,W}\comp(\rho_{X,Y,Z}\times\Id_W)=\rho_{X,Y,ZW}\comp\rho_{XY,Z,W}
\end{equation}
holds, where $W$ is a fourth 1-cell which is composable with $Z$. We say that $\rho$ is {\em a normalized 3-cocycle} if 
$\rho_{Id_{\D}, X,Y}=\rho_{X,Id_{\C},Y}=\rho_{X,Y, Id_{\B}}=id_{XY}$ for every 1-cells $X,Y$ as above. 

\medskip

For any 0-cell $\A$ in $\K$ there is a strict monoidal 1-category $\K(\A)$ whose objects are 1-cells $X:\A\to\A$ and 
morphisms 2-cells $\zeta: X\to Y$ in $\K$. The objects of $\K(\A)$ are composable 1-cells in $\K$. If the monoidal category 
$\K(\A)$ happens to be a tensor category (most importantly abelian), then the above definition of a 2-cocycle recovers that of 
2-cocycles in the Yetter cohomology, called by its introduction in \cite{Y1}, but it was independantly introduced also in \cite{Dav}.

\medskip 


In \cite[Definition 4.5]{Femic6} we introduced monad Hopf data. There we required that the 2-cells $\psi_{B,F}$ 
and $\mu_M$ have specific forms. However, a monad Hopf datum can be considered in a more general setting, we resume it here. 
Consider a wreath $F$ around $B$ given by $(B,F,\psi, \mu_M, \eta_M)$ with the associated 7 axioms. Suppose there are 2-cells 
$\Epsilon_B:=
\gbeg{1}{2}
\got{1}{B} \gnl
\gcu{1}  \gnl
\gend, 
\Epsilon_F:=
\gbeg{1}{2}
\got{1}{F} \gnl
\gcu{1}  \gnl
\gend$ so that 
$\gbeg{2}{3}
\got{1}{B} \got{1}{B} \gnl
\gmu \gnl
\gvac{1} \hspace{-0,22cm} \gcu{1}  \gnl
\gend=
\gbeg{2}{3}
\got{1}{B} \got{1}{B} \gnl
\gcl{1} \gcl{1} \gnl
\gcu{1} \gcu{1} \gnl
\gend$. When you apply $\Epsilon_B$ and independantly $\Epsilon_F$ to the 7 axioms of the wreath, you get 14 new axioms. 
The obtained data, together with $\Epsilon_B, \Epsilon_F$, where we only suppose that $B$ is a monad, is a {\em monad Hopf datum}. Though, if we lack of the 2-cell $\Epsilon_B$ 
and we apply only $\Epsilon_F$ to the 7 axioms of the wreath $(B,F,\psi, \mu_M, \eta_M)$, we obtain the following 7 axioms: \vspace{-0,5cm} 
 \begin{center} \hspace{-1,5cm} 
\begin{tabular}{p{7.2cm}p{0cm}p{6cm}}
\begin{equation}\eqlabel{F mod alg}
\gbeg{3}{5}
\got{1}{B}\got{1}{B}\got{1}{F}\gnl
\gcl{1} \glmptb \gnot{\hspace{-0,34cm}\psi} \grmptb \gnl
\grm \gcl{1} \gnl
\gwmu{3} \gnl
\gob{3}{B}
\gend=
\gbeg{3}{5}
\got{1}{B}\got{1}{B}\got{1}{F}\gnl
\gmu \gcn{1}{1}{1}{0} \gnl
\gvac{1} \hspace{-0,32cm} \grm \gnl
\gvac{1} \gcl{1} \gnl
\gvac{1} \gob{1}{B}
\gend
\end{equation} & & 
\begin{equation}\eqlabel{F mod alg unit}
\gbeg{2}{4}
\got{3}{F} \gnl
\gu{1} \gcl{1} \gnl
\grm \gnl
\gob{1}{B}
\gend=
\gbeg{2}{4}
\got{1}{F} \gnl
\gcu{1} \gnl
\gu{1} \gnl
\gob{1}{B}
\gend
\end{equation}
\end{tabular}
\end{center}
\vspace{-0,6cm}
$$ \textnormal{ \footnotesize  module monad}  \hspace{5cm}  \textnormal{\footnotesize module monad unity} $$ \vspace{-0,7cm}

\pagebreak

\begin{center} 
\begin{tabular} {p{6cm}p{1cm}p{6cm}} 
\begin{equation} \eqlabel{weak action} 
\gbeg{3}{6}
\got{1}{B} \got{1}{F} \got{1}{F} \gnl
\glmptb \gnot{\hspace{-0,34cm}\psi} \grmptb \gcl{1} \gnl
\gcl{1} \glmptb \gnot{\hspace{-0,34cm}\psi} \grmptb \gnl
\glmptb \gnot{\hspace{-0,34cm}\sigma} \grmpt \gcl{1} \gnl
\gwmu{3} \gnl
\gob{3}{B}
\gend=
\gbeg{3}{5}
\got{1}{B} \got{1}{F} \got{1}{F}\gnl
\gcl{1} \glmptb \gnot{\hspace{-0,34cm}\mu_M} \grmptb \gnl
\grm \gcl{1} \gnl
\gwmu{3} \gnl
\gob{3}{B}
\gend
\end{equation} & & 
\begin{equation} \eqlabel{weak action unity} 
\gbeg{3}{4}
\got{1}{} \got{3}{B}\gnl
\glmpb \gnot{\hspace{-0,34cm}\eta_M} \grmpb \gcl{1} \gnl
\gcu{1} \gmu \gnl
\gob{4}{B}
\gend=
\gbeg{3}{5}
\got{1}{B}\gnl
\gcl{1} \glmpb \gnot{\hspace{-0,34cm}\eta_M} \grmpb \gnl
\grm \gcl{1} \gnl
\gwmu{3} \gnl
\gob{3}{B}
\gend
\end{equation} 
\end{tabular}
\end{center}
\vspace{-0,6cm}
$$ \textnormal{ \footnotesize twisted action}  \hspace{5,5cm}  \textnormal{\footnotesize twisted action unity} $$ \vspace{-0,7cm}

\begin{center} \hspace{4cm} 
\begin{tabular} {p{6cm}p{1cm}p{6.8cm}} 
\begin{equation}\eqlabel{2-cocycle condition}
\gbeg{3}{6}
\got{1}{F} \got{1}{F} \got{1}{F}\gnl
\glmpb \gnot{\beta} \gcmpb \grmptb \gnl
\gcl{1} \glmptb \gnot{\hspace{-0,34cm}\mu_M} \grmptb \gnl
\glmpt \gnot{\hspace{-0,34cm}\sigma} \grmptb \gcl{1} \gnl
\gvac{1} \gmu \gnl
\gob{4}{B}
\gend=
\gbeg{3}{6}
\got{1}{F} \got{1}{F} \got{1}{F} \gnl
\glmptb \gnot{\hspace{-0,34cm}\mu_M} \grmptb \gcl{1} \gnl
\gcl{1} \glmptb \gnot{\hspace{-0,34cm}\psi} \grmptb \gnl
\glmpt \gnot{\hspace{-0,34cm}\sigma} \grmptb \gcl{1} \gnl
\gvac{1} \gmu \gnl
\gob{4}{B}
\gend
\end{equation} & & 
\begin{equation}\eqlabel{normalized 2-cocycle}
\gbeg{3}{6}
\got{1}{} \got{3}{F}\gnl
\glmpb \gnot{\hspace{-0,34cm}\eta_M} \grmpb \gcl{1} \gnl
\gcl{1} \glmptb \gnot{\hspace{-0,34cm}\psi} \grmptb \gnl
\glmpt \gnot{\hspace{-0,34cm}\sigma} \grmptb \gcl{1} \gnl
\gvac{1} \gmu \gnl
\gob{4}{B}
\gend=
\gbeg{1}{4}
\got{1}{F}\gnl
\gcu{1} \gnl
\gu{1} \gnl
\gob{1}{B}
\gend=
\gbeg{3}{5}
\got{1}{F}\gnl
\gcl{1} \glmpb \gnot{\hspace{-0,34cm}\eta_M} \grmpb \gnl
\glmpt \gnot{\hspace{-0,34cm}\sigma} \grmptb \gcl{1} \gnl
\gvac{1} \gmu \gnl
\gob{4}{B}
\gend
\end{equation} 
\end{tabular}
\end{center} \vspace{-0,5cm}
$$ \textnormal{ \footnotesize 2-cocycle condition}  \hspace{5,5cm}  \textnormal{\footnotesize normalized 2-cocycle} $$ \vspace{-0,7cm}
Here 
$\gbeg{2}{3}
\got{1}{B} \got{1}{F} \gnl
\grm \gnl
\gob{1}{B} 
\gend:=
\gbeg{2}{4}
\got{1}{B} \got{1}{F} \gnl
\glmptb \gnot{\hspace{-0,34cm}\psi_{B,F}} \grmptb \gnl
\gcu{1} \gcl{1} \gnl
\gob{3}{B} 
\gend, 
\sigma:=
\gbeg{2}{4}
\got{1}{F} \got{1}{F} \gnl
\glmptb \gnot{\hspace{-0,34cm}\mu_M} \grmptb \gnl
\gcu{1} \gcl{1} \gnl
\gob{3}{B} 
\gend$ and $\beta: FFF\to FFF$ is a 3-cocycle on $FFF$ in $\K$. The 2-cell \vspace{0,7cm}

\noindent $\beta$ in \equref{2-cocycle condition} is an identity in a wreath. In the context of this paper it will appear non-trivial in \seref{rep coq}. 
In \cite{Femic5, Femic6} we showed that the upper data (where $\beta$ is identity) determines Sweedler's 
crossed product in a 2-categorical setting. To the above 2-cell $\sigma$ we will refer to as to Sweedler's 2-cocycle in $\K$. 
For this reason, the data $(B,F,\psi,\mu_M,\eta_M,\Epsilon_F, \beta)$, where $B$ is a monad and $F$ a 1-cell, both over a 0-cell $\A$, so that 
\equref{F mod alg}--\equref{normalized 2-cocycle} hold, we will call {\em Sweedler's Hopf datum}. 
The relation between 2-cocycles in \equref{2-cocycle condition} and \equref{2-coc} will be clarified in \leref{relation Sweedler cocycle}.

\medskip

The name ``module monad'' in \equref{F mod alg}--\equref{F mod alg unit} in the literature is usually termed {\em $F$ measures $B$}. A priori, $F$ is not a monad. 
Observe that if the 2-cocycle $\sigma$ is trivial, that is if $\sigma=\eta_B(\Epsilon_F\times\Epsilon_F)$, 
\equref{weak action}--\equref{weak action unity} state that $B$ is a proper right $F$-module, 
where a priori non-associative product on $F$ is given by 
$\gbeg{2}{3}
\got{1}{F} \got{1}{F} \gnl
\gmu \gnl
\gob{2}{F} 
\gend:=
\gbeg{2}{4}
\got{1}{F} \got{1}{F} \gnl
\glmptb \gnot{\hspace{-0,34cm}\mu_M} \grmptb \gnl
\gcl{1} \gcu{1} \gnl
\gob{1}{F} 
\gend$.

\bigskip

In a similar fashion as we introduced Sweedler's Hopf datum, we do the folowing. Consider a mixed wreath 
$(B,F, \psi, \Delta_M, \Epsilon_M)$, that is, a cowreath $F$ around a monad $B$, and assume there is a 
2-cell $\eta_F:=
\gbeg{1}{3} \gnl
\gu{1} \gnl
\gob{1}{F} 
\gend$. Apply $\eta_F$ to the 7 axioms of the mixed wreath to obtain: 

 \begin{center} \hspace{-1,5cm} 
\begin{tabular}{p{7.2cm}p{0cm}p{6cm}}
\begin{equation}\eqlabel{F comod alg}
\gbeg{3}{5}
\got{1}{B}\got{3}{B}\gnl
\gcl{1} \glcm \gnl
\glmptb \gnot{\hspace{-0,34cm}\psi} \grmptb \gcl{1} \gnl
\gcl{1} \gmu \gnl
\gob{1}{F} \gob{2}{B}
\gend=
\gbeg{3}{5}
\got{1}{\hspace{0,22cm}B}\got{1}{\hspace{0,22cm}B}\gnl
\gvac{1} \hspace{-0,34cm} \gmu \gnl
\gvac{1} \hspace{-0,22cm} \glcm \gnl
\gvac{1} \gcl{1} \gcl{1} \gnl
\gvac{1} \gob{1}{F} \gob{1}{B}
\gend
\end{equation} & &
\begin{equation}\eqlabel{F comod alg unit}
\gbeg{2}{4}
\got{3}{} \gnl
\gvac{1} \gu{1} \gnl
\glcm \gnl
\gob{1}{F} \gob{1}{B}
\gend=
\gbeg{2}{4}
\got{1}{} \gnl
\gu{1} \gu{1} \gnl
\gcl{1} \gcl{1} \gnl
\gob{1}{F} \gob{1}{B}
\gend
\end{equation}

\\  {\hspace{2cm} \footnotesize comodule monad} & &  { \hspace{1cm} \footnotesize comodule monad unity} \\
\end{tabular}
\end{center} 

\begin{center} \hspace{-1,4cm} 
\begin{tabular} {p{7cm}p{0.6cm}p{6cm}} 
\begin{equation} \eqlabel{quasi coaction} 
\gbeg{4}{5}
\gvac{3} \got{1}{B} \gnl
\gvac{2} \glcm \gnl
\glmpb \gnot{\Delta_M} \gcmpb \grmptb \gcl{1} \gnl
\gcl{1} \gcl{1} \gmu \gnl
\gob{1}{F} \gob{1}{F} \gob{2}{B}
\gend=
\gbeg{3}{6}
\got{1}{B} \gnl
\gcl{1} \gcn{1}{1}{2}{1} \gelt{\s\Phi_{\lambda}} \gcn{1}{1}{0}{1} \gnl
\glmptb \gnot{\hspace{-0,34cm}\psi} \grmptb \gcl{1} \gcl{2} \gnl
\gcl{1} \glmptb \gnot{\hspace{-0,34cm}\psi} \grmptb \gnl
\gcl{1} \gcl{1} \gmu \gnl
\gob{1}{F} \gob{1}{F} \gob{2}{B}
\gend
\end{equation} & &
\begin{equation} \eqlabel{quasi coaction counity} 
\gbeg{2}{5}
\got{3}{B} \gnl
\glcm \gnl
\gbmp{\Epsilon_M} \gcl{1} \gnl
\gmu \gnl
\gob{2}{B} 
\gend=
\gbeg{2}{5}
\got{1}{B} \gnl
\gcl{1} \gu{1} \gnl
\gcl{1} \gbmp{\Epsilon_M} \gnl
\gmu \gnl
\gob{2}{B} 
\gend
\end{equation} 
\end{tabular}
\end{center} \vspace{-0,5cm}
$$ \textnormal{ \footnotesize quasi coaction}  \hspace{5,5cm}  \textnormal{\footnotesize quasi coaction counity} $$ \vspace{-0,7cm}

\begin{center} 
\begin{tabular}{p{7.2cm}p{1cm}p{6.8cm}}
\begin{equation} \eqlabel{3-cocycle cond fi-lambda} 
\gbeg{5}{5}
\gvac{2} \gcn{1}{1}{2}{1} \gelt{\s\Phi_{\lambda}} \gcn{1}{1}{0}{1} \gnl
\glmpb \gnot{\Delta_M} \gcmpb \grmptb \gcl{1} \gcl{1} \gnl
\gcl{1} \gcl{1}  \glmptb \gnot{\hspace{-0,34cm}\psi} \grmptb \gcl{1} \gnl
\glmptb \gnot{\beta} \gcmptb \grmptb \gmu \gnl
\gob{1}{F} \gob{1}{F} \gob{1}{F} \gob{2}{B} \gnl
\gend=
\gbeg{5}{5}
\gcn{1}{1}{2}{1} \gelt{\s\Phi_{\lambda}} \gcn{2}{2}{0}{5}  \gnl
\gcl{1} \gcl{1} \gnl
\gcl{1} \glmptb \gnot{\Delta_M} \gcmpb \grmpb \gcl{1} \gnl
\gcl{1} \gcl{1} \gcl{1} \gmu \gnl
\gob{1}{F} \gob{1}{F} \gob{1}{F} \gob{2}{B} \gnl
\gend
\end{equation} &  &
\begin{equation}\eqlabel{normalized 3-cocycle fi-lambda}
\gbeg{3}{5}
\gcn{1}{1}{2}{1} \gelt{\s\Phi_{\lambda}} \gcn{1}{1}{0}{1} \gnl 
\gbmp{\Epsilon_M} \gcl{1} \gcl{2} \gnl
\glmptb \gnot{\hspace{-0,34cm}\psi} \grmptb \gnl
\gcl{1} \gmu \gnl
\gob{1}{F} \gob{2}{B} \gnl
\gend=
\gbeg{2}{3}
\gu{1} \gu{1} \gnl
\gcl{1} \gcl{1} \gnl
\gob{1}{F} \gob{1}{B} \gnl
\gend=
\gbeg{3}{4}
\gcn{1}{1}{2}{1} \gelt{\s\Phi_{\lambda}} \gcn{1}{1}{0}{1} \gnl 
\gcl{1} \gbmp{\Epsilon_M} \gcl{1} \gnl
\gcl{1} \gmu \gnl
\gob{1}{F} \gob{2}{B} \gnl
\gend
\end{equation}
\end{tabular}
\end{center} \vspace{-0,5cm}
$$ \textnormal{ \footnotesize 2-cocycle condition for $\Phi_{\lambda}$}  \hspace{5cm}  \textnormal{\footnotesize normalized 2-cocycle $\Phi_{\lambda}$} $$ 
Here 
$\gbeg{2}{3}
\got{3}{B} \gnl
\glcm \gnl
\gob{1}{F} \gob{1}{B}
\gend:=
\gbeg{2}{4}
\got{1}{B} \gnl
\gcl{1} \gu{1} \gnl
\glmptb \gnot{\hspace{-0,34cm}\psi} \grmptb \gnl
\gob{1}{F} \gob{1}{B} 
\gend, 
\gbeg{3}{3}
\gcn{1}{1}{2}{1} \gelt{\s\Phi_{\lambda}} \gcn{1}{1}{0}{1} \gnl  %
\gcl{1} \gcl{1} \gcl{1} \gnl
\gob{1}{F} \gob{1}{F} \gob{1}{B} \gnl
\gend:=
\gbeg{3}{4}
\gvac{1} \gu{1} \gnl
\glmpb \gnot{\Delta_M} \gcmptb \grmpb \gnl
\gcl{1} \gcl{1} \gcl{1} \gnl
\gob{1}{F} \gob{1}{F} \gob{1}{B} \gnl
\gend$ and $\beta: FFF\to FFF$ is a 3-cocycle on $FFF$. The 2-cell $\beta$ in \equref{3-cocycle cond fi-lambda} is an identity in a mixed wreath. 
As above, 
it 
will be non-trivial in the context appearing in \seref{rep quasi}. 
In \cite{Femic5} we showed that the above data (where $\beta$ is identity) determines a comodule algebra $B$ over a bimonad $F$ in a 2-categorical setting. 
(In \cite{HN} $F$ was actually a quasi-bialgebra over a commutative ring $R$ and the rest of the structure is the same, for $\K$ the 2-category induced by the monoidal category of modules over $R$. )
For this reason, the data $(B,F, \psi, \Delta_M, \Epsilon_M, \eta_F, \beta)$, where $B$ is a monad and $F$ a 1-cell, both over a 0-cell $\A$, so that the identities 
\equref{F comod alg} -- \equref{normalized 3-cocycle fi-lambda} hold, we will call {\em a Hausser-Nill datum}.  In \leref{relation HN cocycle} we will explain the relation 
between the ``2-cocycle'' $\Phi_{\lambda}$ in \equref{3-cocycle cond fi-lambda} and a 2-cocycle in $\K$.

\medskip

Again, the name ``comodule monad'' in \equref{F comod alg}--\equref{F comod alg unit} is conditional, we could say that {\em $F$ comeasures $B$}. A priori, $F$ is not a comonad. 
Observe that if the 2-cocycle $\Phi_{\lambda}$ is trivial, \equref{quasi coaction}--\equref{quasi coaction counity} state that $B$ is a proper left $F$-comodule, 
where a priori non-coassociative coproduct on $F$ is given by 
$\gbeg{2}{3}
\got{2}{F} \gnl
\gcmu \gnl
\gob{1}{F} \gob{1}{F} 
\gend:=
\gbeg{3}{4}
\got{1}{F} \gnl
\glmptb \gnot{\Delta_M} \gcmpb \grmpb \gnl
\gcl{1} \gcl{1} \gcu{1} \gnl
\gob{1}{F} \gob{1}{F} 
\gend$.

\subsection{Categories of Tambara modules}

Let us recall the 2-category of monads $\Mnd(\K)$. Its 0-cells are monads $(\A, B)\cong (\A, B, \mu, \eta)$, 
1-cells are morphisms of monads, that is $(X, \psi): (\A, B)\to(\A', B')$ where $X:\A\to\A'$ is a 1-cell in $\K$ and 
$\psi: B'X\to XB$ is a 2-cell in $\K$ so that \equref{psi laws} hold, and 2-cells are $\zeta: (X,\psi)\to(Y,\psi')$ where $\zeta: X\to Y$ is a 2-cell in $\K$ so that 
\equref{2-cells in Mnd} holds. 
\begin{center} 
\begin{tabular}{p{8cm}p{1cm}p{4.5cm}}
\begin{equation}\eqlabel{psi laws}
\gbeg{3}{5}
\got{1}{B'}\got{1}{B'}\got{1}{X}\gnl
\gcl{1} \glmpt \gnot{\hspace{-0,34cm}\psi} \grmptb \gnl
\glmptb \gnot{\hspace{-0,34cm}\psi} \grmptb \gcl{1} \gnl
\gcl{1} \gmu \gnl
\gob{1}{X} \gob{2}{B}
\gend=
\gbeg{3}{5}
\got{1}{B'}\got{1}{B'}\got{1}{B'}\gnl
\gmu \gcn{1}{1}{1}{0} \gnl
\gvac{1} \hspace{-0,32cm} \glmptb \gnot{\hspace{-0,34cm}\psi} \grmptb  \gnl
\gvac{1} \gcl{1} \gcl{1} \gnl
\gvac{1} \gob{1}{X} \gob{1}{B}
\gend,
\qquad
\gbeg{2}{4}
\got{3}{X} \gnl
\gu{1} \gcl{1} \gnl
\glmptb \gnot{\hspace{-0,34cm}\psi} \grmptb \gnl
\gob{1}{X} \gob{1}{B}
\gend=
\gbeg{3}{4}
\got{1}{X} \gnl
\gcl{1} \gu{1} \gnl
\gcl{1} \gcl{1} \gnl
\gob{1}{X} \gob{1}{B}
\gend
\end{equation} & & 
\begin{equation}\eqlabel{2-cells in Mnd}
\gbeg{2}{4}
\got{1}{B'} \got{1}{X}\gnl
\glmptb \gnot{\hspace{-0,34cm}\psi} \grmptb \gnl
\gbmp{\zeta}  \gcl{1} \gnl
\gob{1}{Y} \gob{1}{B}
\gend=
\gbeg{2}{4}
\got{1}{B'} \got{1}{X}\gnl
\gcl{1} \gbmp{\zeta} \gnl
\glmptb \gnot{\hspace{-0,34cm}\psi'} \grmptb \gnl
\gob{1}{Y} \gob{1}{B}
\gend
\end{equation}
\end{tabular}
\end{center} 
The horizontal composition of 1-cells: $(X,\tau_{B,X}):(\A, B)\to(\A', B')$ and $(Y,\tau_{B',Y}):(\A', B')\to(\A'', B'')$, 
where $\tau_{B,X}: B'X\to XB$ and $\tau_{B',Y}: B'' Y\to Y B'$, is given by: 
\begin{equation} \eqlabel{tau B XY}
(Y,\tau_{B',Y})(X,\tau_{B,X})=(YX, \tau_{B,YX})=(YX, \hspace{0,3cm}
\gbeg{3}{4}
\got{1}{B''} \got{1}{Y} \got{1}{X} \gnl
\glmptb \gnot{\hspace{-0,34cm}\tau_{B', Y}} \grmptb \gcl{1} \gnl
\gcl{1} \glmptb \gnot{\hspace{-0,34cm}\tau_{B,X}} \grmptb \gnl
\gob{1}{Y} \gob{1}{X} \gob{1}{B}
\gend).
\end{equation}
The vertical and horizontal composition of 2-cells is given as in $\K$.
The identity 1-cell on a 0-cell $(\A, B)$ is given by: $(\id_{\A}, \Id_B): (\A, B)\to(\A, B)$. %
The identity 2-cell on a 1-cell $(X,\tau_{B,X}): (\A, B)\to(\A', B')$ is given by $\Id_X$. 

\medskip

The 2-category of comonads is $\Comnd(\K)=\Mnd(\K_{op})$, where $\K_{op}$ differs from $\K$ in that the 2-cells are reversed with respect to those in $\K$. This means 
that in $\Comnd(\K)$ one has 2-cells $\phi: XB\to B'X$ and $\zeta:X\to Y$ which satisfy the up-side down versions of diagrams \equref{psi laws} -- \equref{2-cells in Mnd}. 

\smallskip

The strict monoidal category $\Mnd(\K)(B)$ has then for objects monad morphisms $(X, \psi): (\A, B)\to(\A, B)$, where $X:\A\to\A$ is a 1-cell and 
$\psi: BX\to XB$ is a 2-cell in $\K$, and morphisms are $\zeta: (X,\psi)\to(Y,\psi')$, where $\zeta: X\to Y$ is a 2-cell in $\K$, so that \equref{psi laws} 
and \equref{2-cells in Mnd} hold with $B=B'$. From now on we will denote $\Tau(\A, B):=\Mnd(\K)(B)$. 
This monoidal category was studied in \cite{BC} and the notation is to 
evoke Tambara who studied in \cite{Tamb} a monoidal category of {\em transfer morphisms} which is $\Tau(\A, B)$ for $\K$ being 
the 2-category induced by the monoidal category of vector spaces. 


\medskip

It is clear that $B$ is a (left and right) module over itself. Given a 2-cell $\tau_{B,B}:BB\to BB$ it is very well known that the composition 1-cell $BB$ is a monad with structure 2-cells:
\begin{equation}  \eqlabel{wreath (co)product}
\nabla_{BB}=
\gbeg{3}{4}
\got{1}{B} \got{1}{B} \got{1}{B} \got{1}{B}  \gnl
\gcl{1}  \glmptb \gnot{\hspace{-0,34cm}\tau_{B,B}} \grmptb \gcl{1} \gnl
\gmu \gmu \gnl
\gob{2}{B} \gob{2}{B}
\gend\hspace{1,5cm}
\eta_{BB}=
\gbeg{2}{3}
\gu{1} \gu{1} \gnl
\gcl{1} \gcl{1} \gnl
\gob{1}{B} \gob{1}{B}
\gend
\end{equation}
if and only if 
$(B, \tau_{B,B})$ is a 1-cell both in $\Mnd(\K)$ and in $\Mnd(\K^{op})$. 
{\em We will assume throughout that given a monad $B$ the 2-cell $\tau_{B,B}$ is both a left and a right monadic distributive law. }

Given two left $B$-modules $X,Y:\A'\to\A$ in $\K$, for a 2-cell $\zeta: X\to Y$ we will say that it is left $B$-linear, or a morphism of left $B$-modules, if the following is fulfilled: 
$$\gbeg{2}{4}
\got{1}{B} \got{1}{X} \gnl
\glm \gnl
\gvac{1} \gbmp{\zeta} \gnl
\gob{1}{} \gob{1}{Y} 
\gend=
\gbeg{2}{4}
\got{1}{B} \got{1}{X} \gnl
\gcl{1} \gbmp{\zeta} \gnl
\glm \gnl
\gob{1}{} \gob{1}{Y.} 
\gend
$$

\begin{defn}
Let $B:\A\to\A$ be a monad in $\K$. We denote by ${}_{(\A, B)}\Tau$ the following category. 
Its objects are triples $(X,\tau_{B,X}, \nu)$ where $(X,\tau_{B,X})$ are objects of $\Tau(\A, B)$ such that $(X, \nu)$ is a left $B$-module in $K$ and the action 
$\nu: (BX, \tau_{B, BX})\to (X, \tau_{B, X})$ is a morphism in $\Tau(\A, B)$. 
Morphisms of ${}_{(\A, B)}\Tau$ are left $B$-linear morphisms in $\Tau(\A, B)$. 

The objects of ${}_{(\A, B)}\Tau$ we will call {\em left Tambara $B$-modules}. 
\end{defn}

When the action $\nu: (BX, \tau_{B, BX})\to (X, \tau_{B, X})$ is a morphism in $\Tau(\A, B)$, we will also say that {\em $\tau_{B,X}$ is natural with respect to the left action}, meaning: 
\begin{equation} \eqlabel{nat lm}
\gbeg{3}{5}
\got{1}{B} \got{1}{B} \got{1}{X} \gnl
\gcl{1} \glm \gnl
\glmptb \gnot{\tau_{B,X}} \gcmp \grmptb \gnl
\gcl{1} \gvac{1} \gcl{1} \gnl
\gob{1}{X} \gob{3}{B}
\gend=
\gbeg{3}{5}
\got{1}{B} \got{1}{B} \got{1}{X} \gnl
\glmptb \gnot{\hspace{-0,34cm}\tau_{B,B}} \grmptb \gcl{1} \gnl
\gcl{1} \glmptb \gnot{\hspace{-0,34cm}\tau_{B,X}} \grmptb \gnl
\glm \gcl{1} \gnl
\gvac{1} \gob{1}{X} \gob{1}{B}
\gend
\end{equation}

We will abuse notation in that we will denote also by $(X,\tau_{B,X})$ the objects of ${}_{(\A, B)}\Tau$ without expliciting the $B$-action. 

Observe that $(B, \tau_{B,B})$ is an object of ${}_{(\A, B)}\Tau$ and that given any object $(X,\tau_{B,X})$ in ${}_{(\A, B)}\Tau$ 
the left $B$-action on $X$ is a morphism in ${}_{(\A, B)}\Tau$. 


\begin{prop} \prlabel{B act XY quasi-bim}
Let $B:\A\to\A$ be a monad in $\K$ such that there is a 2-cell 
$\gbeg{2}{3}
\got{2}{B} \gnl
\gcmu \gnl 
\gob{1}{B} \gob{1}{B} \gnl
\gend$ 
satisfying:
$$
\gbeg{3}{5}
\got{1}{B} \got{3}{B} \gnl
\gwmu{3} \gnl
\gvac{1} \gcl{1} \gnl
\gwcm{3} \gnl
\gob{1}{B}\gvac{1}\gob{1}{B}
\gend=
\gbeg{4}{5}
\got{2}{B} \got{2}{B} \gnl
\gcmu \gcmu \gnl
\gcl{1} \glmptb \gnot{\hspace{-0,34cm}\tau_{B,B}} \grmptb \gcl{1} \gnl
\gmu \gmu \gnl
\gob{2}{B} \gob{2}{B} \gnl
\gend\quad\textnormal{and}\quad 
\gbeg{2}{3}
\gu{1}  \gu{1} \gnl
\gcl{1} \gcl{1} \gnl
\gob{1}{B} \gob{1}{B}
\gend=
\gbeg{2}{3}
\gu{1} \gnl
\hspace{-0,34cm} \gcmu \gnl
\gob{1}{B} \gob{1}{B.}
\gend 
$$
If $(X,\tau_{B,X}),(Y,\tau_{B,Y})\in {}_{(\A, B)}\Tau$, then $(XY,\tau_{B,XY})\in {}_{(\A, B)}\Tau$ with the $B$-action given by:   
\begin{equation} \eqlabel{B act XY tau}
\gbeg{3}{5}
\got{1}{B} \got{3}{XY} \gnl
\gcl{1} \gvac{1} \gcl{2} \gnl
\gcn{1}{1}{1}{3}  \gnl
\gvac{1} \glm \gnl
\gvac{2} \gob{1}{XY} 
\gend=
\gbeg{3}{5}
\got{2}{B} \got{1}{X} \got{1}{Y} \gnl
\gcmu \gcl{1} \gcl{2} \gnl
\gcl{1} \glmptb \gnot{\hspace{-0,34cm}\tau_{B,X}} \grmptb \gnl
\glm \glm \gnl
\gvac{1} \gob{1}{X} \gob{3}{Y}
\gend
\end{equation}
and where $\tau_{B,XY}$ is from \equref{tau B XY}. 
\end{prop}

\begin{proof}
The proof that $XY$ is a $B$-module via \equref{B act XY tau} is direct, using that $\tau_{B,X}$ is natural with respect to 
$\gbeg{2}{3}
\got{1}{B} \got{1}{X} \gnl
\glm \gnl 
\gob{3}{X} \gnl
\gend$ and that it is a left monadic distributive law. 
\qed\end{proof}

As a matter of fact, that \equref{B act XY tau} endows $XY$ with a structure of a left $B$-module is a consequence of the following result. 

\begin{prop} \prlabel{B act XM psi}
Let $B:\A\to\A$ be a monad, $(X, \psi_{B,X})$ an object in $\Tau(\A,B)$ and $M:\A'\to\A$ a left $B$-module in $\K$. 
Then $XM$ is a left $B$-module via
\begin{equation} \eqlabel{B act XY psi}
\gbeg{3}{4}
\got{1}{B} \got{3}{XM} \gnl
\gcn{1}{1}{1}{3} \gvac{1} \gcl{1} \gnl
\gvac{1} \glm \gnl
\gvac{2} \gob{1}{XM} 
\gend=
\gbeg{3}{4}
\got{1}{B} \got{1}{X} \got{1}{M} \gnl
\glmptb \gnot{\hspace{-0,34cm}\psi_{B,X}} \grmptb \gcl{1} \gnl
\gcl{1} \glm \gnl
\gob{1}{X} \gob{3}{M}
\gend
\end{equation}
\end{prop}

\begin{rem} \rmlabel{psi versus tau}
The proof of the above Proposition is direct. Its result was used in the literature for $\K$ being the 2-category induced by the 
monoidal category of vector spaces, see {\em e.g.} \cite{Sch}. In the proof of  \prref{B act XY quasi-bim} one is actually proving 
that under specified conditions $(X, \psi_{B,X})$ is an object in $\Tau(\A,B)$ with $\psi_{B,X}$ given by: 
$$\psi_{B,X}=
\gbeg{3}{5}
\got{2}{B} \got{1}{X} \gnl
\gcmu \gcl{1} \gnl
\gcl{1}  \glmptb \gnot{\hspace{-0,34cm}\tau_{B,X}} \grmptb \gnl
\glm \gcl{1} \gnl
\gob{1}{} \gob{1}{X} \gob{1}{B.} 
\gend$$
\end{rem}

Let us record few more properties that are in the line of the above results. Let $B:\A\to\A$ be a monad with a 2-cell $\Epsilon_B:B\to\Id_{\A}$ so that 
$\Id_{\A}$ is a left $B$-module by $\Epsilon_B$. The left hand-side version of the definition of a module monad from \equref{F mod alg}, saying that a 
left $B$-module and monad $F:\A\to\A$ is a left $B$-module monad, is: 
\begin{equation} \eqlabel{left mod monad}
\gbeg{3}{5}
\got{1}{B} \got{1}{F} \got{3}{F}\gnl
\gcl{1} \gwmu{3} \gnl
\gcn{1}{1}{1}{3} \gvac{1} \gcl{2} \gnl
\gvac{1} \glm \gnl
\gob{5}{F}
\gend=
\gbeg{3}{5}
\got{1}{B} \got{1}{F} \got{1}{F} \gnl
\glmptb \gnot{\hspace{-0,34cm}\psi} \grmptb \gcl{1} \gnl
\gcl{1} \glm \gnl
\gwmu{3} \gnl
\gob{3}{F}
\gend \hspace{2,5cm}
\gbeg{3}{4}
\got{1}{B} \gnl
\gcl{1} \gu{1} \gnl
\glm \gnl
\gob{3}{F}
\gend=
\gbeg{2}{4}
\got{1}{B} \gnl
\gcu{1} \gnl
\gu{1} \gnl
\gob{1}{F.}
\gend
\end{equation} 
The following are straightforwardly proved:

\begin{prop}
Let  $(F:\A\to\A, \mu_F, \eta_F)$ be a monad and a left $B$-module and assume that $(F,\psi_{B,F})$ is an object of $\Tau(\A,B)$. 
Then $F$ is a left $B$-module monad if and only if $\mu_F$ and $\eta_F$ are left $B$-linear, where $FF$ is a left $B$-module by \equref{B act XY psi}. 
\end{prop}

\begin{prop}
Let $(F,\psi_{B,F})$ be an object of $\Tau(\A,B)$ and consider $BF$ as a left $B$-module by $\mu_B\times id_F$ and $FB$ by \equref{B act XY psi}. 
Then $\psi_{B,F}:BF\to FB$ is left $B$-linear. 
\end{prop}

\section{Quasi-bimonads and coquasi-bimonads in 2-categories} \selabel{quasi}

Given a monad $B$ on a 0-cell $\A$ in $\K$. The identity 1-cell $\Id_{\A}$ is trivially a comonad 
and we can consider the monad of the 2-cells $\Id_{\A}\to B$ in $\K$, which is indeed a convolution algebra in the monoidal category $\K(\A)$.

\begin{defn}
A quasi-bimonad in $\K$ is an octuple $(\A, F, \mu, \eta, \Delta, \Epsilon, \tau_{F,F}, \Phi)$, where $(\A, F, \mu, \eta)$ is a monad, 
$\tau_{F,F}:FF\to FF$ is a left and right monadic and comonadic distributive law, the 2-cells 
$\Delta=\gbeg{2}{3}
\got{2}{F} \gnl
\gcmu \gnl 
\gob{1}{F} \gob{1}{F} \gnl
\gend$ and 
$\Epsilon=\gbeg{2}{2}
\got{1}{F} \gnl
\gcu{1} \gnl 
\gend$ 
satisfy:
$$
\gbeg{2}{4}
\got{2}{F} \gnl
\gcmu \gnl
\gcu{1} \gcl{1} \gnl
\gob{1}{} \gob{1}{F}
\gend=
\gbeg{1}{4}
\got{1}{F} \gnl
\gcl{2} \gnl
\gob{1}{F}
\gend=
\gbeg{2}{4}
\got{2}{F} \gnl
\gcmu \gnl
\gcl{1} \gcu{1} \gnl
\gob{1}{F} 
\gend, \quad
\gbeg{3}{5}
\got{1}{F} \got{3}{F} \gnl
\gwmu{3} \gnl
\gvac{1} \gcl{1} \gnl
\gwcm{3} \gnl
\gob{1}{F}\gvac{1}\gob{1}{F}
\gend=
\gbeg{4}{5}
\got{2}{F} \got{2}{F} \gnl
\gcmu \gcmu \gnl
\gcl{1} \glmptb \gnot{\hspace{-0,34cm}\tau_{F,F}} \grmptb \gcl{1} \gnl
\gmu \gmu \gnl
\gob{2}{F} \gob{2}{F} \gnl
\gend,\quad 
\gbeg{2}{3}
\gu{1}  \gu{1} \gnl
\gcl{1} \gcl{1} \gnl
\gob{1}{F} \gob{1}{F}
\gend=
\gbeg{2}{3}
\gu{1} \gnl
\hspace{-0,34cm} \gcmu \gnl
\gob{1}{F} \gob{1}{F}
\gend,\quad 
\gbeg{2}{3}
\got{1}{F} \got{1}{F} \gnl
\gcl{1} \gcl{1} \gnl
\gcu{1}  \gcu{1} \gnl
\gend=
\gbeg{2}{3}
\got{1}{F} \got{1}{F} \gnl
\gmu \gnl
\gvac{1} \hspace{-0,2cm} \gcu{1} \gnl
\gend,\quad 
\gbeg{1}{2}
\gu{1} \gnl
\gcu{1} \gnl
\gob{1}{}
\gend=
\Id_{id_{\A}}
$$
and the 2-cell $\Phi: \Id_\A\to FFF$ is convolution invertible and normalized: 
$\gbeg{3}{3}
\gcn{1}{1}{2}{1} \gelt{\s\Phi} \gcn{1}{1}{0}{1} \gnl  %
\gcl{1} \gcu{1} \gcl{1} \gnl
\gob{1}{F} \gob{3}{F} 
\gend=
\gbeg{3}{3}
\gu{1} \gu{1} \gnl
\gcl{1} \gcl{1} \gnl
\gob{1}{F} \gob{1}{F} 
\gend
$ and it obeys:  
\begin{center} \hspace{-0,6cm}
\begin{tabular}{p{6.6cm}p{-1cm}p{8.6cm}}
\begin{equation} \eqlabel{quasi coass.}
\gbeg{5}{7}
\gvac{3} \got{2}{F} \gnl
\gcn{1}{1}{2}{0} \gelt{\Phi} \gcn{1}{1}{0}{1} \gcmu \gnl
\gcn{1}{3}{0}{0} \gcn{1}{1}{1}{0} \glmptb \gnot{\hspace{-0,34cm}\tau_{F,F}} \grmptb \gcl{1} \gnl
\gvac{1} \gcn{1}{1}{0}{0} \hspace{-0,24cm} \gcmu \hspace{-0,2cm} \gmu \gnl 
\gvac{2} \hspace{-0,2cm} \glmptb \gnot{\hspace{-0,34cm}\tau_{F,F}} \grmptb \gcl{1} \gcl{2} \gnl
\gvac{1} \gmu \gmu \gnl
\gvac{1} \gob{2}{F} \gob{2}{F} \gob{1}{F} 
\gend=
\gbeg{6}{8}
\gvac{1} \got{2}{F} \gnl
\gvac{1} \gcmu \gnl
\gcn{2}{1}{3}{2} \gcl{2} \gnl
\gcmu \gvac{1} \gcn{1}{1}{2}{1} \gelt{\Phi} \gcn{1}{1}{0}{1} \gnl 
\gcl{2} \gcl{1} \glmptb \gnot{\hspace{-0,34cm}\tau_{F,F}} \grmptb \gcl{1} \gcl{2} \gnl
\gvac{1} \glmptb \gnot{\hspace{-0,34cm}\tau_{F,F}} \grmptb \glmptb \gnot{\hspace{-0,34cm}\tau_{F,F}} \grmptb \gnl
\gmu \gmu \gmu \gnl
\gob{2}{F} \gob{2}{F} \gob{2}{F} 
\gend
\end{equation} & & 
\begin{equation} \eqlabel{3-coc. cond.} 
\gbeg{9}{12}
\gvac{1} \gcn{1}{2}{2}{-2} \gelt{\s\Phi} \gcn{1}{1}{0}{2} \gvac{1} \gcn{1}{1}{2}{0} \gelt{\s\Phi} \gcn{2}{2}{0}{4}  \gnl
\gvac{2} \gcn{2}{1}{1}{2} \hspace{-0,22cm} \glmptb \gnot{\hspace{-0,34cm}\tau_{F,F}} \grmptb \gcn{1}{1}{2}{5} \gnl
\gcl{1} \gvac{2} \glmptb \gnot{\hspace{-0,34cm}\tau_{F,F}} \grmptb \gcn{1}{1}{1}{5} \gvac{2} \gcl{4} \gcl{5} \gnl
\gwmu{4} \gcn{1}{1}{1}{5} \gvac{2} \gcl{1} \gnl
\gvac{2} \gcn{1}{7}{0}{0} \gcn{1}{2}{2}{0} \gelt{\s\Phi} \gcn{1}{1}{0}{1} \gcl{1} \gcl{1} \gnl
\gvac{4} \gcn{1}{1}{1}{0} \glmptb \gnot{\hspace{-0,34cm}\tau_{F,F}} \grmptb \gcl{1} \gnl
\gvac{2} \gcn{1}{3}{0}{0} \gcn{1}{2}{0}{0} \gcn{1}{1}{0}{0} \hspace{-0,22cm} \gcmu \hspace{-0,2cm} \gmu \gcn{1}{1}{1}{0} \gnl
\gvac{5} \hspace{-0,34cm} \glmptb \gnot{\hspace{-0,34cm}\tau_{F,F}} \grmptb \gcl{1} \glmptb \gnot{\hspace{-0,34cm}\tau_{F,F}} \grmptb \gcn{1}{1}{2}{1} \gnl
\gvac{4} \gmu \gmu \gcn{1}{1}{1}{0}  \glmptb \gnot{\hspace{-0,34cm}\tau_{F,F}} \grmptb\gnl
\gvac{5} \hspace{-0,22cm} \gcl{1} \gvac{1} \glmptb \gnot{\hspace{-0,34cm}\tau_{F,F}} \grmptb \gcn{1}{1}{2}{1} \gcn{1}{2}{2}{2} \gnl
\gvac{5} \gwmu{3} \gmu \gnl
\gvac{3} \gob{2}{F} \gob{3}{F} \gob{2}{F} \gob{2}{F} \gnl
\gend=
\gbeg{7}{7}
\gcn{1}{1}{2}{1} \gelt{\s\Phi} \gcn{1}{1}{0}{1} \gcn{1}{1}{2}{1} \gelt{\s\Phi} \gcn{1}{1}{0}{1} \gnl
\gcn{1}{1}{1}{0} \gcn{1}{1}{1}{0} \glmptb \gnot{\hspace{-0,34cm}\tau_{F,F}} \grmptb\gcn{2}{2}{1}{4} \gcn{1}{2}{-1}{2}  \gnl
\gcn{1}{2}{0}{0} \gcn{1}{1}{0}{0} \hspace{-0,22cm} \gcmu \gcn{1}{1}{0}{2} \gnl
\gvac{1} \glmptb \gnot{\hspace{-0,34cm}\tau_{F,F}} \grmptb \gcl{1} \gcmu \gcl{1} \gcl{2} \gnl
\gmu \gmu \gcl{1} \glmptb \gnot{\hspace{-0,34cm}\tau_{F,F}} \grmptb \gnl
\gcn{1}{1}{2}{2} \gvac{1} \gcn{1}{1}{2}{2} \gvac{1} \gmu \gmu \gnl
\gob{2}{F} \gob{2}{F} \gob{2}{F} \gob{2}{F} \gnl
\gend
\end{equation}
\end{tabular}
\end{center} 
$$ \textnormal{ \hspace{-3cm} \footnotesize quasi coassociativity}  \hspace{5,5cm}  \textnormal{\footnotesize 3-cocycle condition} $$ 
We will often write shortly $(\A, F, \Phi)$ for a quasi-bimonad $(\A, F, \mu, \eta, \Delta, \Epsilon, \tau_{F,F}, \Phi)$. 
\end{defn}

When we deal with Tambara modules over a quasi-bimonad $(\A, F, \Phi)$ it makes sense to require that the 2-cells $\tau_{F,X}$ assigned 
to Tambara modules $X$ be also (right) comonadic distributive laws. As a matter of fact, this property will be needed 
in order to prove that the category of Tambara modules over a quasi-bimonad $(\A, F, \Phi)$ is monoidal. For this proof we will also need two more assumptions. 
Then it becomes natural to introduce the following 2-category. 

\medskip 

Let $\QB(\K)$ denote the 2-category of quasi-bimonads in $\K$, it consists of the following: 

\medskip

\underline{0-cells:} are quasi-bimonads $(\A, F, \tau_{F,F}, \Phi)\cong(\A, F, \mu, \eta, \Delta, \Epsilon, \tau_{F,F}, \Phi)$ in $\K$ such that 
\begin{equation} \eqlabel{Phi nat new} \hspace{-2cm}
\gbeg{4}{5}
\got{7}{F} \gnl
\gcn{1}{1}{2}{1} \gelt{\s\Phi} \gcn{1}{1}{0}{1} \gcl{3} \gnl  %
\gcl{2} \gcl{2} \gcl{2} \gnl
\gob{1}{F} \gob{1}{F} \gob{1}{F} \gob{1}{F} 
\gend=
\gbeg{2}{6}
\got{1}{F} \gnl
\gcl{1} \gcn{1}{1}{2}{1} \gelt{\s\Phi} \gcn{1}{1}{0}{1} \gnl 
\glmptb \gnot{\hspace{-0,34cm}\tau_{F,F}} \grmptb \gcl{1} \gcl{1} \gnl
\gcl{1} \glmptb \gnot{\hspace{-0,34cm}\tau_{F,F}} \grmptb \gcl{1} \gnl
\gcl{1} \gcl{1} \glmptb \gnot{\hspace{-0,34cm}\tau_{F,F}} \grmptb \gnl
\gob{1}{F} \gob{1}{F} \gob{1}{F} \gob{1}{F} 
\gend
\end{equation} 
$$\textnormal{\hspace{-1cm} \footnotesize $\tau_{F,\Id_\A}$ is natural w.r.t. $\Phi$}$$
holds;

\underline{1-cells:} are pairs $(X,\tau_{F,X}): (\A, F, \tau_{F,F}, \Phi)\to(\A', F', \tau_{F',F'}, \Phi')$ where 
$(X, \tau_{F,X}): (\A, F, \mu, \eta)\to(\A', F', \mu', \eta')$ is a 1-cell in $\Mnd(\K)$ and
$(X, \tau_{F,X}): (\A, F, \Delta, \Epsilon)\to(\A', F', \Delta', \Epsilon')$ is a 1-cell in $\Comnd(\K^{op})$, that is, the following identities hold:
\vspace{-1,4cm}
\begin{center} \hspace{-0,6cm}
\begin{tabular}{p{7.4cm}p{0cm}p{8cm}}
\begin{equation}\eqlabel{monadic d.l.}
\gbeg{3}{5}
\got{1}{F'}\got{1}{F'}\got{1}{X}\gnl
\gcl{1} \glmpt \gnot{\hspace{-0,34cm}\tau_{F,X}} \grmptb \gnl
\glmptb \gnot{\hspace{-0,34cm}\tau_{F,X}} \grmptb \gcl{1} \gnl
\gcl{1} \gmu \gnl
\gob{1}{X} \gob{2}{F}
\gend=
\gbeg{3}{5}
\got{1}{F'}\got{1}{F'}\got{1}{X}\gnl
\gmu \gcn{1}{1}{1}{0} \gnl
\gvac{1} \hspace{-0,34cm} \glmptb \gnot{\hspace{-0,34cm}\tau_{F,X}} \grmptb  \gnl
\gvac{1} \gcl{1} \gcl{1} \gnl
\gvac{1} \gob{1}{X} \gob{1}{F}
\gend;
\quad
\gbeg{2}{5}
\got{3}{X} \gnl
\gu{1} \gcl{1} \gnl
\glmptb \gnot{\hspace{-0,34cm}\tau_{F,X}} \grmptb \gnl
\gcl{1} \gcl{1} \gnl
\gob{1}{X} \gob{1}{F}
\gend=
\gbeg{3}{5}
\got{1}{X} \gnl
\gcl{1} \gu{1} \gnl
\gcl{2} \gcl{2} \gnl
\gob{1}{X} \gob{1}{F}
\gend
\end{equation} & &
 \begin{equation}\eqlabel{comonadic d.l.}
\gbeg{3}{5}
\got{2}{F'} \got{1}{X} \gnl
\gcmu \gcl{1} \gnl
\gcl{1} \glmptb \gnot{\hspace{-0,34cm}\tau_{F,X}} \grmptb \gnl
\glmptb \gnot{\hspace{-0,34cm}\tau_{F,X}} \grmptb \gcl{1} \gnl
\gob{1}{X}\gob{1}{F} \gob{1}{F} 
\gend=
\gbeg{3}{5}
\got{1}{F'} \got{1}{X} \gnl
\gcl{1} \gcl{1} \gnl
\glmpt \gnot{\hspace{-0,34cm}\tau_{F,X}} \grmptb \gnl
\gcn{1}{1}{1}{0} \hspace{-0,22cm} \gcmu \gnl
\gob{1}{X} \gob{1}{F} \gob{1}{F} 
\gend;
\quad
\gbeg{3}{5}
\got{1}{F'} \got{1}{X} \gnl
\gcl{1} \gcl{1} \gnl
\glmptb \gnot{\hspace{-0,34cm}\tau_{F,X}} \grmptb \gnl
\gcl{1} \gcu{1} \gnl
\gob{1}{X}
\gend=
\gbeg{3}{5}
\got{1}{F'} \got{1}{X} \gnl
\gcl{1} \gcl{1} \gnl
\gcu{1} \gcl{2}  \gnl
\gob{3}{X}
\gend
\end{equation}
\end{tabular}
\end{center} 
and the following compatibility conditions between $\tau_{F,X}$ and $\Phi$, on one side, and between $\tau_{F,X}$ and $\tau_{F,F}$ on the other, 
are fulfilled:
\begin{center} 
\begin{tabular}{p{4.5cm}p{1,5cm}p{4.5cm}}
\begin{equation} \eqlabel{Phi nat} \hspace{-2,2cm}
\gbeg{4}{5}
\got{1}{X} \gnl
\gcl{1} \gcn{1}{1}{2}{1} \gelt{\s\Phi} \gcn{1}{1}{0}{1} \gnl  %
\gcl{2} \gcl{2} \gcl{2} \gcl{2} \gnl
\gob{1}{X} \gob{1}{F} \gob{1}{F} \gob{1}{F} 
\gend=
\gbeg{2}{6}
\got{7}{X} \gnl
\gcn{1}{1}{2}{1} \gelt{\s\Phi'} \gcn{1}{1}{0}{1} \gcl{1} \gnl 
\gcl{1} \gcl{1} \glmptb \gnot{\hspace{-0,34cm}\tau_{F,X}} \grmptb \gnl
\gcl{1} \glmptb \gnot{\hspace{-0,34cm}\tau_{F,X}} \grmptb \gcl{1} \gnl
\glmptb \gnot{\hspace{-0,34cm}\tau_{F,X}} \grmptb \gcl{1} \gcl{1} \gnl
\gob{1}{X} \gob{1}{F} \gob{1}{F} \gob{1}{F} 
\gend
\end{equation} & & 
\begin{equation} \eqlabel{YBE BBX} 
\gbeg{3}{5}
\got{1}{F'} \got{1}{F'} \got{1}{X} \gnl
\gcl{1} \glmptb \gnot{\hspace{-0,34cm}\tau_{F,X}} \grmptb \gnl
\glmptb \gnot{\hspace{-0,34cm}\tau_{F,X}} \grmptb \gcl{1} \gnl
\gcl{1} \glmptb \gnot{\hspace{-0,34cm}\tau_{F,F}} \grmptb \gnl
\gob{1}{X} \gob{1}{F} \gob{1}{F}
\gend=
\gbeg{3}{5}
\got{1}{F'} \got{1}{F'} \got{1}{X} \gnl
\glmptb \gnot{\hspace{-0,34cm}\tau_{F',F'}} \grmptb \gcl{1} \gnl
\gcl{1} \glmptb \gnot{\hspace{-0,34cm}\tau_{F,X}} \grmptb \gnl
\glmptb \gnot{\hspace{-0,34cm}\tau_{F,X}} \grmptb \gcl{1} \gnl
\gob{1}{X} \gob{1}{F} \gob{1}{F.}
\gend
\end{equation} 
\end{tabular}
\end{center} 
$$ \textnormal{ \hspace{-1,7cm} \footnotesize $\tau_{\Id_\A,X}$ is natural w.r.t. $\Phi$} \hspace{4,5cm} \textnormal{\footnotesize YBE for $FFX$} $$ \vspace{-0,7cm}

\medskip

\underline{2-cells:} are 2-cells in $\Mnd(\K)$, that is, $\zeta: (X,\tau_{F,X})\to(Y,\tau_{F,Y})$ 
so that the identity:
\begin{equation}  \eqlabel{2-cells QB}
\gbeg{2}{4}
\got{1}{F'} \got{1}{X}\gnl
\glmptb \gnot{\hspace{-0,34cm}\tau_{F,X}} \grmptb \gnl
\gbmp{\zeta} \gcl{1} \gnl
\gob{1}{Y} \gob{1}{F}
\gend=
\gbeg{2}{4}
\got{1}{F'} \got{1}{X}\gnl
\gcl{1} \gbmp{\zeta} \gnl
\glmptb \gnot{\hspace{-0,34cm}\tau_{F,Y}} \grmptb \gnl
\gob{1}{Y} \gob{1}{F}
\gend
\end{equation} 
holds.



The horizontal composition of 1-cells, the vertical and horizontal composition of 2-cells, 
the identity 1-cell on a 0-cell and 
the identity 2-cell on a 1-cell 
are given as in $\Mnd(\K)$.

\bigskip

\begin{rem}
For any $(X, \tau_{F,X})\in\Mnd(\K)$ the 2-cell 
$\tau_{FF,X}=
\gbeg{3}{4}
\got{1}{F'} \got{1}{F'} \got{1}{X} \gnl
\gcl{1} \glmptb \gnot{\hspace{-0,34cm}\tau_{F,X}} \grmptb \gnl
\glmptb \gnot{\hspace{-0,34cm}\tau_{F,X}} \grmptb \gcl{1} \gnl
\gob{1}{X} \gob{1}{F} \gob{1}{F} 
\gend$ 
fulfills the conditions \equref{monadic d.l.}--\equref{comonadic d.l.}. On the other hand, note that 
$\tau_{\Id_\A,X}$ is nothing but identity 2-cell on $X$ in $\K$. Then to say that $\tau_{\Id_\A,X}$ is natural with respect to 
$\Phi: \Id_\A\to FFF$, is to say that: $(X\times\Phi)\comp\tau{\Id_\A,X}=\tau_{FFF,X}\comp(\Phi'\times X)$ holds, which is \equref{Phi nat}. A similar 
situation happens in \equref{Phi nat new}.  
\end{rem}

The strict monoidal category $\QB(\K)(F)$ then has for objects pairs $(X, \tau_{F,X})$ where $X:\A\to\A$ is a 1-cell and $\tau_{F,X}: FX\to XF$ is a 2-cell in $\K$, 
so that \equref{monadic d.l.} -- \equref{YBE BBX} hold with $\A=\A'$ and $F=F'$. Morphisms of $\QB(\K)(F)$ are the same as those of $\Tau(\A, F)$. 

\medskip

\begin{defn} \delabel{module cat over F Phi}
The category of left Tambara modules over a quasi-bimonad $(\A, F, \Phi)$ we denote by ${}_{(\A,F,\Phi)}\Tau$ and we define it as follows. Its objects are 
triples $(X,\tau_{F,X}, \nu)$, where $(X,\tau_{F,X})$ are objects of $\QB(\K)(F)$ so that $(X, \nu)$ is a left $F$-module in $\K$ and the action 
$\nu: (FX, \tau_{F, FX})\to (X, \tau_{F, X})$ is a morphism in $\QB(\K)(F)$. 
Morphisms of ${}_{(\A,F,\Phi)}\Tau$ are left $F$-linear morphisms in $\QB(\K)(F)$. 
\end{defn}

Again, we will abuse notation and we will denote by $(X,\tau_{F,X})$ the objects of ${}_{(\A,F,\Phi)}\Tau$ without expliciting the $F$-action. 

\medskip

Loosely speaking, the objects of ${}_{(\A,F,\Phi)}\Tau$ are pairs $(X,\tau_{F,X})$ so that $X$ is a left $F$-module and the 2-cell $\tau_{F,X}$ 
is both monadic and comonadic with respect to $F$, it is natural with respect to the left $F$-module action on $X$ and satisfies \equref{Phi nat} and 
\equref{YBE BBX}. The 2-cell $\tau_{F,F}$ is left and right monadic and left and right comonadic distributive law. 
Morphisms of ${}_{(\A,F,\Phi)}\Tau$ are left $F$-linear 2-cells $\zeta: X\to Y$ in $\K$ so that \equref{2-cells in Mnd} holds with $B=B'=F$.

\bigskip

Let us now introduce some useful tool for the further computations. If $F$ is  a monad, then so is $FF$ with structure morphisms \equref{wreath (co)product}, 
and similarly $FFF$. 
We now set the following notation: 
\begin{equation} \eqlabel{FF module} 
\gbeg{3}{4}
\got{1}{FF} \got{3}{XY} \gnl
\gcn{2}{1}{1}{3} \gcl{1}  \gnl
\gvac{1} \glmf \gcn{1}{1}{-1}{-1} \gnl
\gvac{2} \gob{1}{XY} 
\gend=
\gbeg{3}{4}
\got{1}{F} \got{1}{F} \got{1}{X} \got{1}{Y} \gnl
\gcl{1} \glmptb \gnot{\hspace{-0,34cm}\tau_{F,X}} \grmptb \gcl{1} \gnl
\glm \glm \gnl
\gvac{1} \gob{1}{X} \gob{3}{Y}
\gend\hspace{2,4cm} 
\gbeg{3}{4}
\got{1}{FFF} \got{3}{XYZ} \gnl
\gcn{2}{1}{1}{3} \gcl{1}  \gnl
\gvac{1} \glmf \gcn{1}{1}{-1}{-1} \gnl
\gvac{2} \gob{1}{XYZ} 
\gend=
\gbeg{4}{5}
\got{1}{F} \got{1}{F} \got{1}{F} \got{1}{X} \got{1}{Y} \got{1}{Z} \gnl
\gcl{1} \gcl{1} \glmptb \gnot{\hspace{-0,34cm}\tau_{F,X}} \grmptb \gcl{1} \gcl{2} \gnl
\gcl{1} \glmptb \gnot{\hspace{-0,34cm}\tau_{F,X}} \grmptb \glmptb \gnot{\hspace{-0,34cm}\tau_{F,Y}} \grmptb \gnl
\glm \glm \glm \gnl
\gvac{1} \gob{1}{X} \gvac{1} \gob{1}{Y} \gvac{1} \gob{1}{Z} 
\gend
\end{equation}
Observe that by \equref{B act XY tau} one has: 
$$
\gbeg{4}{4}
\got{1}{} \got{1}{F} \got{3}{XYZ} \gnl
\gvac{1} \gcn{2}{1}{1}{3} \gcl{1}  \gnl
\gvac{2} \glm \gnl
\gvac{3} \gob{1}{XYZ} 
\gend=
\gbeg{4}{4}
\got{1}{F} \got{3}{XYZ} \gnl
\glmpt \gnot{\hspace{-0,34cm}\Delta^2} \grmpb \gcl{1}  \gnl
\gvac{1} \glmf \gcn{1}{1}{-1}{-1}  \gnl
\gvac{2} \gob{1}{XYZ} 
\gend
$$
where $\Delta^2=(\Delta\times id_F)\Delta=(\id_F\times\Delta)\Delta: F\to FFF$. 

\begin{prop} \prlabel{FFF rules}
Let $(\A, F, \Phi)$ be a quasi-bimonad in $\K$ and $(X,\tau_{F,X}),(Y,\tau_{F,Y}), (Z, \tau_{F,Z})\in {}_{(\A,F,\Phi)}\Tau$. 
The following relations hold true: 
\begin{center} 
\begin{tabular}{p{9cm}p{0.5cm}p{5cm}}
\begin{equation} \eqlabel{FF module rule}
\gbeg{5}{6}
\got{1}{FFF} \gvac{1} \got{1}{FFF} \got{3}{XYZ} \gnl
\gcn{2}{2}{1}{5} \gcn{2}{1}{1}{3} \gcl{1}  \gnl
\gvac{3} \glmf \gcn{1}{1}{-1}{-1} \gnl
\gvac{2} \gcn{2}{1}{1}{3} \gcl{1}  \gnl
\gvac{3} \glmf \gcn{1}{1}{-1}{-1} \gnl
\gvac{3} \gob{3}{XYZ} 
\gend=
\gbeg{3}{6}
\got{1}{FFF} \got{3}{FFF} \got{1}{XYZ} \gnl
\gwmu{3} \gvac{1} \gcl{2}  \gnl
\gvac{1} \gcn{2}{1}{1}{5} \gnl
\gvac{3} \glmf \gcn{1}{2}{-1}{-1} \gnl
\gvac{4} \gob{1}{XYZ} 
\gend
\end{equation} & & 
\begin{equation*} 
\gbeg{3}{5}
\got{1}{} \got{3}{XYZ} \gnl
\gu{1} \gvac{1} \gcl{1}  \gnl
\gcn{2}{1}{1}{3} \gcl{1}  \gnl
\gvac{1} \glmf \gcn{1}{1}{-1}{-1} \gnl
\gvac{2} \gob{1}{XYZ} 
\gend=
\gbeg{3}{5}
\got{1}{XYZ} \gnl
\gcl{3}  \gnl
\gob{1}{XYZ} 
\gend
\end{equation*} 
\end{tabular}
\end{center}

\begin{equation} \eqlabel{nat FFF}
\gbeg{4}{5}
\got{1}{F} \got{2}{FFF} \got{2}{XYZ} \gnl
\glmptb \gnot{\hspace{-0,34cm}\tau_{F,FFF}} \grmptb \gcn{1}{1}{3}{1} \gnl
\gcl{1} \glmptb \gnot{\hspace{-0,34cm}\tau_{F,XYZ}} \grmptb \gnl
\glmf \gcn{1}{1}{-1}{-1} \gcn{1}{1}{-1}{-1} \gnl
\gob{2}{XYZ} \gob{1}{F} 
\gend=
\gbeg{3}{5}
\got{1}{F} \got{2}{FFF} \got{2}{XYZ} \gnl
\gcn{2}{2}{1}{5} \glmf \gcn{1}{2}{-1}{-1} \gnl
 \gnl
\gvac{2} \glmptb \gnot{\hspace{-0,34cm}\tau_{F,XYZ}} \grmptb  \gnl
\gvac{1} \gob{2}{XYZ} \gob{1}{F} 
\gend
\end{equation}
\end{prop}

\begin{proof}
For the first identity first apply that the $F$-action is a morphism in $\QB(\K)(F)$, then apply the Yang-Baxter identity \equref{YBE BBX} (first once, then twice) 
and finally the monadic distributive law. The second identity is obvious, and for the third one apply first the Yang-Baxter identity \equref{YBE BBX} (first once, then twice) 
and then the fact that the $F$-action is a morphism in $\QB(\K)(F)$. 
\qed\end{proof}

We now may prove:

\begin{thm} \thlabel{quasi-bim monoidal}
For a quasi-bimonad $(\A,F,\Phi)$ in $\K$, 
the category ${}_{(\A,F,\Phi)}\Tau$ is non-strict monoidal with the unit 
object $(\Id_{\A}, id_F)$ with the trivial $F$-action, the tensor product as in \prref{B act XY quasi-bim}, the unit constraints given by identities and the associativity constraint 
$\alpha: (XY)Z\to X(YZ)$ given by: 
\begin{equation} \eqlabel{assoc. lm}
\alpha_{X,Y,Z}=
\gbeg{6}{6}
\gvac{3} \got{1}{X} \got{1}{Y} \got{1}{Z} \gnl
\gcn{1}{1}{2}{1} \gelt{\s\Phi} \gcn{1}{1}{0}{1} \gcl{1} \gcl{2} \gcl{3} \gnl 
\gcl{1} \gcl{1} \glmptb \gnot{\hspace{-0,34cm}\tau_{F,X}} \grmptb \gnl
\gcl{1} \glmptb \gnot{\hspace{-0,34cm}\tau_{F,X}} \grmptb \glmptb \gnot{\hspace{-0,34cm}\tau_{F,Y}} \grmptb \gcl{1} \gnl
\glm \glm \glm \gnl
\gvac{1} \gob{1}{X} \gob{1}{} \gob{1}{Y} \gob{1}{} \gob{1}{Z} 
\gend=
\gbeg{3}{4}
\got{3}{XYZ} \gnl
\gelt{\s\Phi} \gcl{1}  \gnl
\glmf \gcn{1}{1}{-1}{-1} \gnl
\gvac{1} \gob{1}{XYZ} 
\gend
\end{equation}
for $(X,\tau_{F,X}),(Y,\tau_{F,Y}), (Z, \tau_{F,Z})\in {}_{(\A,F,\Phi)}\Tau$. 
\end{thm}

\begin{proof}
Let us first show that $\alpha_{X,Y,Z}$ is an isomorphism in ${}_{(\A,F,\Phi)}\Tau$. It is clearly invertible. 
We see that the identity \equref{2-cells in Mnd} is fulfilled: 
\begin{equation} \eqlabel{alfa is Tambara} 
\gbeg{3}{5}
\gvac{1} \got{1}{F} \got{2}{XYZ} \gnl
\gvac{1} \glmptb \gnot{\hspace{-0,34cm}\tau_{F,XYZ}} \grmptb \gnl
\gelt{\s\Phi} \gcl{1} \gcl{2} \gnl
\glmf \gcn{1}{1}{-1}{-1} \gnl
\gob{2}{XYZ} \gob{1}{F} 
\gend\stackrel{\equref{Phi nat new} }{=}
\gbeg{3}{6}
\got{1}{F} \gvac{1} \got{2}{XYZ} \gnl
\gcl{1} \gelt{\s\Phi} \gcl{2} \gnl
\glmptb \gnot{\hspace{-0,34cm}\tau_{F,FFF}} \grmptb \gnl
\gcl{1} \glmptb \gnot{\hspace{-0,34cm}\tau_{F,XYZ}} \grmptb \gnl
\glmf \gcn{1}{1}{-1}{-1} \gcn{1}{1}{-1}{-1} \gnl
\gob{2}{XYZ} \gob{1}{F} 
\gend\stackrel{\equref{nat FFF} }{=}
\gbeg{3}{6}
\got{2}{F} \gvac{1} \got{1}{XYZ} \gnl
\gcn{1}{1}{2}{2} \gvac{1} \gelt{\s\Phi} \gcl{2} \gnl
\gcn{2}{2}{2}{5} \glmf \gcn{1}{2}{-1}{-1} \gnl
 \gnl
\gvac{2} \glmptb \gnot{\hspace{-0,34cm}\tau_{F,XYZ}} \grmptb  \gnl
\gvac{2} \gob{1}{XYZ} \gob{1}{\hspace{0,24cm}F} 
\gend
\end{equation} 
Observe that left $F$-module structure on $(XY)Z$ and on $X(YZ)$ are both given via 
$\psi_{F,XY}$ (see \equref{B act XY psi}). Now, left $F$-linearity of $\alpha_{X,Y,Z}:(XY)Z\to X(YZ)$ can be stated like this: 
$$
\gbeg{4}{6}
\got{1}{} \got{1}{F} \got{3}{XYZ} \gnl
\gvac{1} \gcn{2}{1}{1}{3} \gcl{1}  \gnl
\gvac{1} \gelt{\s\Phi}  \glm \gnl
\gvac{1} \gcn{2}{1}{1}{3} \gcl{1}  \gnl
\gvac{2} \glmf \gcn{1}{1}{-1}{-1} \gnl
\gvac{3} \gob{1}{XYZ} 
\gend=
\gbeg{3}{6}
\got{2}{F} \gvac{1} \got{1}{XYZ} \gnl
\gcn{1}{1}{2}{2} \gvac{1} \gelt{\s\Phi} \gcl{2} \gnl
\gcn{2}{2}{2}{5} \glmf \gcn{1}{2}{-1}{-1} \gnl
 \gnl
\gvac{2} \glm \gnl
\gvac{3} \gob{1}{XYZ} 
\gend\qquad\Leftrightarrow\qquad
\gbeg{4}{6}
\got{1}{} \got{1}{F} \got{3}{XYZ} \gnl
\gvac{1} \glmpt \gnot{\hspace{-0,34cm}\Delta^2} \grmpb \gcl{1}  \gnl
\gelt{\s\Phi} \gvac{1} \glmf \gcn{1}{1}{-1}{-1}  \gnl
\gvac{1} \gcn{2}{1}{-1}{3} \gcl{1}  \gnl
\gvac{2} \glmf \gcn{1}{1}{-1}{-1} \gnl
\gvac{3} \gob{1}{XYZ} 
\gend=
\gbeg{3}{6}
\got{2}{F} \gvac{1} \got{1}{XYZ} \gnl
\gcn{1}{1}{2}{2} \gvac{1} \gelt{\s\Phi} \gcl{2} \gnl
\gcn{2}{1}{2}{3} \glmf \gcn{1}{2}{-1}{-1} \gnl
\gvac{1} \glmpt \gnot{\hspace{-0,34cm}\Delta^2} \grmpb \gnl
\gvac{2} \glmf \gcn{1}{1}{-1}{-1} \gnl
\gvac{3} \gob{1}{XYZ} 
\gend\qquad\stackrel{\equref{FF module rule}}{\Leftrightarrow}\qquad
\gbeg{4}{6}
\got{1}{} \got{1}{F} \got{3}{XYZ} \gnl
\gelt{\s\Phi} \glmpt \gnot{\hspace{-0,34cm}\Delta^2} \grmpb \gcl{3} \gnl
\gwmu{3}  \gnl
\gcn{1}{1}{3}{5} \gnl
\gvac{2} \glmf \gcn{1}{1}{-1}{-1} \gnl
\gvac{3} \gob{1}{XYZ} 
\gend=
\gbeg{3}{6}
\got{1}{F} \gvac{2} \got{1}{XYZ} \gnl
\glmptb \gnot{\hspace{-0,34cm}\Delta^2} \grmp \gelt{\s\Phi} \gcl{3} \gnl
\gwmu{3}  \gnl
\gcn{1}{1}{3}{5} \gnl
\gvac{2} \glmf \gcn{1}{1}{-1}{-1} \gnl
\gvac{3} \gob{1}{XYZ} 
\gend 
$$
which is equivalent to \equref{quasi coass.} (choose $X=Y=Z=F$ and compose with $\eta_F$). In the pentagon axiom for $\alpha$: 
\begin{equation}  \eqlabel{pentagon alfa}
(\Id_X\times\alpha_{Y,Z,W})\comp\alpha_{X,YZ,W}\comp(\alpha_{X,Y,Z}\times\Id_W)=\alpha_{X,Y,ZW}\comp\alpha_{XY,Z,W}
\end{equation} 
by the comonadic distributive law we have: 
$$\alpha_{XY,Z,W}=
\gbeg{4}{7}
\got{1}{} \got{1}{} \got{5}{XYZ} \gnl
\gvac{1} \gelt{\s\Phi} \gvac{2} \gcl{3}  \gnl
\gvac{1} \gcl{1} \gnl
\glmp \gcmpt \gnot{\hspace{-0,34cm}\Delta\times\Id\times\Id} \gcmpb \grmp \gnl
\gvac{2} \gcn{2}{1}{1}{3} \gcl{1}  \gnl
\gvac{3} \glmf \gcn{1}{1}{-1}{-1} \gnl
\gvac{4} \gob{1}{XYZ} 
\gend\qquad
\alpha_{X,YZ,W}=
\gbeg{4}{7}
\got{1}{} \got{1}{} \got{5}{XYZ} \gnl
\gvac{1} \gelt{\s\Phi} \gvac{2} \gcl{3}  \gnl
\gvac{1} \gcl{1} \gnl
\glmp \gcmpt \gnot{\hspace{-0,34cm}\Id\times\Delta\times\Id} \gcmpb \grmp \gnl
\gvac{2} \gcn{2}{1}{1}{3} \gcl{1}  \gnl
\gvac{3} \glmf \gcn{1}{1}{-1}{-1} \gnl
\gvac{4} \gob{1}{XYZ} 
\gend\qquad
\alpha_{X,Y,ZW}=
\gbeg{4}{7}
\got{1}{} \got{1}{} \got{5}{XYZ} \gnl
\gvac{1} \gelt{\s\Phi} \gvac{2} \gcl{3}  \gnl
\gvac{1} \gcl{1} \gnl
\glmp \gcmpt \gnot{\hspace{-0,34cm}\Id\times\Id\times\Delta} \gcmpb \grmp \gnl
\gvac{2} \gcn{2}{1}{1}{3} \gcl{1}  \gnl
\gvac{3} \glmf \gcn{1}{1}{-1}{-1} \gnl
\gvac{4} \gob{1}{XYZ} 
\gend
$$
then \equref{pentagon alfa} becomes: 
$$
\gbeg{8}{9}
\got{2}{} \gvac{4} \got{1}{XYZ} \got{2}{W} \gnl
\gvac{3} \gelt{\s\Phi} \gvac{1} \gelt{\s\Phi} \gcl{2} \gcl{4} \gnl
\gvac{3} \gcl{1} \gvac{1} \glmf \gcn{1}{2}{-1}{-1} \gnl
\gelt{\s\Phi} \gvac{1} \glmp \gcmptb \gnot{\hspace{-0,34cm}\Id\times\Delta\times\Id} \gcmpb \grmp \gnl
\gcl{1} \gvac{2} \gcl{1} \glmptb \gnot{\tau_{F,XYZ}} \gcmp \grmptb \gnl
\glmpt \gnot{\eta_F\times\Id_{FFF}} \gcmpb \grmpb \glmf \gcn{1}{1}{-1}{-1} \glm \gnl
\gvac{1} \gcl{1} \glmptb \gnot{\tau_{F,XYZ}} \gcmp \grmpb \gcn{1}{1}{5}{1} \gnl
\gvac{1} \glmf \gcn{1}{1}{-1}{-1} \glm \gnl
\gvac{2} \gob{1}{XYZ}  \gob{5}{W}
\gend=
\gbeg{8}{9}
\got{2}{} \gvac{5} \got{1}{XYZ} \got{2}{W} \gnl
\gvac{4} \gelt{\s\Phi} \gvac{2} \gcl{3} \gcl{4} \gnl
\gvac{4} \gcl{1} \gvac{1} \gnl
\gvac{1} \gelt{\s\Phi} \gvac{1} \glmp \gcmptb \gnot{\hspace{-0,34cm}\Delta\times\Id\times\Id} \gcmpb \grmp \gnl
\gvac{1} \gcl{1} \gvac{2} \gcl{1} \glmptb \gnot{\tau_{F,XYZ}} \gcmp \grmpb \gnl
 \glmp \gcmpt \gnot{\hspace{-0,34cm}\Id\times\Id\times\Delta} \gcmpb \grmp\glmf \gcn{1}{1}{-1}{-1} \glm \gnl
\gvac{2} \gcl{1} \glmptb \gnot{\tau_{F,XYZ}} \gcmp \grmpb \gcn{1}{1}{5}{1} \gnl
\gvac{2} \glmf \gcn{1}{1}{-1}{-1} \glm \gnl
\gvac{3} \gob{1}{XYZ}  \gob{5}{W}
\gend\qquad\Leftrightarrow\qquad
\gbeg{8}{9}
\got{2}{} \gvac{4} \got{1}{XYZ}  \gnl
\gvac{3} \gelt{\s\Phi} \gvac{1} \gelt{\s\Phi} \gcl{2} \gnl
\gvac{3} \gcl{1} \gvac{1} \glmf \gcn{1}{2}{-1}{-1} \gnl
\gelt{\s\Phi} \gvac{1} \glmp \gcmptb \gnot{\hspace{-0,34cm}\Id\times\Delta\times\Id} \gcmpb \grmp \gnl
\gcl{1} \gvac{2} \gcl{1} \glmptb \gnot{\tau_{F,XYZ}} \gcmp \grmpb \gnl
\glmpt \gnot{\eta_F\times\Id_{FFF}} \gcmpb \grmpb \glmf \gcn{1}{1}{-1}{-1} \gcl{2} \gnl
\gvac{1} \gcl{1} \glmptb \gnot{\tau_{F,XYZ}} \gcmp \grmpb \gnl
\gvac{1} \glmf \gcn{1}{1}{-1}{-1} \gwmu{3} \gnl
\gvac{2} \gob{1}{XYZ}  \gob{5}{F}
\gend=
\gbeg{8}{9}
\got{2}{} \gvac{5} \got{1}{XYZ}  \gnl
\gvac{4} \gelt{\s\Phi} \gvac{2} \gcl{3} \gnl
\gvac{4} \gcl{1} \gvac{1} \gnl
\gvac{1} \gelt{\s\Phi} \gvac{1} \glmp \gcmptb \gnot{\hspace{-0,34cm}\Delta\times\Id\times\Id} \gcmpb \grmp \gnl
\gvac{1} \gcl{1} \gvac{2} \gcl{1} \glmptb \gnot{\tau_{F,XYZ}} \gcmp \grmpb \gnl
 \glmp \gcmpt \gnot{\hspace{-0,34cm}\Id\times\Id\times\Delta} \gcmpb \grmp\glmf \gcn{1}{1}{-1}{-1} \gcl{2} \gnl
\gvac{2} \gcl{1} \glmptb \gnot{\tau_{F,XYZ}} \gcmp \grmpb  \gnl
\gvac{2} \glmf \gcn{1}{1}{-1}{-1} \gwmu{3} \gnl
\gvac{3} \gob{1}{XYZ}  \gob{5}{F}
\gend
$$
the latter by \equref{nat FFF} is equivalent to: 
$$
\gbeg{11}{11}
\got{2}{} \gvac{7} \got{1}{XYZ}  \gnl
\gvac{1} \gelt{\s\Phi} \gvac{2} \gelt{\s\Phi} \gvac{2} \gelt{\s\Phi} \gvac{1} \gcl{4} \gnl
\gvac{1} \gcl{1} \gvac{2} \gcl{1} \gvac{2} \gcl{2} \gnl
\glmp \gnot{\eta_F\times\Id_{FFF}} \gcmpb \grmpb \glmp \gcmptb \gnot{\hspace{-0,34cm}\Id\times\Delta\times\Id} \gcmpb \grmp \gnl
\gvac{1} \gcl{3} \glmptb \gnot{\tau_{F,FFF}} \gcmp \grmptb \glmpt \gnot{\tau_{F,FFF}} \gcmp \grmptb \gnl
\gvac{2}  \gcl{1} \gvac{1} \glmptb \gnot{\tau_{F,FFF}} \gcmp \grmpt \glmptb \gnot{\tau_{F,XYZ}} \gcmp \grmptb  \gnl
\gvac{2} \gwmu{3} \glmptb \gnot{\tau_{F,XYZ}} \gcmp \grmptb \gvac{1} \gcl{1} \gnl
\gvac{1} \gwmu{3} \gvac{1} \gcl{2} \gvac{1} \gwmu{3} \gnl
\gvac{2} \gcn{1}{1}{1}{5} \gvac{5} \gcl{2} \gnl
\gvac{4} \glmf \gcn{1}{1}{-1}{-1} \gnl
\gvac{5} \gob{1}{XYZ}  \gob{5}{F}
\gend=
\gbeg{11}{9}
\got{2}{} \gvac{6} \got{1}{XYZ}  \gnl
\gvac{1} \gelt{\s\Phi} \gvac{3} \gelt{\s\Phi} \gvac{2} \gcl{3} \gnl
\gvac{1} \gcl{1} \gvac{3} \gcl{1} \gnl
\glmp \gcmptb \gnot{\hspace{-0,34cm}\Id\times\Id\times\Delta} \gcmpb \grmp \glmp \gcmpt \gnot{\hspace{-0,34cm}\Delta\times\Id\times\Id} \gcmpb \grmp \gnl
\gvac{1} \gcl{1} \glmptb \gnot{\tau_{F,FFF}} \gcmp \grmptb \gvac{1} \glmpt \gnot{\tau_{F,XYZ}} \gcmp \grmptb \gnl
\gvac{1} \gmu \gvac{1} \glmptb \gnot{\tau_{F,XYZ}} \gcmp \grmptb \gvac{1} \gcl{1}  \gnl
\gvac{1} \gcn{1}{1}{2}{5} \gvac{2} \gcl{1} \gvac{1} \gwmu{3} \gnl
\gvac{3} \glmf \gcn{1}{1}{-1}{-1} \gvac{1} \gcl{1} \gnl
\gvac{4} \gob{1}{XYZ}  \gob{5}{F}
\gend
$$
which by the monadic distributive law is further equivalent to: 
$$
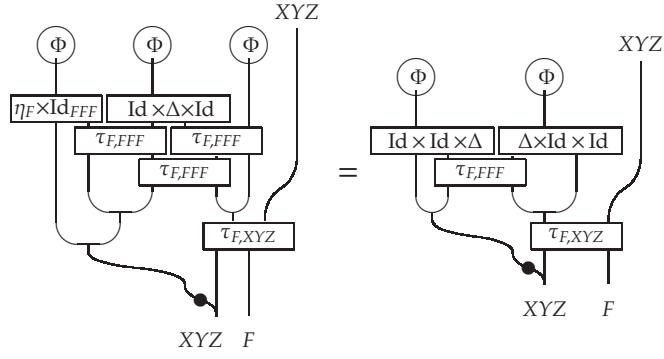
\gbeg{10}{11}
\got{2}{} \gvac{6} \got{2}{XYZ}  \gnl
\gvac{1} \gelt{\s\Phi} \gvac{2} \gelt{\s\Phi} \gvac{2} \gelt{\s\Phi} \gcn{1}{4}{2}{2} \gnl
\gvac{1} \gcl{1} \gvac{2} \gcl{1} \gvac{2} \gcl{2} \gnl
\glmp \gnot{\eta_F\times\Id_{FFF}} \gcmpb \grmpb \glmp \gcmptb \gnot{\hspace{-0,34cm}\Id\times\Delta\times\Id} \gcmpb \grmp \gnl
\gvac{1} \gcl{3} \glmptb \gnot{\tau_{F,FFF}} \gcmp \grmptb \glmpt \gnot{\tau_{F,FFF}} \gcmp \grmptb \gnl
\gvac{2}  \gcl{1} \gvac{1} \glmptb \gnot{\tau_{F,FFF}} \gcmp \grmptb \gcl{1} \gcn{1}{2}{2}{0} \gnl
\gvac{2} \gwmu{3} \gvac{1} \gmu \gnl
\gvac{1} \gwmu{3} \gvac{2} \glmpb \gnot{\tau_{F,XYZ}} \gcmpb \grmp \gnl
\gvac{2} \gcn{3}{1}{1}{7} \gvac{1} \gcl{1} \gcl{2} \gnl
\gvac{5} \glmf \gcn{1}{1}{-1}{-1} \gnl
\gvac{5} \gob{2}{XYZ}  \gob{1}{F}
\gend=
\gbeg{9}{9}
\got{2}{} \gvac{6} \got{1}{XYZ}  \gnl
\gvac{1} \gelt{\s\Phi} \gvac{3} \gelt{\s\Phi} \gvac{2} \gcl{3} \gnl
\gvac{1} \gcl{1} \gvac{3} \gcl{1} \gnl
\glmp \gcmptb \gnot{\hspace{-0,34cm}\Id\times\Id\times\Delta} \gcmpb \grmp \glmp \gcmpt \gnot{\hspace{-0,34cm}\Delta\times\Id\times\Id} \gcmpb \grmp \gnl
\gvac{1} \gcl{1} \glmptb \gnot{\tau_{F,FFF}} \gcmp \grmptb \gvac{1} \gcl{1} \gcn{1}{2}{3}{1} \gnl
\gvac{1} \gmu \gvac{1} \gwmu{3}  \gnl
\gvac{1} \gcn{1}{1}{2}{7} \gvac{3} \glmpt \gnot{\tau_{F,XYZ}} \gcmp \grmptb\gnl
\gvac{4} \glmf \gcn{1}{1}{-1}{-1} \gcl{1} \gnl
\gvac{5} \gob{1}{XYZ}  \gob{3}{F}
\gend$$
and this (choose $X=Y=Z=F$ and compose with $\eta_F$) is equivalent to: 
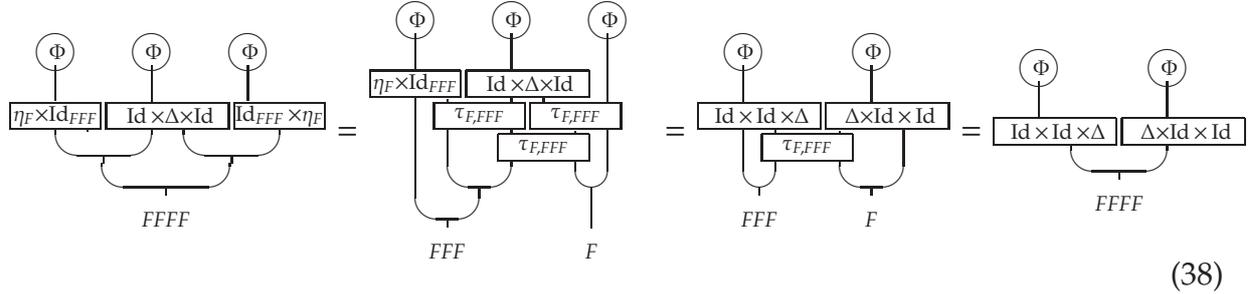
\begin{equation} \eqlabel{name 3-coc}
\gbeg{10}{6}
\gvac{1} \gelt{\s\Phi} \gvac{2} \gelt{\s\Phi} \gvac{2} \gelt{\s\Phi} \gnl
\gvac{1} \gcl{1} \gvac{2} \gcl{1} \gvac{2} \gcl{1} \gnl
\glmp \gnot{\eta_F\times\Id_{FFF}} \gcmpb \grmpb \glmp \gcmptb \gnot{\hspace{-0,34cm}\Id\times\Delta\times\Id} \gcmpb \grmp \glmp \gnot{\Id_{FFF}\times\eta_F} \gcmpb \grmp \gnl
\gvac{1} \gwmu{4} \gwmu{4} \gnl
\gvac{3} \hspace{-0,34cm} \gwmu{5} \gnl
\gvac{5} \gob{1}{FFFF}
\gend=
\gbeg{9}{8}
\gvac{1} \gelt{\s\Phi} \gvac{2} \gelt{\s\Phi} \gvac{2} \gelt{\s\Phi} \gnl
\gvac{1} \gcl{1} \gvac{2} \gcl{1} \gvac{2} \gcl{2} \gnl
\glmp \gnot{\eta_F\times\Id_{FFF}} \gcmpb \grmp \glmp \gcmptb \gnot{\hspace{-0,34cm}\Id\times\Delta\times\Id} \gcmpb \grmp \gnl
\gvac{1} \gcl{3} \glmptb \gnot{\tau_{F,FFF}} \gcmp \grmptb \glmpt \gnot{\tau_{F,FFF}} \gcmp \grmptb \gnl
\gvac{2}  \gcl{1} \gvac{1} \glmptb \gnot{\tau_{F,FFF}} \gcmp \grmptb \gcl{1} \gnl
\gvac{2} \gwmu{3} \gvac{1} \gmu \gnl
\gvac{1} \gwmu{3} \gvac{2} \gcn{1}{1}{2}{2} \gnl
\gvac{2} \gob{1}{FFF}  \gob{8}{F}
\gend=
\gbeg{8}{6}
\gvac{1} \gelt{\s\Phi} \gvac{3} \gelt{\s\Phi} \gnl
\gvac{1} \gcl{1} \gvac{3} \gcl{1} \gnl
\glmp \gcmptb \gnot{\hspace{-0,34cm}\Id\times\Id\times\Delta} \gcmpb \grmp \glmp \gcmpt \gnot{\hspace{-0,34cm}\Delta\times\Id\times\Id} \gcmpb \grmp \gnl
\gvac{1} \gcl{1} \glmptb \gnot{\tau_{F,FFF}} \gcmp \grmptb \gvac{1} \gcl{1} \gnl
\gvac{1} \gmu \gvac{1} \gwmu{3}  \gnl
\gvac{1} \gob{2}{FFF}  \gob{5}{F}
\gend=
\gbeg{8}{5}
\gvac{1} \gelt{\s\Phi} \gvac{3} \gelt{\s\Phi} \gnl
\gvac{1} \gcl{1} \gvac{3} \gcl{1} \gnl
\glmp \gcmpt \gnot{\hspace{-0,34cm}\Id\times\Id\times\Delta} \gcmpb \grmp \glmp \gcmptb \gnot{\hspace{-0,34cm}\Delta\times\Id\times\Id} \gcmp \grmp \gnl
\gvac{2} \gwmu{4}  \gnl
\gvac{3} \gob{2}{FFFF}
\gend
\end{equation}
Observe that this is precisely \equref{3-coc. cond.} (recall the comonad structure of $FFF$ and apply \equref{Phi nat}).

The unity constraints taken as identities are left $F$-linear by the $\Delta\x\Epsilon$ compatibility of $F$. Their coherence with $\alpha$ means: 
$$
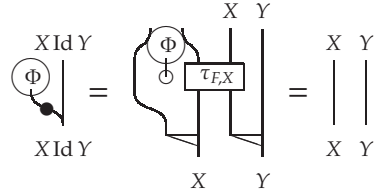
\gbeg{2}{4}
\got{3}{X\Id Y} \gnl
\gelt{\s\Phi} \gcl{1}  \gnl
\glmf \gcn{1}{1}{-1}{-1} \gnl
\gvac{1} \gob{1}{X\Id Y} 
\gend=
\gbeg{5}{6}
\gvac{3} \got{1}{X} \got{1}{Y} \gnl
\gcn{1}{1}{2}{1} \gelt{\s\Phi} \gcn{1}{1}{0}{1} \gcl{1} \gcl{3} \gnl 
\gcl{1} \gcu{1} \glmptb \gnot{\hspace{-0,34cm}\tau_{F,X}} \grmptb \gnl
\gcn{2}{1}{1}{3} \gcl{1} \gcl{1} \gnl 
\gvac{1} \glm \glm \gnl
\gvac{2} \gob{1}{X} \gob{1}{} \gob{1}{Y} 
\gend=
\gbeg{2}{4}
\got{1}{X} \got{1}{Y} \gnl
\gcl{2} \gcl{2} \gnl
\gob{1}{X} \gob{1}{Y} 
\gend
$$
which is fulfilled because $\Phi$ is normalized. 
\qed\end{proof}

The following is direct to prove and it justifies the name ``3-cocycle'' for $\Phi$:

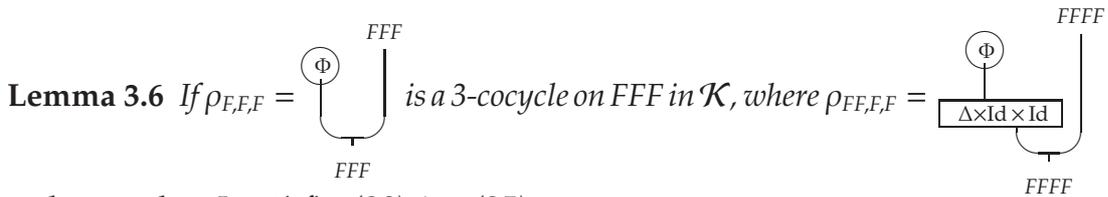
\begin{lma}
If $\rho_{F,F,F}=
\gbeg{3}{5}
\got{5}{FFF} \gnl
\gelt{\s\Phi}\gvac{1}   \gcl{2}  \gnl 
\gcl{1} \gnl
\gwmu{3} \gnl
\gob{3}{FFF} 
\gend$ is a 3-cocycle on $FFF$ in $\K$, where 
$\rho_{FF,F,F}=
\gbeg{9}{6}
\got{9}{FFFF} \gnl
\gvac{1} \gelt{\s\Phi} \gvac{2} \gcl{3} \gnl 
\gvac{1} \gcl{1} \gnl
\glmp \gcmpt \gnot{\hspace{-0,34cm}\Delta\times\Id\times\Id} \gcmpb \grmp \gnl
\gvac{2} \gwmu{3} \gnl
\gob{7}{FFFF} 
\gend$ and so on, then $\Phi$ satisfies \equref{name 3-coc}, {\em i.e.} \equref{3-coc. cond.}. 
\end{lma}

If 1-cells of the 2-category $\K$ in question posses elements, then the converse of the above Theorem is also true. In this case, if $\alpha$ is an associativity constraint, 
one defines $\Phi:= \alpha_{F,F,F}(\eta_F\times\eta_F\times\eta_F)$. The named condition is fulfilled for example for the 2-category $\Tensor$ of tensor categories. Its objects 
are tensor categories, and given two such objects $\C$ and $\D$, the category $\Tensor(\C,\D)$ defining 1- and 2-cells is given by the category of $C\x\D$-bimodule categories $\C\x\D\x\Bimod$. Then 1-cells are $C\x\D$-bimodule categories $\M$ and given two 
$C\x\D$-bimodule categories $\M$ and $\N$, a 2-cell is given by a $C\x\D$-bilinear functor $\F:\M\to\N$. 
The composition of 1-cells is given by the relative Deligne product (over the corresponding tensor category). 
A quasi-bimonad in $\Tensor$ is then a $\C$-bimodule monoidal category $\M$ and a ``quasi'' coring category with a $\C$-bilinear functor $\Phi: \C\to\M\Del_{\C}\M\Del_{\C}\M$ so that 
the suitable compatibility conditions are satisfied. We defined corings (and quasi-corings, but there the prefix {\em quasi} had a different meaning) in \cite{Femic3}.

\bigskip

Dually to quasi-bimonads we define coquasi-bimonads in $\K$.

\begin{defn}
A coquasi-bimonad in $\K$ is an octuple $(\A, F, \mu, \eta, \Delta, \Epsilon, \tau_{F,F}, \omega)$, where $(\A, F, \Delta, \Epsilon)$ is a comonad, 
$\tau_{F,F}:FF\to FF$ is a left and right monadic and comonadic distributive law, the 2-cells 
$\mu=\gbeg{2}{3}
\got{1}{F} \got{1}{F} \gnl
\gmu \gnl 
\gob{2}{F} \gnl
\gend$ and 
$\eta=\gbeg{2}{2}
\gu{1} \gnl 
\gob{1}{F} \gnl
\gend$ 
satisfy:
$$
\gbeg{2}{4}
\got{1}{} \got{1}{F} \gnl
\gu{1} \gcl{1} \gnl
\gmu \gnl
\gob{2}{F}
\gend=
\gbeg{1}{4}
\got{1}{F} \gnl
\gcl{2} \gnl
\gob{1}{F}
\gend=
\gbeg{2}{4}
\got{1}{F} \got{1}{} \gnl
\gcl{1} \gu{1} \gnl
\gmu \gnl
\gob{2}{F} 
\gend, \quad
\gbeg{3}{5}
\got{1}{F} \got{3}{F} \gnl
\gwmu{3} \gnl
\gvac{1} \gcl{1} \gnl
\gwcm{3} \gnl
\gob{1}{F}\gvac{1}\gob{1}{F}
\gend=
\gbeg{4}{5}
\got{2}{F} \got{2}{F} \gnl
\gcmu \gcmu \gnl
\gcl{1} \glmptb \gnot{\hspace{-0,34cm}\tau_{F,F}} \grmptb \gcl{1} \gnl
\gmu \gmu \gnl
\gob{2}{F} \gob{2}{F} \gnl
\gend,\quad 
\gbeg{2}{3}
\gu{1}  \gu{1} \gnl
\gcl{1} \gcl{1} \gnl
\gob{1}{F} \gob{1}{F}
\gend=
\gbeg{2}{3}
\gu{1} \gnl
\hspace{-0,34cm} \gcmu \gnl
\gob{1}{F} \gob{1}{F}
\gend,\quad 
\gbeg{2}{3}
\got{1}{F} \got{1}{F} \gnl
\gcl{1} \gcl{1} \gnl
\gcu{1}  \gcu{1} \gnl
\gend=
\gbeg{2}{3}
\got{1}{F} \got{1}{F} \gnl
\gmu \gnl
\gvac{1} \hspace{-0,2cm} \gcu{1} \gnl
\gend,\quad 
\gbeg{1}{2}
\gu{1} \gnl
\gcu{1} \gnl
\gob{1}{}
\gend=
\Id_{id_{\A}}
$$
and the 2-cell $\omega: FFF\to\Id_\A$ is convolution invertible and normalized: 
$\gbeg{3}{3}
\got{1}{F} \got{3}{F} \gnl
\gcl{1} \gu{1} \gcl{1} \gnl
\gcn{1}{1}{1}{2} \gelt{\s\omega} \gcn{1}{1}{1}{0} \gnl
\gend=
\gbeg{2}{3}
\got{1}{F} \got{1}{F} \gnl
\gcl{1} \gcl{1} \gnl
\gcu{1} \gcu{1} \gnl
\gend
$ and it obeys:  
$$\gbeg{4}{4}
\got{3}{FFF} \gnl
\gwcm{3}  \gnl
\gelt{\s\omega} \glmp \gnot{\hspace{-0,34cm}\mu^2} \grmptb \gnl
\gob{5}{F} \gnl
\gend=
\gbeg{3}{4}
\got{3}{FFF} \gnl
\gwcm{3}  \gnl
\glmptb \gnot{\hspace{-0,34cm}\mu^2} \grmp \gelt{\s\omega} \gnl
\gob{1}{F} \gnl
\gend \qquad\text{and}\qquad
\gbeg{10}{6}
\gvac{4} \got{1}{FFFF} \gnl
\gvac{2} \gwcm{5} \gnl
\gvac{1} \hspace{-0,34cm} \gwcm{4} \gwcm{4} \gnl
\glmp \gnot{\Epsilon_F\times\Id_{FFF}} \gcmptb \grmp \glmp \gcmptb \gnot{\hspace{-0,34cm}\Id\times\mu\times\Id} \gcmpt \grmp \glmp \gnot{\s\Id_{FFF}\times\Epsilon_F} \gcmptb \grmp \gnl
\gvac{1} \gcl{1} \gvac{2} \gcl{1} \gvac{3} \gcl{1} \gnl
\gvac{1} \gelt{\s\omega} \gvac{2} \gelt{\s\omega} \gvac{3} \gelt{\s\omega} \gnl
\gend=
\gbeg{8}{5}
\gvac{3} \got{1}{FFFF} \gnl
\gvac{1} \gwcm{5}  \gnl
\glmp \gcmptb \gnot{\hspace{-0,34cm}\Id\times\Id\times\mu} \gcmp \grmp \glmp \gcmptb \gnot{\hspace{-0,34cm}\mu\times\Id\times\Id} \gcmp \grmp \gnl
\gvac{1} \gcl{1} \gvac{3} \gcl{1} \gnl
\gvac{1} \gelt{\s\omega} \gvac{3} \gelt{\s\omega} \gnl
\gend
$$
\end{defn}

The 2-category of coquasi-bimonads in $\K$ we denote by $\CQB(\K)$, its 0-cells are coquasi-bimonads $(\A, F, \omega)$, 1- and 2-cells have the same form as in $\QB(\K)$ 
and they satisfy the same axioms \equref{monadic d.l.} -- \equref{comonadic d.l.} and \equref{YBE BBX}--\equref{2-cells QB}, while the axioms \equref{Phi nat new} and \equref{Phi nat} 
change to the first two in: 
\begin{equation} \eqlabel{dual cond}
\gbeg{4}{5}
\got{1}{F'} \got{1}{F'} \got{1}{F'} \got{1}{X} \gnl
\gcl{2} \gcl{2} \gcl{2} \gcl{3} \gnl
\gnl
\gcn{1}{1}{1}{2} \gelt{\s\omega'} \gcn{1}{1}{1}{0} \gnl
\gob{7}{X} \gnl
\gend=
\gbeg{2}{6}
\got{1}{F'} \got{1}{F'} \got{1}{F'} \got{1}{X} \gnl
\gcl{1} \gcl{1} \glmptb \gnot{\hspace{-0,34cm}\tau_{F,X}} \grmptb \gnl
\gcl{1} \glmptb \gnot{\hspace{-0,34cm}\tau_{F,X}} \grmptb \gcl{1} \gnl
\glmptb \gnot{\hspace{-0,34cm}\tau_{F,X}} \grmptb \gcl{1} \gcl{1} \gnl
\gcl{1} \gcn{1}{1}{1}{2} \gelt{\s\omega} \gcn{1}{1}{1}{0} \gnl 
\gob{1}{X} \gnl
\gend \hspace{2,5cm} 
\gbeg{4}{5}
\got{1}{F} \got{1}{F} \got{1}{F} \got{1}{F} \gnl
\gcl{2} \gcl{2} \gcl{2} \gcl{2} \gnl
\gnl
\gcl{1} \gcn{1}{1}{1}{2} \gelt{\s\omega} \gcn{1}{1}{1}{0} \gnl  %
\gob{1}{F} \gnl
\gend=
\gbeg{2}{6}
\got{1}{F} \got{1}{F} \got{1}{F} \got{1}{F} \gnl
\glmptb \gnot{\hspace{-0,34cm}\tau_{F,F}} \grmptb \gcl{1} \gcl{1} \gnl
\gcl{1} \glmptb \gnot{\hspace{-0,34cm}\tau_{F,F}} \grmptb \gcl{1} \gnl
\gcl{1} \gcl{1} \glmptb \gnot{\hspace{-0,34cm}\tau_{F,F}} \grmptb \gnl
\gcn{1}{1}{1}{2} \gelt{\s\omega} \gcn{1}{1}{1}{0} \gcl{1} \gnl 
\gob{7}{F} \gnl
\gend \hspace{2,5cm} 
\gbeg{3}{5}
\got{1}{F'} \got{1}{F'} \got{1}{X} \gnl
\gcl{1} \glmptb \gnot{\hspace{-0,34cm}\tau_{F,X}} \grmptb \gnl
\glmptb \gnot{\hspace{-0,34cm}\tau_{F,X}} \grmptb \gcl{1} \gnl
\gcl{1} \glmptb \gnot{\hspace{-0,34cm}\tau_{F,F}} \grmptb \gnl
\gob{1}{X} \gob{1}{X} \gob{1}{F}
\gend=
\gbeg{3}{5}
\got{1}{F'} \got{1}{F'} \got{1}{X} \gnl
\glmptb \gnot{\hspace{-0,34cm}\tau_{F',F'}} \grmptb \gcl{1} \gnl
\gcl{1} \glmptb \gnot{\hspace{-0,34cm}\tau_{F,X}} \grmptb \gnl
\glmptb \gnot{\hspace{-0,34cm}\tau_{F,X}} \grmptb \gcl{1} \gnl
\gob{1}{X} \gob{1}{X} \gob{1}{F}
\gend
\end{equation} 
$$ \textnormal{ \hspace{-1cm} \footnotesize $\tau_{\Id_\A,X}$ is natural w.r.t. $\omega$}  \hspace{1,8cm} \textnormal{ \footnotesize $\tau_{F,\Id_\A}$ is natural w.r.t. $\omega$}  \hspace{2,4cm}
\textnormal{\footnotesize YBE for $FFX$} $$ \vspace{-0,7cm}

\medskip

\begin{defn}
The category of right Tambara comodules over a coquasi-bimonad $(\A, F, \omega)$ we denote by $\Tau^{(\A,F,\omega)}$ and we define it as follows. Its objects are 
triples $(X,\tau_{F,X}, \rho)$, where $(X,\tau_{F,X})$ are objects of $\CQB(\K)(F)$ so that $(X, \rho)$ is a right $F$-comodule in $\K$ and the coaction 
$\rho: (X, \tau_{F, X})\to (XF, \tau_{F, XF})$ is a morphism in $\CQB(\K)(F)$. 
Morphisms of $\Tau^{(\A,F,\omega)}$ are right $F$-colinear morphisms in $\CQB(\K)(F)$. 
\end{defn}

The dual of \thref{quasi-bim monoidal} holds: 
the category $\Tau^{(\A,F,\omega)}$ of right Tambara $F$-comodules over a coquasi-bimonad $F$ in $\K$, 
is monoidal. 
Its associativity constraint is given by \equref{assoc. rcm} and $XY$ is a right $F$-comodule via \equref{F coact XY tau}, for $(X,\tau_{F,X}),(Y,\tau_{F,Y}), (Z, \tau_{F,Z})\in\Tau^{(\A,F,\omega)}$. 
\begin{center} 
\begin{tabular}{p{8cm}p{0cm}p{6.4cm}}
\begin{equation} \eqlabel{assoc. rcm}
\alpha_{X,Y,Z}=
\gbeg{6}{6}
\got{1}{X} \gvac{1} \got{1}{Y} \got{3}{Z} \gnl
\grcm \grcm \grcm \gnl
\gcl{1} \glmptb \gnot{\hspace{-0,34cm}\tau_{F,Y}} \grmptb \glmptb \gnot{\hspace{-0,34cm}\tau_{F,Z}} \grmptb \gcl{2} \gnl
\gcl{1} \gcl{1} \glmptb \gnot{\hspace{-0,34cm}\tau_{F,Z}} \grmptb \gcn{2}{1}{1}{1} \gnl
\gcl{1} \gcl{1} \gcl{1} \gcn{1}{1}{1}{2} \gelt{\omega} \gcn{1}{1}{1}{0} \gnl   
\gob{1}{X} \gob{1}{Y} \gob{1}{Z}  
\gend=
\gbeg{3}{5}
\got{1}{XYZ} \gnl
\gcl{3} \hspace{-0,42cm} \glmf  \gnl
\gvac{1} \gcl{1} \gcn{1}{1}{1}{3}  \gnl
\gvac{3} \gelt{\omega} \gnl
\gob{3}{XYZ} 
\gend
\end{equation} & & \vspace{0,24cm}
\begin{equation} \eqlabel{F coact XY tau}
\gbeg{3}{5}
\got{1}{XY} \gnl
\grcm \gnl
\gcl{1} \gcn{2}{1}{1}{3}  \gnl
\gcl{1} \gvac{1} \gcl{1} \gnl
\gob{1}{XY} \gob{3}{F} 
\gend=
\gbeg{3}{5}
\got{1}{X} \got{3}{Y} \gnl
\grcm \grcm \gnl
\gcl{1} \glmptb \gnot{\hspace{-0,34cm}\tau_{F,Y}} \grmptb \gcl{1} \gnl
\gcl{1} \gcl{1} \gmu \gnl
\gob{1}{X} \gob{1}{Y} \gob{2}{F.} 
\gend
\end{equation}
\end{tabular}
\end{center} 
Here we used the notation: 
$$
\gbeg{3}{4}
\got{1}{XYZ} \gnl
\gcl{1} \hspace{-0,42cm} \glmf  \gnl
\gvac{1} \gcl{1} \gcn{1}{1}{1}{3}  \gnl
\gob{2}{XYZ} \gob{3}{F} 
\gend=
\gbeg{4}{5}
\got{1}{X} \gvac{1} \got{1}{Y} \gvac{1} \got{1}{Z} \gnl
\grcm \grcm \grcm \gnl
\gcl{1} \glmptb \gnot{\hspace{-0,34cm}\tau_{F,Y}} \grmptb \glmptb \gnot{\hspace{-0,34cm}\tau_{F,Z}} \grmptb \gcl{2} \gnl
\gcl{1} \gcl{1} \glmptb \gnot{\hspace{-0,34cm}\tau_{F,Z}} \grmptb \gcl{1} \gnl
\gob{1}{X} \gob{1}{Y} \gob{1}{Z} \gob{1}{F} \gob{1}{F}\gob{1}{F}
\gend
$$
similarly as before and the dual statement of \prref{FFF rules} holds.

\section{Actions of monoidal categories}

Let $B:\A\to\A$ be a monad in $\K$ and $\C$ a monoidal category such that there is a quasi-monoidal faithful 
functor $\U: \C\to\Tau(\A,B)$. We will denote the objects of $\C$ by $\crta X$, and write $\U(\crta X)=(X, \tau_{B,X})$. The monoidal category $\C$ may be non-strict, 
we denote by $\alpha: (\crta X\crta Y)\crta Z\to \crta X(\crta Y\crta Z)$ its associativity constraint, 
for $\crta X,\crta Y,\crta Z\in\C$. Thus $\U(\alpha)$ is a morphism in $\Tau(\A,B)$. The unity constraints we will consider though as identities. 
Let ${}_B\K$ denote the category whose objects are left $B$-modules $M: \A'\to\A$ 
for any 0-cell $\A'$ in $\K$ and morphisms are left $B$-linear 2-cells in $\K$.  

\begin{thm} \thlabel{main} 
Let $B:\A\to\A$ be a monad and $\C$ as above. Suppose that 
there is a quasi-monoidal functor $\F:\C\to\Tau(\A,B)$ that factors through $\U: \C\to\Tau(\A,B)$ and let us write $\F(\crta X)=(X,\psi_{B,X})$. The following are equivalent:
\begin{enumerate}
\item there is an action $\C\times {}_B\K\to {}_B\K$ of $\C$ on ${}_B\K$ given by 
$(\crta X, M) \mapsto XM$ 
where $XM$ is a left $B$-module via \equref{B act XY psi} and where the category action associativity isomorphism 
$r_{X,Y,M}: (XY)M\to X(YM)$ (respectively $\rho_{X,Y,M}: X(YM)\to (XY)M$) satisfies: 
\begin{equation} \eqlabel{ro asoc isom novo}
r_{X,Y,M}=
\gbeg{4}{5}
\got{1}{X} \got{1}{Y} \got{3}{M} \gnl
\gcl{1} \gcl{1} \gu{1} \gcl{1} \gnl
\glmptb \gnot{r_{X,Y,B}}  \gcmptb \grmptb \gcl{1} \gnl
\gcl{1} \gcl{1} \glm \gnl
\gob{1}{X} \gob{1}{Y} \gob{3}{M}
\gend
\end{equation}
for every $\crta X,\crta Y\in\C$ and $M\in {}_B\K$; 
and the unity-action isomorphism is taken to be identity;
\item for every $\crta X,\crta Y\in\C$ there is an invertible 2-cell $\hat\rho: (X, \psi_{B,X})(Y, \psi_{B,Y})\to (X, \psi_{B,X})(Y, \psi_{B,Y})$ 
in $\EM^M(\K)$ that is a normalized 2-cocycle \equref{inv 2-coc} (respectively \equref{2-coc}) in $\EM^M(\K)$. 
\end{enumerate}
\end{thm}

\begin{proof}
We know from \prref{B act XM psi} that \equref{B act XY psi} determines a left $B$-module structure on $XM$. 
Then it is straightforward to prove that the category action functor is well-defined on morphisms. 
Given a natural isomorphism $r_{X,Y,M}$ satisfying \equref{ro asoc isom novo}, define $\crta\rho_{X,Y}: XY\to XYB$ by $\crta\rho_{X,Y}=r_{X,Y,B}(id_{XY}\times\eta_B)$. Conversely, given 
a 2-cell $\crta\rho_{X,Y}: XY\to XYB$ natural in $X$ and $Y$ define 
\begin{equation} \eqlabel{ro asoc isom}
r_{X,Y,M}=
\gbeg{3}{4}
\got{1}{X} \got{1}{Y} \got{1}{M} \gnl
\glmptb \gnot{\hspace{-0,34cm}\crta\rho_{X,Y}} \grmptb \gcl{1} \gnl
\gcl{1} \glm \gnl
\gob{1}{XY} \gob{3}{M.}
\gend
\end{equation}
We are going to prove that $r_{X,Y,M}$ ({\em i.e. } $\rho_{X,Y,M}$) is an isomorphism in ${}_B\K$ natural in $X,Y$ and $M$, and it satisfies the pentagon and the triangle 
defining the coherence of the action if and only if $\crta\rho_{X,Y}$ induces an invertible and normalized 2-cocycle in $\EM^M(\K)$. 
Observe that 
$r_{X,Y,M}$ and $\rho_{X,Y,M}$ are 
equal on 1-cells and that left $B$-module structure on $(XY)M$ and on $X(YM)$ are both given via $\psi_{B,XY}$ (see \equref{B act XY psi}). The only difference 
between the two of them, one encounters when dealing with the pentagon axiom, as we will see below. 

For the naturality of $r_{X,Y,M}$ take 2-cells {\em i.e.} morphisms $\zeta: \crta X\to \crta {X'}, \xi: \crta Y\to \crta {Y'}$ in $\C$ and a $B$-linear morphism 
$\beta: M\to M'$ in $\K$. The naturality then translates into the condition: 
\begin{equation} \eqlabel{natur}
\gbeg{3}{4}
\got{1}{X} \got{1}{Y} \gnl
\glmptb \gnot{\crta\rho_{X,Y}} \gcmptb \grmpb \gnl
\gbmp{\zeta} \gbmp{\xi} \gcl{1} \gnl
\gob{1}{X'} \gob{1}{Y'} \gob{1}{B}
\gend=
\gbeg{3}{4}
\got{1}{X} \got{1}{Y} \gnl
\gbmp{\zeta} \gbmp{\xi} \gnl
\glmptb \gnot{\crta\rho_{X',Y'}} \gcmptb \grmpb \gnl
\gob{1}{X'} \gob{1}{Y'} \gob{1}{B}
\gend
\end{equation}
(the morphism $\beta$ ``cancels out'' in the calculation, we will comment on it further below).  

If we suppose that the inverse $r_{X,Y,M}^{-1}$ of $r_{X,Y,M}$ also has the form \equref{ro asoc isom}, 
then $\Id_{XYM}=r_{X,Y,M}^{-1}\comp r_{X,Y,M}$ and $\Id_{XYM}=r_{X,Y,M}\comp r_{X,Y,M}^{-1}$, where $\comp$ denotes the vertical composition of 2-cells in $\K$, are equivalent to: 
\begin{equation} \eqlabel{vert comp M}
\gbeg{4}{5}
\gvac{1} \got{1}{X}  \got{1}{Y} \gnl
\gvac{1} \glmptb \gnot{\crta\rho_{X,Y}} \gcmptb \grmpb \gnl
\glmpb \gnot{\crta\rho_{X,Y}^{-1}} \gcmptb \grmpb \gcl{1} \gnl
\gcl{1} \gcl{1} \gmu \gnl
\gob{1}{X} \gob{1}{Y} \gob{2}{B}
\gend=
\gbeg{3}{4}
\got{1}{X}  \got{1}{Y} \gnl
\gcl{2} \gcl{2} \gu{1} \gnl
\gvac{2} \gcl{1} \gnl
\gob{1}{X} \gob{1}{Y} \gob{1}{B}
\gend\quad\text{and}\qquad
\gbeg{4}{5}
\gvac{1} \got{1}{X}  \got{1}{Y} \gnl
\gvac{1} \glmptb \gnot{\crta\rho_{X,Y}^{-1}} \gcmptb \grmpb \gnl
\glmpb \gnot{\crta\rho_{X,Y}} \gcmptb \grmpb \gcl{1} \gnl
\gcl{1} \gcl{1} \gmu \gnl
\gob{1}{X} \gob{1}{Y} \gob{2}{B}
\gend=
\gbeg{3}{4}
\got{1}{X}  \got{1}{Y} \gnl
\gcl{2} \gcl{2} \gu{1} \gnl
\gvac{2} \gcl{1} \gnl
\gob{1}{X} \gob{1}{Y} \gob{1}{B}
\gend
\end{equation}
where $\crta\rho_{X,Y}^{-1}$ is a formal symbol with the clear meaning. 
$B$-linearity of $r_{X,Y,M}$ is equivalent to: 
\begin{equation}  \eqlabel{2-cells EM^M}
\gbeg{4}{6}
\gvac{1} \got{1}{B} \got{1}{X} \got{1}{Y}\gnl
\gvac{1} \glmptb \gnot{\hspace{-0,34cm}\psi_{B,X}} \grmptb \gcl{1} \gnl
\gvac{1} \gcl{1} \glmptb \gnot{\hspace{-0,34cm}\psi_{B,Y}} \grmptb \gnl
\glmpb \gnot{\crta\rho_{X,Y}} \gcmptb \grmptb \gcl{1} \gnl
\gcl{1} \gcl{1} \gmu \gnl
\gob{1}{X} \gob{1}{Y} \gob{2}{B}
\gend=
\gbeg{3}{6}
\got{1}{B} \gvac{1} \got{1}{X} \got{1}{Y}\gnl
\gcl{1} \glmpb \gnot{\crta\rho_{X,Y}} \gcmptb \grmpb \gnl
\glmptb \gnot{\hspace{-0,34cm}\psi_{B,X}} \grmptb \gcl{1} \gcl{1} \gnl
\gcl{1} \glmptb \gnot{\hspace{-0,34cm}\psi_{B,Y}} \grmptb \gcl{1} \gnl
\gcl{1} \gcl{1} \gmu \gnl
\gob{1}{X} \gob{1}{Y} \gob{2}{B}
\gend
\end{equation} 
Observe that by \equref{tau B XY} the upper identity we may also write as follows: 
\begin{equation}  \eqlabel{2-cells EM^M one psi}
\gbeg{4}{5}
\gvac{1} \got{1}{B} \got{1}{X} \got{1}{Y}\gnl
\gvac{1} \glmptb \gnot{\psi_{B,XY}} \gcmptb \grmptb \gnl
\glmpb \gnot{\crta\rho_{X,Y}} \gcmptb \grmpb \gcl{1} \gnl
\gcl{1} \gcl{1} \gmu \gnl
\gob{1}{X} \gob{1}{Y} \gob{2}{B}
\gend=
\gbeg{3}{5}
\got{1}{B} \gvac{1} \got{1}{X} \got{1}{Y}\gnl
\gcl{1} \glmpb \gnot{\crta\rho_{X,Y}} \gcmptb \grmptb \gnl
\glmptb \gnot{\psi_{B,XY}} \gcmptb \grmptb \gcl{1} \gnl
\gcl{1} \gcl{1} \gmu \gnl
\gob{1}{X} \gob{1}{Y} \gob{2}{B}
\gend
\end{equation} 
Let $\crta X,\crta Y,\crta Z\in\C$ and $M\in {}_B\K$, then the pentagon for $r_{X,Y,M}$ reads: 
\begin{equation}  \eqlabel{pentagon new}
(\Id_X\times r_{Y,Z,M})\comp r_{X,YZ,M}\comp(\alpha_{X,Y,Z}\times\Id_M)=r_{X,Y,ZM}\comp r_{XY,Z,M}
\end{equation} 
where $\alpha_{X,Y,Z}:(XY)Z\to X(YZ)$ is the associativity constraint of $\C$. 
Observe that the pentagon for  $\rho_{X,Y,M}$ reads: 
\begin{equation}  \eqlabel{pentagon}
(\alpha_{X,Y,Z}\times\Id_M)\comp\rho_{XY,Z,M}\comp\rho_{X,Y,ZM}=\rho_{X,YZ,M}\comp(\Id_X\times\rho_{Y,Z,M}).
\end{equation} 
By \equref{ro asoc isom} the above two pentagon identities translate into: 
\begin{center} 
\begin{tabular}{p{7cm}p{0cm}p{7cm}}
\begin{equation} \eqlabel{monad law ro new}
\gbeg{5}{6}
\got{1}{X} \got{1}{Y} \got{1}{Z}\gnl
\glmptb \gnot{\alpha} \gcmptb \grmpb \gnl
\glmptb \gnot{\hspace{0,7cm}\crta r_{X,YZ}} \gcmpt \gcmptb \gcmpb \grmpb \gnl
\gcl{1} \glmpb \gnot{\crta r_{Y,Z}} \gcmptb \grmptb \gcl{1} \gnl
\gcl{1} \gcl{1} \gcl{1} \gmu \gnl
\gob{1}{X} \gob{1}{Y} \gob{1}{Z} \gob{2}{B}
\gend=
\gbeg{5}{6}
\gvac{1} \got{1}{X} \got{1}{Y} \got{1}{Z}\gnl
\glmp \gnot{\hspace{0,7cm}\crta r_{XY,Z}} \gcmptb \gcmptb \gcmptb \grmpb \gnl
\glmpb \gnot{\crta r_{X,Y}} \gcmptb \grmptb \gcl{1} \gcl{1} \gnl
\gcl{1} \gcl{1} \glmptb \gnot{\hspace{-0,34cm}\psi_{B,Z}} \grmptb \gcl{1} \gnl
\gcl{1} \gcl{1} \gcl{1} \gmu \gnl
\gob{1}{X} \gob{1}{Y} \gob{1}{Z} \gob{2}{B}
\gend
\end{equation}  & & 
\begin{equation} \eqlabel{monad law ro}
\gbeg{5}{6}
\gvac{1} \got{1}{X} \got{1}{Y} \got{3}{Z}\gnl
\gvac{1} \glmptb \gnot{\crta\rho_{X,Y}} \gcmptb \grmpb \gcl{1} \gnl
\gvac{1} \gcl{1} \gcl{1} \glmptb \gnot{\hspace{-0,34cm}\psi_{B,Z}} \grmptb \gnl
\glmpb \gnot{\hspace{0,34cm}\crta\rho_{XY,Z}} \gcmptb \gcmptb \grmptb \gcl{1} \gnl
\glmptb \gnot{\alpha} \gcmptb \grmpb \gmu \gnl
\gob{1}{X} \gob{1}{Y} \gob{1}{Z} \gob{2}{B}
\gend=
\gbeg{5}{5}
\gvac{1} \got{1}{X} \got{1}{Y} \got{1}{Z}\gnl
\gvac{1} \gcl{1} \glmptb \gnot{\crta\rho_{Y,Z}} \gcmptb \grmpb \gnl
\glmpb \gnot{\hspace{0,34cm}\crta\rho_{X,YZ}} \gcmptb \gcmptb \grmptb \gcl{1} \gnl
\gcl{1} \gcl{1} \gcl{1} \gmu \gnl
\gob{1}{X} \gob{1}{Y} \gob{1}{Z} \gob{2}{B}
\gend
\end{equation} 
\end{tabular}
\end{center}
respectively. 
The coherence of the associativity constraint of $\C$ and the unity constraints (which are identities) mean: 
\begin{equation} \eqlabel{normalized in EM}
\gbeg{3}{4}
\got{1}{X} \got{3}{M} \gnl
\glmptb \gnot{\hspace{-0,34cm}\crta\rho_{X,\Id}} \grmpb \gcl{1} \gnl
\gcl{1} \glm \gnl
\gob{1}{X} \gob{3}{M}
\gend=
\gbeg{2}{4}
\got{1}{X} \got{1}{M} \gnl
\gcl{2} \gcl{2} \gnl
\gob{1}{X} \gob{1}{M}
\gend=
\gbeg{3}{4}
\got{1}{} \got{1}{X} \got{1}{M} \gnl
\glmpb \gnot{\hspace{-0,34cm}\crta\rho_{\Id,X}} \grmptb \gcl{1} \gnl
\gcl{1} \glm \gnl
\gob{1}{X} \gob{3}{M}
\gend\qquad\Leftrightarrow\qquad
\gbeg{2}{5}
\got{1}{X} \gnl
\gcl{1} \gnl
\glmptb \gnot{\hspace{-0,34cm}\crta\rho_{X,\Id}} \grmpb \gnl
\gcl{1} \gcl{1} \gnl
\gob{1}{X} \gob{1}{B}
\gend=
\gbeg{2}{4}
\got{1}{X} \gnl
\gcl{1} \gu{1} \gnl
\gcl{1} \gcl{1} \gnl
\gob{1}{X} \gob{1}{B}
\gend=
\gbeg{3}{5}
\got{1}{} \got{1}{X} \gnl
\gvac{1} \gcl{1} \gnl
\glmpb \gnot{\hspace{-0,34cm}\crta\rho_{\Id,X}} \grmptb \gnl
\gcl{1} \gcl{1} \gnl
\gob{1}{X} \gob{1}{B}
\gend
\end{equation}
We have not detailed the proofs of the above claims 
but the idea is the same in all of them: one uses the 
associativity of the $B$-action on $M$, then one choses $M=B$ and composes the equation in question with the unit $\eta_B$ of the monad $B$ at the appropriate place. 
Note that \equref{vert comp M} means that the 2-cells $\crta\rho_{X,Y}$ and $\crta\rho_{X,Y}^{-1}$ induce two 2-cells inverse to each other in the 2-category $\EM^M(\K)$, 
\equref{2-cells EM^M} means that $\crta\rho_{X,Y}$ induces a 2-cell in $\EM^M(\K)$, and \equref{monad law ro new}, \equref{monad law ro} mean that $\crta\rho$ induces an 
operator which is an (inverse) 2-cocycle in $\EM^M(\K)$, while \equref{normalized in EM} means that it is normalized. 
\qed\end{proof}

\medskip

Observe that the point 2) in the above Theorem is equivalent to saying that the identities \equref{natur}, \equref{vert comp M}, 
\equref{2-cells EM^M one psi}, \equref{monad law ro new} {\em i.e.} \equref{monad law ro} and \equref{normalized in EM} hold.

\section{Representation category of coquasi-bimonads in $\K$} \selabel{rep coq}

Let $(F, \omega)$ be a coquasi-bimonad 
and $B:\A\to\A$ a monad in $\K$. 
Suppose that for every $(X, \tau_{F,X}, \rho)\in\Tau^{(\A, F, \omega)}$ - including when $X=F$ - 
there is a distributive law $\tau_{B,X}$ such that $(X, \tau_{B,X})\in\Tau(\A, B)$, the coaction $\rho: (X, \tau_{B,X})\to(XF, \tau_{B,XF})$ 
is a morphism in $\Tau(\A, B)$ ({\em i.e.} $\tau_{B,X}$ is natural with respect to $\rho$, \equref{nat rcm}) 
and the Yang-Baxter equation \equref{YBE BFX} is fulfilled. We require moreover that $\tau_{B, \Id_{\A}}$ be natural 
with respect to $\omega$, \equref{nat B omega}. 

\hspace{-0,6cm}
\begin{tabular}{p{4.3cm}p{0cm}p{4.5cm}p{0cm}p{5.3cm}}
\begin{equation}\eqlabel{nat rcm}
\gbeg{3}{5}
\got{1}{B} \got{3}{X} \gnl
\gcl{1} \gvac{1} \gcl{1} \gnl
\glmptb \gnot{\tau_{B,X}} \gcmp \grmptb \gnl
\grcm \gcl{1} \gnl
\gob{1}{X} \gob{1}{F} \gob{1}{B} 
\gend=
\gbeg{3}{5}
\got{1}{B} \got{1}{X} \gnl
\gcl{1} \grcm \gnl
\glmptb \gnot{\hspace{-0,34cm}\tau_{B,X}} \grmptb \gcl{1} \gnl
\gcl{1} \glmptb \gnot{\hspace{-0,34cm}\tau_{B,F}} \grmptb \gnl
\gob{1}{X} \gob{1}{F} \gob{1}{B.} 
\gend
\end{equation}  & & 
\begin{equation} \eqlabel{YBE BFX} 
\gbeg{3}{5}
\got{1}{B} \got{1}{F} \got{1}{X} \gnl
\gcl{1} \glmptb \gnot{\hspace{-0,34cm}\tau_{F,X}} \grmptb \gnl
\glmptb \gnot{\hspace{-0,34cm}\tau_{B,X}} \grmptb \gcl{1} \gnl
\gcl{1} \glmptb \gnot{\hspace{-0,34cm}\tau_{B,F}} \grmptb \gnl
\gob{1}{X} \gob{1}{F} \gob{1}{B}
\gend=
\gbeg{2}{5}
\got{1}{B} \got{1}{F} \got{1}{X} \gnl
\glmptb \gnot{\hspace{-0,34cm}\tau_{B,F}} \grmptb \gcl{1} \gnl
\gcl{1} \glmptb \gnot{\hspace{-0,34cm}\tau_{B,X}} \grmptb \gnl
\glmptb \gnot{\hspace{-0,34cm}\tau_{F,X}} \grmptb \gcl{1} \gnl
\gob{1}{X} \gob{1}{F} \gob{1}{B}
\gend
\end{equation}   & &  \vspace{-0,5cm}
\begin{equation}\eqlabel{nat B omega}
\gbeg{4}{5}
\got{1}{B} \got{1}{F} \got{1}{F} \got{1}{F} \gnl
\gcl{2} \gcl{2} \gcl{2} \gcl{2} \gnl
\gnl
\gcl{1} \gcn{1}{1}{1}{2} \gelt{\s\omega} \gcn{1}{1}{1}{0} \gnl  %
\gob{1}{B} \gnl
\gend=
\gbeg{4}{6}
\got{1}{B} \got{1}{F} \got{1}{F} \got{1}{F} \gnl
\glmptb \gnot{\hspace{-0,34cm}\tau_{B,F}} \grmptb \gcl{1} \gcl{1} \gnl
\gcl{1} \glmptb \gnot{\hspace{-0,34cm}\tau_{B,F}} \grmptb \gcl{1} \gnl
\gcl{1} \gcl{1} \glmptb \gnot{\hspace{-0,34cm}\tau_{B,F}} \grmptb \gnl
\gcn{1}{1}{1}{2} \gelt{\s\omega} \gcn{1}{1}{1}{0} \gcl{1} \gnl 
\gob{7}{B} \gnl
\gend  
\end{equation} 
\end{tabular}

In particular, for $X=F$ we get that $\tau_{B,F}$ is comonadic with respect to $F$ and obeys \equref{nat rcm} for $X=F$.
Let ${}_{\tau_B}\Tau^{(\A, F, \omega)}$ denote the category whose objects are triples $(X, \tau_{B,X}, \tau_{F,X})$ with the 
above properties and morphisms are 
$\zeta: (X, \tau_{B,X}, \tau_{F,X})\to (Y, \tau_{B,Y}, \tau_{F,Y})$ given by right $F$-colinear morphisms $\zeta: X\to Y$ 
in $\Mnd(\K)(B)$ and in $\Mnd(\K)(F)$ (that is, both $\tau_{B,X}$ and $\tau_{F,X}$ are 
natural with respect to $\zeta$). The latter means that the following are fulfilled: 
\begin{equation} \eqlabel{morphisms in comod}
\gbeg{2}{4}
\got{1}{X} \gnl
\gbmp{\zeta} \gnl
\grcm \gnl
\gob{1}{X} \gob{1}{F} 
\gend=
\gbeg{2}{4}
\got{1}{X} \gnl
\grcm \gnl
\gbmp{\zeta} \gcl{1} \gnl
\gob{1}{X} \gob{1}{F} 
\gend,\hspace{1cm} 
\gbeg{3}{4}
\got{1}{B} \got{1}{X} \gnl
\glmptb \gnot{\hspace{-0,34cm}\tau_{B,X}} \grmptb \gnl
\gbmp{\zeta} \gcl{1} \gnl
\gob{1}{Y} \gob{1}{B} 
\gend=
\gbeg{3}{4}
\got{1}{B} \got{1}{X} \gnl
\gcl{1} \gbmp{\zeta} \gnl
\glmptb \gnot{\hspace{-0,34cm}\tau_{B,Y}} \grmptb \gnl
\gob{1}{Y} \gob{1}{B} 
\gend,\hspace{1cm} 
\gbeg{3}{4}
\got{1}{F} \got{1}{X} \gnl
\glmptb \gnot{\hspace{-0,34cm}\tau_{F,X}} \grmptb \gnl
\gbmp{\zeta} \gcl{1} \gnl
\gob{1}{Y} \gob{1}{F} 
\gend=
\gbeg{3}{4}
\got{1}{F} \got{1}{X} \gnl
\gcl{1} \gbmp{\zeta} \gnl
\glmptb \gnot{\hspace{-0,34cm}\tau_{F,Y}} \grmptb \gnl
\gob{1}{Y} \gob{1}{F.} 
\gend
\end{equation}
We already know that the category $\Tau^{(\A, F, \omega)}$ is monoidal. In order to prove that so is ${}_{\tau_B}\Tau^{(\A, F, \omega)}$ it rests to check 
the conditions \equref{YBE BFX} and \equref{nat rcm} for $\tau_{B,XY}$ as in \equref{tau B XY}, and that $\alpha$ from \equref{assoc. rcm} is a morphism in $\Tau(\A, B)$. 
We comment this last fact, the rest we leave to the reader as an easy check. The proof that $\alpha$ is a morphism in $\Tau(\A, B)$ is analogous to the proof \equref{alfa is Tambara} 
that $\alpha$ from \equref{assoc. lm} is a morphism in $\Tau(\A, F)$. In the current case we use \equref{nat B omega} (in place of \equref{Phi nat new} from before) and 
\equref{YBE BFX} (in place of \equref{YBE BBX}, which we used to prove \equref{nat FFF}).

\medskip

Assume moreover that $B$ is a right $F$-module monad in the sense of \equref{F mod alg}--\equref{F mod alg unit}. 
In this setting we require for the 2-cells $\tau_{B,X}$ to be also natural with respect to the right $F$-module action on $B$: 
\begin{equation} \eqlabel{nat rm}
\gbeg{3}{5}
\got{1}{B} \got{1}{F} \got{1}{X} \gnl
\grmo \gcl{1} \gvac{1} \gcl{1} \gnl
\glmptb \gnot{\tau_{B,X}} \gcmp \grmptb \gnl
\gcl{1} \gvac{1} \gcl{1} \gnl
\gob{1}{X} \gob{3}{B}
\gend=
\gbeg{3}{5}
\got{1}{B} \got{1}{F} \got{1}{X} \gnl
\gcl{1} \glmptb \gnot{\hspace{-0,34cm}\tau_{F,X}} \grmptb \gnl
\glmptb \gnot{\hspace{-0,34cm}\tau_{B,X}} \grmptb \gcl{1} \gnl
\gcl{1} \grm \gnl
\gob{1}{X} \gob{1}{B} 
\gend
\end{equation}
As above, it is directly proved that $\tau_{B,XY}$ has the latter property, so the new category is also monoidal. We will again abuse notation and 
denote this category also by ${}_{\tau_B}\Tau^{(\A, F, \omega)}$, understanding implicitly the added properties once we assume that $B$ is a right $F$-module monad. 

\medskip

\begin{lma} \lelabel{psi_2}
Let $(F, \omega)$ be a coquasi-bimonad, $B:\A\to\A$ be a right $F$-module monad with $\psi_{B,F}$ given by \equref{psi_2 for BF}, 
$(X, \rho)$ a right $F$-comodule and $\tau_{B,X}$ 
such that $\rho: (X, \tau_{B,X})\to(XF, \tau_{B,XF})$ is a morphism in $\Tau(\A, B)$. 
Then $(X, \psi_{B,X})\in\Tau(\A, B)$ with $\psi_{B,X}$ being given via:  
\begin{center} 
\begin{tabular}{p{5cm}p{1cm}p{5cm}}
\begin{equation}\eqlabel{psi_2 for BF}
\psi_{B,F}=
\gbeg{3}{5}
\got{1}{B} \got{2}{F} \gnl
\gcl{1} \gcmu \gnl
\glmptb \gnot{\hspace{-0,34cm}\tau_{B,F}} \grmptb \gcl{1} \gnl
\gcl{1} \grmo \gcl{1} \gnl
\gob{1}{F} \gob{1}{B} 
\gend
\end{equation} & & 
\begin{equation}\eqlabel{psi_2 for X}
\psi_{B,X}=
\gbeg{3}{5}
\got{1}{B} \got{1}{X} \gnl
\gcl{1} \grcm \gnl
\glmptb \gnot{\hspace{-0,34cm}\tau_{B,X}} \grmptb \gcl{1} \gnl
\gcl{1} \grm \gnl
\gob{1}{X} \gob{1}{B.} 
\gend
\end{equation}
\end{tabular}
\end{center}
\end{lma}

\begin{proof}
We only prove the first distributive law for $\psi_{B,X}$:
$$
\gbeg{3}{5}
\got{1}{B}\got{1}{B}\got{1}{X}\gnl
\gmu \gcn{1}{1}{1}{0} \gnl
\gvac{1} \hspace{-0,34cm} \glmptb \gnot{\hspace{-0,34cm}\psi_{B,X}} \grmptb  \gnl
\gvac{1} \gcl{1} \gcl{1} \gnl
\gvac{1} \gob{1}{X} \gob{1}{B}
\gend=
\gbeg{4}{5}
\got{1}{B} \got{1}{B} \got{1}{\hspace{-0,24cm} X} \gnl
\gmu \hspace{-0,22cm} \grcm \gnl
\gvac{1} \glmptb \gnot{\hspace{-0,34cm}\tau_{B,X}} \grmptb \gcl{1} \gnl
\gvac{1} \gcl{1} \grm \gnl
\gvac{1} \gob{1}{X} \gob{1}{B} 
\gend\stackrel{d.l.}{=}
\gbeg{4}{7}
\got{1}{B} \got{1}{B} \got{1}{X} \gnl
\gcl{2} \gcl{1} \grcm \gnl
\gvac{1} \glmptb \gnot{\hspace{-0,34cm}\tau_{B,X}} \grmptb \gcl{2} \gnl
\glmptb \gnot{\hspace{-0,34cm}\tau_{B,X}} \grmptb \gcl{1} \gnl
\gcl{2} \gmu \gcn{1}{1}{1}{0} \gnl
\gvac{2} \hspace{-0,34cm} \grm \gnl
\gob{1}{X} \gob{3}{B} 
\gend\stackrel{\equref{F mod alg}}{=}
\gbeg{5}{8}
\got{1}{B} \got{1}{B} \got{1}{X} \gnl
\gcl{2} \gcl{1} \grcm \gnl
\gvac{1} \glmptb \gnot{\hspace{-0,34cm}\tau_{B,X}} \grmptb \gcn{1}{1}{1}{2} \gnl
\glmptb \gnot{\hspace{-0,34cm}\tau_{B,X}} \grmptb \gcl{1} \gcmu \gnl
\gcl{2} \gcl{1} \glmptb \gnot{\hspace{-0,34cm}\tau_{B,F}} \grmptb \gcl{1} \gnl
\gvac{1} \grm \grm \gnl
\gcl{1} \gwmu{3} \gnl
\gob{1}{X} \gob{3}{B} 
\gend\stackrel{comod.}{\stackrel{\equref{nat rcm}}{=}}
\gbeg{5}{9}
\got{1}{B} \gvac{1} \got{1}{B} \got{1}{X} \gnl
\gcl{1} \gvac{1} \gcl{1} \grcm \gnl
\gcl{1} \gvac{1} \glmptb \gnot{\hspace{-0,34cm}\tau_{B,X}} \grmptb \gcl{1} \gnl
\gcl{1} \gvac{1} \gcn{1}{1}{1}{-1} \grm \gnl
\gcl{1} \grcm \gcl{3} \gnl
\glmptb \gnot{\hspace{-0,34cm}\tau_{B,X}} \grmptb \gcl{1} \gnl
\gcl{1} \grm \gnl
\gcl{1} \gwmu{3} \gnl
\gob{1}{X} \gob{3}{B} 
\gend=
\gbeg{3}{5}
\got{1}{B}\got{1}{B}\got{1}{X}\gnl
\gcl{1} \glmpt \gnot{\hspace{-0,34cm}\psi_{B,X}} \grmptb \gnl
\glmptb \gnot{\hspace{-0,34cm}\psi_{B,X}} \grmptb \gcl{1} \gnl
\gcl{1} \gmu \gnl
\gob{1}{X} \gob{2}{B}
\gend
$$
\qed\end{proof}

A coquasi-bimonad in $\K$ is a proper comonad. Given a coquasi-bimonad $(F, \omega)$ and a monad $B:\A\to\A$ in $\K$. 
Then it is directly proved that the 2-cells $FF\to B$ in $\K$ form a monad in $\K$, that is actually a convolution algebra in the monoidal category $\K(\A)$.

\begin{prop} \prlabel{Sch case}
Let $(F, \omega)$ be a coquasi-bimonad and $B:\A\to\A$ a right $F$-module monad in $\K$ with $\psi_{B,F}$ given by \equref{psi_2 for BF}. 
The following are equivalent: 
\begin{enumerate}
\item 
set $\C={}_{\tau_B}\Tau^{(\A, F, \omega)}$ and let $\psi_{B,X}$ 
be given by \equref{psi_2 for X}; given a 2-cell $\sigma: FF\to B$ 
such that \equref{nat sigma} holds, the 2-cell $\crta\rho_{X,Y}$ given by \equref{crta ro Sch} 
satisfies the conditions in \thref{main}; 
\begin{center}
\begin{tabular} {p{6cm}p{0cm}p{6.8cm}} 
\begin{equation} \eqlabel{nat sigma}
\gbeg{3}{5}
\got{1}{F} \got{1}{F} \got{1}{X} \gnl
\gcl{1} \glmptb \gnot{\hspace{-0,34cm}\tau_{F,X}} \grmptb \gnl
\glmptb \gnot{\hspace{-0,34cm}\tau_{F,X}} \grmptb \gcl{1} \gnl
\gcl{1} \glmpt \gnot{\hspace{-0,34cm}\sigma} \grmptb \gnl
\gob{1}{X} \gob{3}{B}
\gend=
\gbeg{3}{5}
\got{1}{F} \got{1}{F} \got{1}{X} \gnl
\glmpt \gnot{\hspace{-0,34cm}\sigma} \grmptb \gcl{1} \gnl
\gvac{1} \glmptb \gnot{\hspace{-0,34cm}\tau_{B,X}} \grmptb \gnl
\gvac{1} \gcl{1} \gcl{1} \gnl
\gvac{1} \gob{1}{X} \gob{1}{B}
\gend
\end{equation} & &
\begin{equation} \eqlabel{crta ro Sch}
\crta\rho_{X,Y}=
\gbeg{3}{5}
\got{1}{X} \got{3}{Y} \gnl
\grcm \grcm \gnl
\gcl{1} \glmptb \gnot{\hspace{-0,34cm}\tau_{F,Y}} \grmptb \gcl{1} \gnl
\gcl{1} \gcl{1} \glmpt \gnot{\hspace{-0,34cm}\sigma} \grmptb \gnl
\gob{1}{X} \gob{1}{Y} \gob{3}{B} 
\gend
\end{equation}
\end{tabular}
\end{center}
\item $(B,F, \psi_{B,F}, \mu_M, \eta_M, \Epsilon_F, \beta)$ 
is a Sweedler's Hopf datum where $\mu_M$ and $\beta$ are given by: 
\begin{equation} \eqlabel{Sw datum}
\mu_M=
\gbeg{3}{5}
\got{2}{F} \got{2}{F} \gnl
\gcmu \gcmu \gnl
\gcl{1} \glmptb \gnot{\hspace{-0,34cm}\tau_{F,F}} \grmptb \gcl{1} \gnl
\gmu \glmpt \gnot{\hspace{-0,34cm}\sigma} \grmptb \gnl
\gob{2}{F} \gob{3}{B} \gnl
\gend
\hspace{2cm}\text{and}\hspace{1,6cm}
\beta=
\gbeg{6}{6}
\got{2}{F} \got{2}{F} \got{2}{F} \gnl
\gcmu \gcmu \gcmu \gnl
\gcl{1} \glmptb \gnot{\hspace{-0,34cm}\tau_{F,F}} \grmptb \glmptb \gnot{\hspace{-0,34cm}\tau_{F,F}} \grmptb \gcl{2} \gnl
\gcl{1} \gcl{1} \glmptb \gnot{\hspace{-0,34cm}\tau_{F,F}} \grmptb \gcn{2}{1}{1}{1} \gnl
\gcn{1}{1}{1}{2} \gelt{\hspace{0,17cm}\omega^{-1}} \gcn{1}{1}{1}{0} \gcl{1} \gcl{1} \gcl{1} \gnl   
\gvac{3} \gob{1}{F} \gob{1}{F} \gob{1}{F}  
\gend
\end{equation}
where the Sweedler's 2-cocycle $\sigma$ in $\K$ is invertible in the convolution algebra $\K(\A)(FF,B)$. Here $\eta_M=\eta_F\times\eta_B$.  
\end{enumerate}
\end{prop}

\begin{proof}
A faithful functor is provided by the forgetful functor $\F:{}_{\tau_B}\Tau^{(\A, F, \omega)}\to \Tau(\A,B)$, given by $\F(X, \tau_{B,X}, \tau_{F,X},\rho)=(X, \tau_{B,X})$, it is 
clearly quasi-monoidal. The 2-cell $\crta\rho_{X,Y}$ given in \equref{crta ro Sch} satisfies the condition \equref{natur} by \equref{morphisms in comod}. 
Let us check when
$\crta\rho_{X,Y}^{-1}:=
\gbeg{4}{5}
\got{1}{X} \got{3}{Y} \gnl
\grcm \grcm \gnl
\gcl{1} \glmptb \gnot{\hspace{-0,34cm}\tau_{F,Y}} \grmptb \gcl{1} \gnl
\gcl{1} \gcl{1} \glmpt \gnot{\hspace{-0,34cm}\tau} \grmptb \gnl
\gob{1}{X} \gob{1}{Y} \gob{3}{B} 
\gend$ 
together with the above $\crta\rho_{X,Y}$ fulfills the identity \equref{vert comp M}. We have: 
$$
\gbeg{6}{11}
\got{1}{X} \got{7}{Y} \gnl
\grcm \gvac{2} \gcl{2} \gnl
\gcl{1} \gcn{2}{2}{1}{5} \gnl
\gcl{1} \gvac{3} \grcm \gnl
\gcl{1} \gvac{2} \glmptb \gnot{\hspace{-0,34cm}\tau_{F,Y}} \grmptb \gcl{1} \gnl
\gcl{1} \gvac{2} \gcn{1}{1}{1}{-1} \glmpt \gnot{\hspace{-0,34cm}\sigma} \grmptb \gnl
\grcm \grcm \gvac{1} \gcl{3} \gnl
\gcl{1} \glmptb \gnot{\hspace{-0,34cm}\tau_{F,Y}} \grmptb \gcl{1} \gnl
\gcl{2} \gcl{2} \glmpt \gnot{\hspace{-0,34cm}\tau} \grmptb \gnl
\gvac{3} \gwmu{3} \gnl
\gob{1}{X} \gob{1}{Y} \gob{5}{B} 
\gend=
\gbeg{6}{9}
\got{1}{X} \got{5}{Y} \gnl
\grcm \gvac{1} \grcm \gnl
\gcl{1} \gcn{2}{1}{1}{2} \gcl{1} \gcn{1}{1}{1}{2} \gnl
\gcl{1} \gcmu \gcl{1} \gcmu \gnl
\gcl{1} \gcl{1} \glmptb \gnot{\hspace{-0,34cm}\tau_{F,Y}} \grmptb \gcl{1} \gcl{2} \gnl
\gcl{1} \glmptb \gnot{\hspace{-0,34cm}\tau_{F,Y}} \grmptb \glmptb \gnot{\hspace{-0,34cm}\tau_{F,F}} \grmptb \gnl
\gcl{2} \gcl{2} \glmpt \gnot{\hspace{-0,34cm}\tau} \grmptb \glmpt \gnot{\hspace{-0,34cm}\sigma} \grmptb \gnl
\gvac{3} \gwmu{3} \gnl
\gob{1}{X} \gob{1}{Y} \gob{5}{B} 
\gend=
\gbeg{6}{9}
\got{1}{X} \got{3}{Y} \gnl
\grcm \grcm \gnl
\gcl{1} \glmptb \gnot{\hspace{-0,34cm}\tau_{F,Y}} \grmptb \gcn{2}{2}{1}{4} \gnl
\gcl{1} \gcl{1} \gcn{1}{1}{1}{2} \gnl
\gcl{1} \gcl{1} \gcmu \gcmu \gnl
\gcl{1} \gcl{1} \gcl{1} \glmptb \gnot{\hspace{-0,34cm}\tau_{F,F}} \grmptb \gcl{1} \gnl
\gcl{2} \gcl{2} \glmpt \gnot{\hspace{-0,34cm}\tau} \grmptb \glmpt \gnot{\hspace{-0,34cm}\sigma} \grmptb \gnl
\gvac{3} \gwmu{3} \gnl
\gob{1}{X} \gob{1}{Y} \gob{5}{B} 
\gend\stackrel{*}{=}
\gbeg{4}{6}
\got{1}{X} \got{3}{Y} \gnl
\grcm \grcm \gnl
\gcl{1} \glmptb \gnot{\hspace{-0,34cm}\tau_{F,Y}} \grmptb \gcu{1}\gnl
\gcl{1} \gcl{1} \gcu{1} \gnl
\gcl{1} \gcl{1} \gu{1} \gnl
\gob{1}{X} \gob{1}{Y} \gob{1}{B} 
\gend=
\gbeg{3}{4}
\got{1}{X}  \got{1}{Y} \gnl
\gcl{2} \gcl{2} \gu{1} \gnl
\gvac{2} \gcl{1} \gnl
\gob{1}{X} \gob{1}{Y} \gob{1}{B}
\gend
$$
where in the first identity we applied the comodule law and naturality of $\tau_{F,Y}$ with respect to the comodule structure, 
in the second one the distributive law for $\tau_{F,Y}$, 
the identity * holds true if and only if $\tau$ is the convolution inverse for $\sigma$, 
and finally we applied the counital distributive law for $\tau_{F,Y}$ and a comodule property. The other identity is shown in the same way 
reversing the order of the 2-cocycles. 

Next, for the identity  \equref{2-cells EM^M one psi} we compute using the same techniques:
$$
L:=\gbeg{7}{13}
\got{1}{B} \got{1}{X} \got{7}{Y} \gnl
\gcl{1} \grcm \gvac{2} \gcl{2} \gnl
\gcl{1} \gcl{1} \gcn{2}{2}{1}{5} \gnl
\gcl{1} \gcl{1} \gvac{3} \grcm \gnl
\gcl{1} \gcl{1} \gvac{2} \glmptb \gnot{\hspace{-0,34cm}\tau_{F,Y}} \grmptb \gcl{1} \gnl
\gcl{1} \gcl{1} \gvac{2} \gcn{1}{1}{1}{-1} \glmpt \gnot{\hspace{-0,34cm}\sigma} \grmptb \gnl
\gcl{1} \grcm \grcm \gvac{1} \gcl{5} \gnl
\glmptb \gnot{\hspace{-0,34cm}\tau_{B,X}} \grmptb \glmptb \gnot{\hspace{-0,34cm}\tau_{F,Y}} \grmptb\gcl{1} \gnl
\gcl{4} \glmptb \gnot{\hspace{-0,34cm}\tau_{B,Y}} \grmptb \gmu \gnl
\gvac{1} \gcl{1} \gcl{1} \gcn{1}{1}{2}{1} \gnl
\gvac{1} \gcl{2} \grm \gnl
\gvac{2} \gwmu{5} \gnl
\gob{1}{X} \gob{1}{Y} \gob{5}{B} 
\gend\stackrel{comod.}{\stackrel{\equref{nat rcm}}{=}}
\gbeg{7}{11}
\got{1}{B} \got{1}{X} \got{5}{Y} \gnl
\gcl{1} \grcm \gvac{1} \grcm \gnl
\gcl{1} \gcl{1} \gcn{2}{1}{1}{2} \gcl{1} \gcn{1}{1}{1}{2} \gnl
\gcl{1} \gcl{1} \gcmu \gcl{1} \gcmu \gnl
\gcl{1} \gcl{1} \gcl{1} \glmptb \gnot{\hspace{-0,34cm}\tau_{F,Y}} \grmptb \gcl{1} \gcl{2} \gnl
\glmptb \gnot{\hspace{-0,34cm}\tau_{B,X}} \grmptb \glmptb \gnot{\hspace{-0,34cm}\tau_{F,Y}} \grmptb \glmpt \gnot{\hspace{-0,34cm}\tau_{F,F}} \grmptb \gnl
\gcl{4} \glmptb \gnot{\hspace{-0,34cm}\tau_{B,Y}} \grmptb \gmu \glmpt \gnot{\hspace{-0,34cm}\sigma} \grmptb \gnl
\gvac{1} \gcl{1} \gcl{1} \gcn{1}{1}{2}{1} \gvac{2} \gcl{2} \gnl
\gvac{1} \gcl{2} \grm \gnl
\gvac{2} \gwmu{5} \gnl
\gob{1}{X} \gob{1}{Y} \gob{5}{B} 
\gend\stackrel{com.d.l.}{=}
\gbeg{5}{7}
\got{1}{B} \got{1}{X} \got{3}{Y} \gnl
\gcl{1} \grcm \grcm \gnl
\glmptb \gnot{\hspace{-0,34cm}\tau_{B,X}} \grmptb \glmptb \gnot{\hspace{-0,34cm}\tau_{F,Y}} \grmptb \gcl{1} \gnl
\gcl{3} \glmptb \gnot{\hspace{-0,34cm}\tau_{B,Y}} \grmptb \glmptb \gnot{\hspace{-0,34cm}\mu_M} \grmptb \gnl
\gvac{1} \gcl{2} \grm \gcl{1} \gnl
\gvac{2} \gwmu{3} \gnl
\gob{1}{X} \gob{1}{Y} \gob{3}{B} 
\gend\stackrel{*}{=}
\gbeg{5}{8}
\got{1}{B} \got{1}{X} \got{3}{Y} \gnl
\gcl{1} \grcm \grcm \gnl
\glmptb \gnot{\hspace{-0,34cm}\tau_{B,X}} \grmptb \glmptb \gnot{\hspace{-0,34cm}\tau_{F,Y}} \grmptb \gcl{1} \gnl
\gcl{4} \glmptb \gnot{\hspace{-0,34cm}\tau_{B,Y}} \grmptb \gcl{1} \gcl{1} \gnl
\gvac{1} \gcl{3} \glmptb \gnot{\psi_{B,FF}} \gcmptb \grmptb \gnl
\gvac{2} \glmpt \gnot{\hspace{-0,34cm}\sigma} \grmptb \gcl{1} \gnl
\gvac{3} \gmu \gnl
\gob{1}{X} \gob{1}{Y} \gob{4}{B} 
\gend=
$$

$$
\gbeg{7}{11}
\got{1}{B} \got{1}{X} \got{3}{Y} \gnl
\gcl{1} \grcm \grcm \gnl
\glmptb \gnot{\hspace{-0,34cm}\tau_{B,X}} \grmptb \glmptb \gnot{\hspace{-0,34cm}\tau_{F,Y}} \grmptb \gcn{2}{2}{1}{4} \gnl
\gcl{7} \glmptb \gnot{\hspace{-0,34cm}\tau_{B,Y}} \grmptb \gcn{1}{1}{1}{2}  \gnl
\gvac{1} \gcl{6} \gcl{1} \gcmu \gcmu \gnl
\gvac{2} \glmptb \gnot{\hspace{-0,34cm}\tau_{B,F}} \grmptb \glmptb \gnot{\hspace{-0,34cm}\tau_{F,F}} \grmptb\gcl{1} \gnl
\gvac{2} \gcl{1} \glmptb \gnot{\hspace{-0,34cm}\tau_{B,F}} \grmptb \gmu \gnl
\gvac{2} \glmpt \gnot{\hspace{-0,34cm}\sigma} \grmptb \gcl{1} \gcn{1}{1}{2}{1} \gnl
\gvac{3} \gcl{1} \grm \gnl
\gvac{3} \gmu \gnl
\gob{1}{X} \gob{1}{Y} \gob{4}{B} 
\gend\stackrel{com.d.l.}{\stackrel{comod.}{=}}
\gbeg{7}{11}
\got{1}{B} \got{1}{X} \got{5}{Y} \gnl
\gcl{3} \grcm \gvac{1} \grcm \gnl
\gvac{1} \gcl{1} \gcn{2}{1}{1}{3} \grcm  \gcn{1}{1}{-1}{1} \gnl
\gvac{1} \grcm \glmptb \gnot{\hspace{-0,34cm}\tau_{F,Y}} \grmptb \gcl{1} \gcl{2} \gnl
\glmptb \gnot{\hspace{-0,34cm}\tau_{B,X}} \grmptb \glmptb \gnot{\hspace{-0,34cm}\tau_{F,Y}} \grmptb \glmptb \gnot{\hspace{-0,34cm}\tau_{F,F}} \grmptb  \gnl
\gcl{5} \glmptb \gnot{\hspace{-0,34cm}\tau_{B,Y}} \grmptb \gcl{1} \gcl{2} \gmu \gnl
\gvac{1} \gcl{4} \glmptb \gnot{\hspace{-0,34cm}\tau_{B,F}} \grmptb \gcn{1}{2}{4}{3} \gnl
\gvac{2} \gcl{1} \glmptb \gnot{\hspace{-0,34cm}\tau_{B,F}} \grmptb \gnl
\gvac{2} \glmpt \gnot{\hspace{-0,34cm}\sigma} \grmptb \grm  \gnl
\gvac{3} \gmu \gnl
\gob{1}{X} \gob{1}{Y} \gob{4}{B} 
\gend\stackrel{\equref{YBE BFX}}{\stackrel{\equref{nat rcm}}{=}}
\gbeg{6}{11}
\got{1}{B} \got{1}{X} \got{5}{Y} \gnl
\gcl{3} \grcm \gvac{1} \gcl{1} \gnl
\gvac{1} \gcl{1} \gcn{2}{1}{1}{3} \grcm  \gnl
\gvac{1} \grcm \glmptb \gnot{\hspace{-0,34cm}\tau_{F,Y}} \grmptb \gcl{1} \gnl
\glmptb \gnot{\hspace{-0,34cm}\tau_{B,X}} \grmptb \gcl{1} \gcl{1} \gmu  \gnl
\gcl{5} \glmptb \gnot{\hspace{-0,34cm}\tau_{B,F}} \grmptb \grcm  \gcn{1}{1}{0}{1} \gnl
\gvac{1} \gcl{1} \glmptb \gnot{\hspace{-0,34cm}\tau_{B,Y}} \grmptb \gcl{1} \gcl{2} \gnl
\gvac{1} \glmptb \gnot{\hspace{-0,34cm}\tau_{F,Y}} \grmptb \glmptb \gnot{\hspace{-0,34cm}\tau_{B,F}} \grmptb \gnl
\gcl{2} \gcl{2} \glmpt \gnot{\hspace{-0,34cm}\sigma} \grmptb \grm  \gnl
\gvac{3} \gmu \gnl
\gob{1}{X} \gob{1}{Y} \gob{4}{B} 
\gend\stackrel{\equref{nat rcm}}{=}
\gbeg{6}{9}
\gvac{1} \got{1}{B} \got{1}{X} \got{3}{Y} \gnl
\gvac{1} \gcl{1} \grcm \grcm \gnl
\gvac{1} \glmptb \gnot{\hspace{-0,34cm}\tau_{B,X}} \grmptb \glmptb \gnot{\hspace{-0,34cm}\tau_{F,Y}} \grmptb \gcl{1} \gnl
\gcn{2}{1}{3}{1} \glmptb \gnot{\hspace{-0,34cm}\tau_{B,Y}} \grmptb \gmu \gnl
\grcm \grcm \gcn{1}{1}{-1}{0} \gcn{1}{2}{0}{1} \gnl
\gcl{3} \glmptb \gnot{\hspace{-0,34cm}\tau_{B,F}} \grmptb \gcl{1} \gcn{1}{1}{0}{1} \gnl
\gcl{2} \gcl{2} \glmpt \gnot{\hspace{-0,34cm}\sigma} \grmptb \grm  \gnl
\gvac{3} \gmu \gnl
\gob{1}{X} \gob{1}{Y} \gob{4}{B} 
\gend:=R
$$
At the place * the identity holds true if and only if the $F$-action on $B$ is twisted by $\sigma$, {\em i.e.} if \equref{weak action} holds. 
Since the rest in the above computation uses the conditions holding in the category $\C$, the equality $L=R$ expressing \equref{2-cells EM^M one psi} 
is equivalent to \equref{weak action}. 

\medskip

Before we proceed we note that we have the following identity: 
\begin{equation} \eqlabel{ab}
\gbeg{6}{6}
\got{1}{X} \gvac{1} \got{1}{Y} \got{3}{Z} \gnl
\grcm \grcm \grcm \gnl
\gcl{1} \glmptb \gnot{\hspace{-0,34cm}\tau_{F,Y}} \grmptb \glmptb \gnot{\hspace{-0,34cm}\tau_{F,Z}} \grmptb \gcl{2} \gnl
\gcl{1} \gcl{1} \glmptb \gnot{\hspace{-0,34cm}\tau_{F,Z}} \grmptb \gcl{1} \gnl
\glmptb \gnot{\alpha} \gcmptb \grmptb \glmptb \gnot{\beta} \gcmptb \grmptb \gnl
\gob{1}{X} \gob{1}{Y} \gob{1}{Z}  \gob{1}{F} \gob{1}{F} \gob{1}{F} 
\gend=
\gbeg{3}{5}
\got{1}{XYZ} \gnl
\gcl{1} \hspace{-0,42cm} \glmf  \gnl
\gvac{1} \gcl{1} \gcn{1}{1}{1}{3} \gnl
\gvac{1} \gbmp{\alpha} \gvac{1} \gbmp{\beta} \gnl
\gvac{1} \gob{1}{XYZ}  \gob{3}{FFF} 
\gend= 
\gbeg{5}{6}
\got{1}{XYZ} \gnl
\gcl{2} \hspace{-0,42cm} \glmf \gnl
\gvac{2} \gcn{1}{1}{1}{6} \gnl
\gvac{1} \gcl{2} \hspace{-0,42cm} \glmf \gvac{1} \gcmu \gnl
\gvac{3} \gelt{\omega} \gvac{1} \gelt{\hspace{0,14cm}\omega^{-1}} \gcl{1} \gnl
\gvac{2} \gob{1}{XYZ}  \gob{7}{FFF} 
\gend\stackrel{\equref{FF module rule}}{\stackrel{coass.}{=}}
\gbeg{3}{5}
\got{1}{XYZ} \gnl
\gcl{3} \hspace{-0,42cm} \glmf \gnl
\gvac{2} \gcn{1}{1}{1}{3} \gnl
\gvac{3} \gcl{1} \gnl
\gvac{1} \gob{1}{XYZ} \gob{3}{FFF}
\gend=
\gbeg{6}{5}
\got{1}{X} \gvac{1} \got{1}{Y} \got{3}{Z} \gnl
\grcm \grcm \grcm \gnl
\gcl{1} \glmptb \gnot{\hspace{-0,34cm}\tau_{F,Y}} \grmptb \glmptb \gnot{\hspace{-0,34cm}\tau_{F,Z}} \grmptb \gcl{2} \gnl
\gcl{1} \gcl{1} \glmptb \gnot{\hspace{-0,34cm}\tau_{F,Z}} \grmptb \gcl{1} \gnl
\gob{1}{X} \gob{1}{Y} \gob{1}{Z}  \gob{1}{F} \gob{1}{F} \gob{1}{F} 
\gend
\end{equation}
where $\alpha$ is the associativity constraint from \thref{quasi-bim monoidal}. 
Now, as for \equref{monad law ro}, we find: 
$$\Sigma:=
\gbeg{7}{11}
\gvac{1} \got{1}{X} \gvac{1} \got{1}{Y} \got{3}{Z} \gnl
\gvac{1} \grcm \grcm \gcl{2} \gnl
\gvac{1} \gcl{1} \glmptb \gnot{\hspace{-0,34cm}\tau_{F,Y}} \grmptb \gcl{1} \gnl
\gcn{1}{1}{3}{1} \gvac{1} \gcl{1} \glmpt \gnot{\hspace{-0,34cm}\sigma} \grmptb \grcm \gnl
\grcm \grcm \glmptb \gnot{\hspace{-0,34cm}\tau_{B,Z}} \grmptb \gcl{1} \gnl
\gcl{1} \glmptb \gnot{\hspace{-0,34cm}\tau_{F,Y}} \grmptb \gcl{1} \gcn{1}{1}{1}{0} \grm \gnl
\gcl{3} \gcl{3} \gmu \hspace{-0,2cm} \grcm  \gcn{1}{1}{0}{1} \gnl
\gvac{3} \glmptb \gnot{\hspace{-0,34cm}\tau_{F,Z}} \grmptb \gcl{1} \gcl{2} \gnl
\gvac{3} \gcn{1}{1}{1}{0} \glmpt \gnot{\hspace{-0,34cm}\sigma} \grmptb \gnl
\gvac{1} \hspace{-0,22cm} \glmptb \gnot{\alpha} \gcmptb \grmptb  \gvac{2} \hspace{-0,2cm} \gmu \gnl
\gvac{2} \hspace{-0,24cm} \gob{1}{X} \gob{1}{Y} \gob{1}{Z} \gob{5}{B} 
\gend\stackrel{\equref{nat rcm}}{\stackrel{comod.}{=}}
\gbeg{8}{12}
\got{1}{X} \gvac{2} \got{1}{Y} \got{5}{Z} \gnl
\grcm \gvac{1} \grcm \gvac{1} \gcl{4} \gnl
\gcl{1} \gcn{2}{1}{1}{2} \gcl{1} \gcn{1}{1}{1}{2} \gnl
\gcl{1} \gcmu \gcl{1} \gcmu \gnl
\gcl{1} \gcl{1} \glmptb \gnot{\hspace{-0,34cm}\tau_{F,Y}} \grmptb \gcl{1} \gcl{2} \gnl
\gcl{1} \glmptb \gnot{\hspace{-0,34cm}\tau_{F,Y}} \grmptb \glmptb \gnot{\hspace{-0,34cm}\tau_{F,F}} \grmptb \gvac{1} \grcm \gnl
\gcl{4} \gcl{4} \gmu \glmpt \gnot{\hspace{-0,34cm}\sigma} \grmptb \gcl{1} \gcn{1}{1}{1}{2} \gnl
\gvac{2} \gcn{3}{1}{2}{5} \glmptb \gnot{\hspace{-0,34cm}\tau_{B,Z}} \grmptb \gcmu \gnl
\gvac{4} \glmptb \gnot{\hspace{-0,34cm}\tau_{F,Z}} \grmptb  \glmptb \gnot{\hspace{-0,34cm}\tau_{B,F}} \grmptb \gcl{1} \gnl
\gvac{2} \gcn{3}{1}{5}{1} \glmpt \gnot{\hspace{-0,34cm}\sigma} \grmptb \grm \gnl
\glmptb \gnot{\alpha} \gcmptb \grmptb  \gvac{3} \gmu \gnl
\gob{1}{X} \gob{1}{Y} \gob{1}{Z} \gob{8}{B} 
\gend\stackrel{com.d.l.}{\stackrel{mod.d.l.}{\stackrel{\equref{nat sigma}}{=}}}
\gbeg{8}{13}
\gvac{1} \got{1}{X} \gvac{1} \got{1}{Y} \got{6}{Z} \gnl
\gvac{1} \grcm \grcm \gcn{1}{2}{4}{4} \gnl
\gvac{1} \gcl{1} \glmptb \gnot{\hspace{-0,34cm}\tau_{F,Y}} \grmptb \gcn{1}{1}{1}{3} \gnl
\gcn{1}{1}{3}{1} \gcn{2}{1}{3}{1} \hspace{-0,22cm} \gcmu \gcmu \grcm \gnl
\gcn{1}{6}{2}{2} \gcn{2}{6}{2}{2} \gcl{3} \glmptb \gnot{\hspace{-0,34cm}\tau_{F,F}} \grmptb \glmptb \gnot{\hspace{-0,34cm}\tau_{F,Z}} \grmptb \gcl{4} \gnl
\gvac{4} \gcl{1} \glmptb \gnot{\hspace{-0,34cm}\tau_{F,Z}} \grmptb \gcl{2} \gnl
\gvac{4} \glmptb \gnot{\hspace{-0,34cm}\tau_{F,Z}} \grmptb \gcl{1} \gnl
\gvac{3} \glmptb \gnot{\hspace{-0,34cm}\tau_{F,Z}} \grmptb \gcl{1} \glmpt \gnot{\hspace{-0,34cm}\sigma} \grmptb \gnl
\gvac{3} \gcn{1}{2}{1}{0} \gmu \gvac{1} \glmptb \gnot{\hspace{-0,34cm}\psi_{B,F}} \grmptb \gnl 
\gvac{5} \gcn{1}{1}{0}{1} \gcn{2}{1}{3}{1} \gcl{2} \gnl 
\gvac{1} \hspace{-0,3cm} \glmptb \gnot{\alpha} \gcmptb \grmptb \gvac{2} \hspace{-0,24cm} \glmpt \gnot{\hspace{-0,34cm}\sigma} \grmptb \gnl
\gvac{2} \gcn{1}{1}{0}{0} \gcn{1}{1}{0}{0} \gcn{1}{1}{0}{0} \gvac{2} \gwmu{3} \gnl
\gvac{2} \hspace{-0,24cm} \gob{1}{X} \gob{1}{Y} \gob{1}{Z} \gob{7}{B} 
\gend
$$

$$\stackrel{YBE \equref{dual cond}}{=}
\gbeg{8}{13}
\gvac{1} \got{1}{X} \gvac{1} \got{1}{Y} \got{6}{Z} \gnl
\gvac{1} \grcm \grcm \gcn{1}{2}{4}{4} \gnl
\gvac{1} \gcl{1} \glmptb \gnot{\hspace{-0,34cm}\tau_{F,Y}} \grmptb \gcn{1}{1}{1}{3} \gnl
\gcn{1}{1}{3}{1} \gcn{2}{1}{3}{1} \hspace{-0,22cm} \gcmu \gcmu \grcm \gnl
\gcn{1}{5}{2}{2} \gcn{2}{5}{2}{2} \gcl{3} \gcl{1} \gcl{1} \glmptb \gnot{\hspace{-0,34cm}\tau_{F,Z}} \grmptb \gcl{5} \gnl
\gvac{4} \gcl{1} \glmptb \gnot{\hspace{-0,34cm}\tau_{F,Z}} \grmptb \gcl{3} \gnl
\gvac{4} \glmptb \gnot{\hspace{-0,34cm}\tau_{F,Z}} \grmptb \gcl{1} \gnl
\gvac{3} \glmptb \gnot{\hspace{-0,34cm}\tau_{F,Z}} \grmptb \glmptb \gnot{\hspace{-0,34cm}\tau_{F,F}} \grmptb \gnl
\gvac{3} \gcn{1}{1}{1}{0} \gmu \glmpt \gnot{\hspace{-0,34cm}\sigma} \grmptb \gnl
\gvac{1} \hspace{-0,2cm} \glmptb \gnot{\alpha} \gcmptb \grmptb \gcn{3}{1}{3}{6} \gvac{1} \hspace{-0,34cm} \glmptb \gnot{\hspace{-0,34cm}\psi_{B,F}} \grmptb \gnl
\gvac{2} \gcn{1}{2}{0}{0} \gcn{1}{2}{0}{0} \gcn{3}{2}{0}{0} \glmpt \gnot{\hspace{-0,34cm}\sigma} \grmptb \gcl{1} \gnl
\gvac{8} \gmu \gnl
\gvac{2} \hspace{-0,24cm} \gob{1}{X} \gob{1}{Y} \gob{1}{Z} \gob{9}{B} 
\gend\stackrel{2\times com.d.l.}{=}
\gbeg{8}{11}
\got{1}{X} \gvac{1} \got{1}{Y} \got{3}{Z} \gnl
\grcm \grcm \grcm \gnl
\gcl{1} \glmptb \gnot{\hspace{-0,34cm}\tau_{F,Y}} \grmptb \glmptb \gnot{\hspace{-0,34cm}\tau_{F,Z}} \grmptb \gcn{1}{1}{1}{4} \gnl
\gcl{4} \gcl{4} \glmptb \gnot{\hspace{-0,34cm}\tau_{F,Z}} \grmptb \gcn{2}{1}{1}{3} \gcn{1}{4}{2}{2} \gnl
\gvac{2} \gcl{3} \hspace{-0,22cm} \gcmu \gcmu \gnl
\gvac{3} \gcn{1}{1}{1}{1} \glmptb \gnot{\hspace{-0,34cm}\tau_{F,F}} \grmptb \gcl{1} \gnl
\gvac{3} \gmu \glmpt \gnot{\hspace{-0,34cm}\sigma} \grmptb \gnl
\gvac{1} \hspace{-0,2cm} \glmptb \gnot{\alpha} \gcmptb \grmptb \gcn{3}{1}{1}{4} \hspace{-0,22cm} \glmptb \gnot{\hspace{-0,34cm}\psi_{B,F}} \grmptb \gnl
\gvac{2} \gcn{1}{2}{0}{0} \gcn{1}{2}{0}{0} \gcn{2}{2}{0}{0} \glmpt \gnot{\hspace{-0,34cm}\sigma} \grmptb \gcl{1}  \gnl
\gvac{7} \gmu \gnl
\gvac{2} \hspace{-0,24cm} \gob{1}{X} \gob{1}{Y} \gob{1}{Z} \gob{7}{B} 
\gend\stackrel{*}{\stackrel{\mu_M}{\stackrel{\equref{2-cocycle condition}}{=}}}
\gbeg{8}{11}
\got{1}{X} \gvac{1} \got{1}{Y} \got{3}{Z} \gnl
\grcm \grcm \grcm \gnl
\gcl{1} \glmptb \gnot{\hspace{-0,34cm}\tau_{F,Y}} \grmptb \glmptb \gnot{\hspace{-0,34cm}\tau_{F,Z}} \grmptb \gcn{1}{2}{1}{3} \gnl
\gcl{1} \gcl{1} \glmptb \gnot{\hspace{-0,34cm}\tau_{F,Z}} \grmptb \gcn{2}{1}{1}{1} \gnl
\glmptb \gnot{\alpha} \gcmptb \grmptb \glmptb \gnot{\hspace{0,7cm}\beta} \gcmpt \gcmp \gcmpt \grmp \gnl
\gcl{5} \gcl{5} \gcl{5} \gcl{2} \gcmu \gcmu \gnl
\gvac{4} \gcl{1} \glmptb \gnot{\hspace{-0,34cm}\tau_{F,F}} \grmptb \gcl{1} \gnl
\gvac{3} \gcn{1}{1}{1}{2} \gmu \glmpt \gnot{\hspace{-0,34cm}\sigma} \grmptb \gnl
\gvac{4} \hspace{-0,24cm} \glmpt \gnot{\hspace{-0,34cm}\sigma} \grmptb \gvac{2} \gcn{1}{1}{0}{-1}  \gnl
\gvac{5} \gwmu{3} \gnl
\gob{1}{\hspace{0,4cm}X} \gob{1}{\hspace{0,4cm}Y} \gob{1}{\hspace{0,4cm}Z} \gob{7}{B} 
\gend
$$

$$\stackrel{\equref{ab}}{\stackrel{com.d.l.}{\stackrel{comod.}{=}}}
\gbeg{9}{11}
\gvac{1} \got{1}{X} \got{1}{Y} \got{6}{Z} \gnl
\gvac{1} \gcl{3} \grcm \gcn{1}{1}{4}{4} \gnl
\gvac{2} \gcl{4} \gcn{3}{1}{1}{3} \hspace{-0,22cm} \grcm \gnl
\gvac{5} \hspace{-0,4cm} \gcmu \grcm \gcn{1}{1}{-1}{1} \gnl
\gvac{1} \gcn{1}{1}{4}{2} \gvac{3}  \gcl{1} \glmptb \gnot{\hspace{-0,34cm}\tau_{F,Z}} \grmptb \gcl{1} \gcl{2}\gnl
\gvac{2} \hspace{-0,22cm} \grcm \gvac{2} \hspace{-0,22cm} \glmptb \gnot{\hspace{-0,34cm}\tau_{F,Z}} \grmptb \glmptb \gnot{\hspace{-0,34cm}\tau_{F,F}} \grmptb  \gnl
\gvac{3} \hspace{-0,22cm} \gcl{4} \glmptb \gnot{\hspace{-0,34cm}\tau_{F,Y}} \grmptb \gvac{1} \hspace{-0,2cm} \gcn{1}{1}{1}{0} \gmu \glmpt \gnot{\hspace{-0,34cm}\sigma} \grmptb \gnl
\gvac{5} \hspace{-0,22cm} \gcl{3} \glmptb \gnot{\hspace{-0,34cm}\tau_{F,Z}} \grmptb \gcn{1}{1}{3}{1} \gcn{2}{2}{6}{3} \gnl
\gvac{6} \gcl{2} \glmpt \gnot{\hspace{-0,34cm}\sigma} \grmptb \gnl
\gvac{8} \gwmu{3} \gnl
\gvac{4} \gob{1}{X} \gob{1}{Y} \gob{1}{Z} \gob{5}{B} 
\gend\stackrel{\equref{nat rcm}}{=}
\gbeg{8}{12}
\got{1}{X} \gvac{1} \got{1}{Y} \got{5}{Z} \gnl
\gcl{4} \gvac{1} \grcm \gcn{1}{1}{3}{3} \gnl
\gvac{2} \gcl{4} \gcn{2}{1}{1}{3} \grcm \gnl
\gvac{4} \glmptb \gnot{\hspace{-0,34cm}\tau_{F,Z}} \grmptb \gcl{1} \gnl
\gvac{4} \gcl{1} \glmpt \gnot{\hspace{-0,34cm}\sigma} \grmptb \gnl
\grcm \grcm \grcm \gcl{5} \gnl
\gcl{5} \gcl{1} \gcl{1} \glmptb \gnot{\hspace{-0,34cm}\tau_{F,Z}} \grmptb \gcl{1} \gnl
\gvac{1} \glmptb \gnot{\hspace{-0,34cm}\tau_{F,Y}} \grmptb \gcl{1} \gmu \gnl
\gvac{1} \gcl{3} \glmptb \gnot{\hspace{-0,34cm}\tau_{F,Z}} \grmptb \gcn{1}{1}{2}{1}  \gnl
\gvac{2} \gcl{2} \glmpt \gnot{\hspace{-0,34cm}\sigma} \grmptb \gnl
\gvac{4} \gwmu{3} \gnl
\gob{1}{X} \gob{1}{Y} \gob{1}{Z} \gob{5}{B} 
\gend:=\Omega
$$
In the equation * $\alpha$ is the associativity constraint given by \equref{assoc. rcm}. 
Similarly as in the previous computation, the equality $\Sigma=\Omega$ expressing \equref{monad law ro} holds true 
if and only if the 2-cocycle condition \equref{2-cocycle condition} for $\sigma$ at the place * holds true. 

At last, for \equref{normalized in EM} we find: 
$$
\gbeg{3}{4}
\got{1}{} \gnl
\grcm \gu{1} \gnl
\gcl{1} \glmpt \gnot{\hspace{-0,34cm}\sigma} \grmptb \gnl
\gob{1}{} \gob{3}{B} 
\gend=
\gbeg{2}{4}
\got{1}{} \gnl
\gcl{1} \gu{1} \gnl
\gcl{1} \gcl{1} \gnl
\gob{1}{} \gob{1}{B}
\gend=
\gbeg{3}{5}
\got{3}{} \gnl
\gu{1} \grcm \gnl
\glmptb \gnot{\hspace{-0,34cm}\tau_{F, \s -}} \grmptb \gcl{1} \gnl
\gcl{1} \glmpt \gnot{\hspace{-0,34cm}\sigma} \grmptb \gnl
\gob{1}{} \gob{3}{B} 
\gend=
\gbeg{3}{6}
\got{3}{} \gnl
\grcm \gnl
\gcl{1} \gcn{1}{1}{1}{3} \gnl
\gcl{1} \gu{1} \gcl{1} \gnl
\gcl{1} \glmpt \gnot{\hspace{-0,34cm}\sigma} \grmptb \gnl
\gob{1}{} \gob{3}{B} 
\gend
$$
the unlabeled 1-cell is taken to be $X$ on the left hand-side, which presents $\crta\rho_{X, \Id}$ (observe that $\tau_{X, \Id}=\id_X$) and it is taken to be $Y$ on the right hand-side, 
which presents $\crta\rho_{\Id,X}$.  Setting $X=Y=F$ and and applying $\Epsilon_F$ we see that \equref{normalized in EM} is fulfilled if and only if \equref{normalized 2-cocycle} holds. 
\qed\end{proof}

For a consequence of \thref{main} and \prref{Sch case} we have:

\begin{cor} \colabel{Sch-Balan}
There is an action of categories ${}_{\tau_B}\Tau^{(\A, F, \omega)}\times {}_B\K\to {}_B\K$ given by $(X,M)\mapsto XM$ where $XM$ is a left $B$-module via \equref{B act XY psi pon} 
and the action associativity isomorphism is given by \equref{ro Sch} 
\begin{center}
\begin{tabular} {p{6cm}p{0cm}p{6.8cm}} 
\begin{equation} \eqlabel{B act XY psi pon}
\gbeg{3}{4}
\got{1}{B} \got{3}{XM} \gnl
\gcn{1}{1}{1}{3} \gvac{1} \gcl{1} \gnl
\gvac{1} \glm \gnl
\gvac{2} \gob{1}{XM} 
\gend=
\gbeg{3}{7}
\got{1}{B} \got{1}{X} \got{3}{M} \gnl
\gcl{1} \grcm \gcl{4} \gnl
\glmptb \gnot{\hspace{-0,34cm}\tau_{B,X}} \grmptb \gcl{1} \gnl
\gcl{1} \grm \gnl
\gcl{1} \gcn{1}{1}{1}{3} \gnl
\gcl{1} \gvac{1} \glm \gnl
\gob{1}{X} \gob{5}{M}
\gend
\end{equation} & &
\begin{equation} \eqlabel{ro Sch}
\rho_{X,Y,M}=
\gbeg{3}{6}
\got{1}{X} \gvac{1} \got{1}{Y} \got{3}{M} \gnl
\grcm \grcm \gcl{3} \gnl
\gcl{1} \glmptb \gnot{\hspace{-0,34cm}\tau_{F,Y}} \grmptb \gcl{1} \gnl
\gcl{2} \gcl{2} \glmpt \gnot{\hspace{-0,34cm}\sigma} \grmptb \gnl
\gvac{3} \glm \gnl 
\gob{1}{X} \gob{1}{Y} \gob{5}{M} 
\gend
\end{equation} 
\end{tabular} 
\end{center}
if and only if  $(B,F, \psi_{B,F}, \mu_M, \eta_M)$ is a Sweedler's Hopf datum where $\psi_{B,F}$ and $\mu_M$ are given as in \equref{Sw datum}. 
\end{cor}

\begin{ex}
Assume $\K$ is induced by the braided monoidal category of modules over a commutative ring $R$. Then $\tau$'s are the flip maps. \coref{Sch-Balan} 
then recovers the results from \cite[Section 4]{Sch} and \cite[Proposition 12]{Balan}. In the latter reference our Sweedler's Hopf datum was called {\em $H$-crossed system} for a coquasi-bialgebra $H$. 
\end{ex}

The following is directly proved and it explains the relation between Sweedler's 2-cocycle $\sigma$ in $K$ and a 2-cocycle in the sense of \deref{2-coc in K}.

\begin{lma} \lelabel{relation Sweedler cocycle}
If \equref{ro sigma} is a 2-cocycle in $\EM^M(\K)$ in the sense of \equref{monad law ro}, where $\rho_{FF,F}$ and $\rho_{F,FF}$ are given as below, then $\sigma$ satisfies \equref{2-cocycle condition}. 
\begin{center}
\begin{tabular} {p{7cm}p{0cm}p{6.8cm}} 
\begin{equation} \eqlabel{ro sigma}
\rho_{\sigma}=
\gbeg{5}{5}
\got{2}{F} \got{2}{F} \gnl
\gcmu \gcmu \gnl
\gcl{1} \glmptb \gnot{\hspace{-0,34cm}\tau_{F,F}} \grmptb \gcl{1} \gnl
\gcl{1} \gcl{1} \glmpt \gnot{\hspace{-0,34cm}\sigma} \grmptb \gnl
\gob{1}{F} \gob{1}{F} \gob{3}{B} \gnl
\gend=
\gbeg{4}{4}
\got{3}{FF} \gnl
\gwcm{3}  \gnl
\gcl{1} \gvac{1} \gbmp{\sigma} \gnl
\gob{1}{FF} \gob{3}{B} \gnl
\gend
\end{equation} & &
\begin{equation*}
\rho_{FF,F}=
\gbeg{4}{5}
\got{1}{F} \got{1}{F} \got{2}{F} \gnl
\gmu \gcn{1}{1}{2}{2} \gnl
\gvac{1} \hspace{-0,34cm} \glmptb \gnot{\rho_{\sigma}} \gcmpb \grmptb \gnl
\gvac{1} \hspace{-0,2cm} \gcmu \gcn{1}{1}{0}{0} \gcn{1}{1}{0}{0} \gnl
\gvac{1} \gob{1}{F} \gob{1}{F} \gob{1}{\hspace{-0,32cm}F} \gob{1}{\hspace{-0,32cm}B} \gnl
\gend, \qquad 
\rho_{F,FF}=
\gbeg{3}{5}
\got{1}{F} \got{2}{F} \got{1}{\hspace{-0,32cm}F} \gnl
\gcl{1}  \gvac{1} \hspace{-0,32cm} \gmu \gnl
\gvac{1} \hspace{-0,2cm} \glmptb \gnot{\rho_{\sigma}} \gcmpb \grmptb \gnl
\gvac{1} \gcl{1} \gcl{1} \hspace{-0,24cm} \gcmu \gnl
\gvac{1} \gob{1}{\hspace{0,32cm}F} \gob{1}{\hspace{0,32cm}F} \gob{1}{F} \gob{1}{B} \gnl
\gend
\end{equation*}
\end{tabular}
\end{center}
\end{lma}

\section{Representation category of quasi-bimonads in $\K$}  \selabel{rep quasi}

In this Section let $(F, \Phi)$ denote a quasi-bimonad in $\K$. In \thref{quasi-bim monoidal} we proved that the 
representation category of $(F, \Phi)$ is monoidal. 
We now assume that $B$ is a left $F$-comodule monad in the sense of \equref{F comod alg}--\equref{F comod alg unit} and 
we are going to consider the category ${}_{\tau_B (\A,F,\Phi)}\Tau$ that is analogous to the category 
${}_{\tau_B}\Tau^{(\A, F, \omega)}$ we introduced before \leref{psi_2}. We describe it here.

The objects of ${}_{\tau_B (\A,F,\Phi)}\Tau$ are triples $(X, \tau_{B,X}, \tau_{F,X})$ - including when $X=F$ - where 
$(X, \tau_{F,X})\in {}_{(\A,F,\Phi)}\Tau$ (from \deref{module cat over F Phi}), 
$(X, \tau_{B,X})\in\Tau(\A, B)$, 
$(X, \nu)$ is a left $F$-module so that 
$\nu: (FX, \tau_{B, FX})\to (X, \tau_{B, X})$ is a morphism in $\Tau(\A,B)$, 
$\tau_{B,X}$ is natural with respect to the left $F$-coaction on $B$ (see \equref{nat lcm}),
and the Yang-Baxter equation \equref{YBE BFX} for $(\tau_{B,F}, \tau_{B,X}, \tau_{F,X})$ holds true. 
In particular, for $X=F$ we have 
that $\tau_{B,F}$ is monadic with respect to $F$ and obeys \equref{nat lcm} for $X=F$. 
We also require that \equref{Phi B nat new} holds, saying that $\tau_{B,\Id_{\A}}$ is natural with respect to $\Phi$. 
\begin{center} 
\begin{tabular}{p{4.5cm}p{0cm}p{4.3cm}p{0cm}p{5.3cm}}
\begin{equation} \eqlabel{nat lm BFX}
\gbeg{3}{5}
\got{1}{B} \got{1}{F} \got{1}{X} \gnl
\gcl{1} \glm \gnl
\glmptb \gnot{\tau_{B,X}} \gcmp \grmptb \gnl
\gcl{1} \gvac{1} \gcl{1} \gnl
\gob{1}{X} \gob{3}{B}
\gend=
\gbeg{3}{5}
\got{1}{B} \got{1}{F} \got{1}{X} \gnl
\glmptb \gnot{\hspace{-0,34cm}\tau_{B,F}} \grmptb \gcl{1} \gnl
\gcl{1} \glmptb \gnot{\hspace{-0,34cm}\tau_{B,X}} \grmptb \gnl
\glm \gcl{1} \gnl
\gvac{1} \gob{1}{X} \gob{1}{B}
\gend
\end{equation} & &
\begin{equation} \eqlabel{nat lcm}
\gbeg{3}{5}
\got{1}{B} \got{3}{X} \gnl
\gcl{1} \gvac{1} \gcl{1} \gnl
\glmptb \gnot{\tau_{B,X}} \gcmp \grmptb \gnl
\gcl{1} \glcm \gnl
\gob{1}{X} \gob{1}{F} \gob{1}{B} 
\gend=
\gbeg{3}{5}
\got{1}{} \got{1}{B} \got{1}{X} \gnl
\glcm \gcl{1}\gnl 
\gcl{1} \glmptb \gnot{\hspace{-0,34cm}\tau_{B,X}} \grmptb \gnl
\glmptb \gnot{\hspace{-0,34cm}\tau_{F,X}} \grmptb \gcl{1} \gnl
\gob{1}{X} \gob{1}{F} \gob{1}{B} 
\gend
\end{equation} & &
\begin{equation} \eqlabel{Phi B nat new} \hspace{-2cm}
\gbeg{4}{5}
\got{7}{B} \gnl
\gcn{1}{1}{2}{1} \gelt{\s\Phi} \gcn{1}{1}{0}{1} \gcl{3} \gnl  %
\gcl{2} \gcl{2} \gcl{2} \gnl
\gob{1}{F} \gob{1}{F} \gob{1}{F} \gob{1}{B} 
\gend=
\gbeg{2}{6}
\got{1}{B} \gnl
\gcl{1} \gcn{1}{1}{2}{1} \gelt{\s\Phi} \gcn{1}{1}{0}{1} \gnl 
\glmptb \gnot{\hspace{-0,34cm}\tau_{B,F}} \grmptb \gcl{1} \gcl{1} \gnl
\gcl{1} \glmptb \gnot{\hspace{-0,34cm}\tau_{B,F}} \grmptb \gcl{1} \gnl
\gcl{1} \gcl{1} \glmptb \gnot{\hspace{-0,34cm}\tau_{B,F}} \grmptb \gnl
\gob{1}{F} \gob{1}{F} \gob{1}{F} \gob{1}{B} 
\gend
\end{equation} 
\end{tabular}
\end{center} 
The morphisms of ${}_{\tau_B (\A,F,\Phi)}\Tau$ are $\zeta: (X, \tau_{B,X}, \tau_{F,X})\to (Y, \tau_{B,Y}, \tau_{F,Y})$ given by left $F$-linear 
morphisms $\zeta: (X, \tau_{B,X})\to (Y, \tau_{B,Y})$ in $\Mnd(\K)(B)$ and $\zeta: (X, \tau_{F,X})\to (Y, \tau_{F,Y})$ in $\Mnd(\K)(F)$ 
(two last identities in \equref{morphisms in comod}). Analogously as before, we have that 
the category ${}_{\tau_B (\A,F,\Phi)}\Tau$ is monoidal.

\medskip

\begin{lma} \lelabel{psi_3}
Let $(F, \Phi)$ be a quasi-bimonad, $B:\A\to\A$ be a left $F$-comodule monad with $\psi_{B,F}$ given by \equref{psi_3 for BF}, 
$(X, \nu)$ a left $F$-module so that $\nu: (FX, \tau_{B, FX})\to (X, \tau_{B, X})$ is a morphism in $\Tau(\A,B)$.  
Then $(X, \psi_{B,X})\in\Tau(\A, B)$ with $\psi_{B,X}$ being given via \equref{psi_3 for X}:  
\begin{center} 
\begin{tabular}{p{5cm}p{1cm}p{5cm}}
\begin{equation}\eqlabel{psi_3 for BF}
\psi_{B,F}= 
\gbeg{3}{5}
\gvac{1} \got{1}{B} \got{1}{F} \gnl
\glcm \gcl{1} \gnl
\gcl{1} \glmptb \gnot{\hspace{-0,34cm}\tau_{B,F}} \grmptb \gnl
\gmu \gcl{1} \gnl
\gob{2}{F} \gob{1}{B} 
\gend
\end{equation} & &
\begin{equation}\eqlabel{psi_3 for X}
\psi_{B,X}=
\gbeg{3}{5}
\gvac{1} \got{1}{B} \got{1}{X} \gnl
\glcm \gcl{1} \gnl
\gcl{1} \glmptb \gnot{\hspace{-0,34cm}\tau_{B,X}} \grmptb \gnl
\glm \gcl{1} \gnl
\gvac{1} \gob{1}{X} \gob{1}{B.} 
\gend
\end{equation}
\end{tabular}
\end{center}
\end{lma}

\begin{proof}
We only prove the first distributive law for $\psi_{B,X}$:
$$
\gbeg{3}{5}
\got{1}{B}\got{1}{B}\got{1}{X}\gnl
\gmu \gcn{1}{1}{1}{0} \gnl
\gvac{1} \hspace{-0,34cm} \glmptb \gnot{\hspace{-0,34cm}\psi} \grmptb  \gnl
\gvac{1} \gcl{1} \gcl{1} \gnl
\gvac{1} \gob{1}{X} \gob{1}{B}
\gend=
\gbeg{3}{6}
\got{1}{B} \got{1}{B} \got{1}{X} \gnl
\gmu \gcl{1} \gnl
\hspace{-0,32cm} \glcm  \gcn{1}{1}{2}{1} \gnl
\gcl{1} \glmptb \gnot{\hspace{-0,34cm}\tau_{B,X}} \grmptb \gnl
\glm \gcl{1} \gnl
\gob{1}{} \gob{1}{X} \gob{1}{B} 
\gend\stackrel{\equref{F comod alg}}{=}
\gbeg{5}{7}
\gvac{1} \got{1}{B} \gvac{1} \got{1}{B} \got{1}{X} \gnl
\glcm \glcm \gcl{2} \gnl
\gcl{1} \glmptb \gnot{\hspace{-0,34cm}\tau_{B,F}} \grmptb \gcl{1} \gnl
\gmu \gmu \gcn{1}{1}{1}{0} \gnl
\gcn{1}{1}{2}{4} \gvac{2} \hspace{-0,34cm} \glmptb \gnot{\hspace{-0,34cm}\tau_{B,X}} \grmptb \gnl
\gvac{2} \glm \gcl{1} \gnl
\gvac{3} \gob{1}{X} \gob{1}{B} 
\gend\stackrel{mod.}{\stackrel{d.l.}{=}}
\gbeg{5}{7}
\gvac{1} \got{1}{B} \gvac{1} \got{1}{B} \got{1}{X} \gnl
\glcm \glcm \gcl{1} \gnl
\gcl{1} \glmptb \gnot{\hspace{-0,34cm}\tau_{B,F}} \grmptb \glmptb \gnot{\hspace{-0,34cm}\tau_{B,X}} \grmptb \gnl
\gcl{1} \gcl{1} \glmptb \gnot{\hspace{-0,34cm}\tau_{B,X}} \grmptb \gcl{1} \gnl
\gcn{1}{1}{1}{3} \glm \gmu \gnl
\gvac{1} \glm \gcn{1}{1}{2}{2} \gnl
\gvac{2} \gob{1}{X} \gob{2}{B} 
\gend\stackrel{\equref{nat lm BFX}}{=}
\gbeg{5}{8}
\gvac{1} \got{1}{B} \gvac{1} \got{1}{B} \got{1}{X} \gnl
\gvac{1} \gcl{3}  \glcm \gcl{1} \gnl
\gvac{2} \gcl{1} \glmptb \gnot{\hspace{-0,34cm}\tau_{B,X}} \grmptb  \gnl
\gvac{2} \glm \gcl{3} \gnl
\glcm \gcn{1}{1}{3}{1} \gnl
\gcl{1} \glmptb \gnot{\hspace{-0,34cm}\tau_{B,X}} \grmptb  \gnl
\glm \gwmu{3} \gnl
\gvac{1} \gob{1}{X} \gob{3}{B} 
\gend=
\gbeg{3}{5}
\got{1}{B}\got{1}{B}\got{1}{X}\gnl
\gcl{1} \glmpt \gnot{\hspace{-0,34cm}\psi} \grmptb \gnl
\glmptb \gnot{\hspace{-0,34cm}\psi} \grmptb \gcl{1} \gnl
\gcl{1} \gmu \gnl
\gob{1}{X} \gob{2}{B}
\gend
$$
\qed\end{proof}

A quasi-bimonad in $\K$ is a proper monad. Given a quasi-bimonad $(F, \Phi)$ and a monad $B:\A\to\A$ in $\K$. Then $FFB$ is a monad, $\Id_{\A}$ is trivially a comonad 
and we can consider the monad of the 2-cells $\Id_{\A}\to FFB$ in $\K$, which is indeed a convolution algebra in the monoidal category $\K(\A)$.

\begin{prop} \prlabel{Martin case}
Let $(F, \Phi)$ be a quasi-bimonad and $B:\A\to\A$ a left $F$-comodule monad in $\K$ with $\psi_{B,F}$ given by \equref{psi_3 for BF}. 
The following are equivalent: 
\begin{enumerate}
\item 
set $\C={}_{\tau_B(\A, F, \Phi)}\Tau$ and let $\psi_{B,X}$ 
be given by \equref{psi_3 for X}; given a 2-cell $\Phi_{\lambda}: \Id_{\A}\to FFB$, 
the 2-cell $\crta\rho_{X,Y}$ given by \equref{crta ro Martin} 
satisfies the conditions in \thref{main}; 
\begin{equation} \eqlabel{crta ro Martin}
\crta\rho_{X,Y}=
\gbeg{7}{6}
\gvac{3} \got{1}{X} \got{1}{Y} \gnl
\gcn{1}{1}{2}{1} \gelt{\Phi_{\lambda}} \gcn{1}{1}{0}{1} \gcl{1} \gcl{2} \gnl 
\gcl{1} \gcl{1} \glmptb \gnot{\hspace{-0,34cm}\tau_{B,X}} \grmptb \gnl
\gcl{1} \glmptb \gnot{\hspace{-0,34cm}\tau_{F,X}} \grmptb \glmptb \gnot{\hspace{-0,34cm}\tau_{B,Y}} \grmptb \gnl
\glm \glm \gcl{1} \gnl
\gvac{1} \gob{1}{X} \gob{1}{} \gob{1}{Y} \gob{1}{B} 
\gend
\end{equation}
\item $(B,F, \psi_{B,F}, \Delta_M, \Epsilon_M, \eta_F, \beta)$ 
is a Hausser-Nill datum where $\Delta_M$ and $\beta$ are given by: 
\begin{equation} \eqlabel{HN datum}
\Delta_M=
\gbeg{5}{7}
\gvac{1} \got{5}{F} \gnl
\gcn{1}{2}{2}{0} \gelt{\s\Phi_{\lambda}} \gcn{1}{1}{0}{1} \gcl{1} \gnl
\gvac{1} \gcn{1}{1}{1}{0} \glmptb \gnot{\hspace{-0,34cm}\tau_{B,F}} \grmptb \gnl
\gcn{1}{2}{0}{0} \gcn{1}{1}{0}{0} \hspace{-0,22cm} \gcmu \gcn{1}{1}{0}{1} \gnl
\gvac{1} \glmptb \gnot{\hspace{-0,34cm}\tau_{F,F}} \grmptb \gcl{1}  \gcl{2} \gnl
\gmu \gmu \gnl
\gob{2}{F} \gob{2}{F} \gob{1}{B} \gnl
\gend
\qquad\text{and}\qquad
\beta=
\gbeg{7}{6}
\got{1}{F} \got{1}{F} \got{1}{F} \gnl
\gcl{3} \gcl{2} \gcl{1} \gcn{1}{1}{2}{1} \gelt{\s\Phi^{-1}} \gcn{1}{1}{0}{1} \gnl 
\gvac{2} \glmptb \gnot{\hspace{-0,34cm}\tau_{F,F}} \grmptb \gcl{1} \gcl{1} \gnl
\gcl{1} \glmptb \gnot{\hspace{-0,34cm}\tau_{F,F}} \grmptb \glmptb \gnot{\hspace{-0,34cm}\tau_{F,F}} \grmptb \gcl{1} \gnl
\gmu \gmu \gmu \gnl
\gob{2}{F} \gob{2}{F} \gob{2}{F} 
\gend 
\end{equation}
where the Hausser-Nill 2-cocycle $\Phi_{\lambda}$ in $\K$ is invertible in the convolution algebra $\K(\A)(\Id_{\A},FFB)$. 
Here $\Epsilon_M=\eta_B\comp\Epsilon_F$. 
\end{enumerate}
\end{prop}

\begin{proof}
A faithful functor is provided by the forgetful functor $\F:{}_{\tau_B(\A, F, \Phi)}\Tau\to \Tau(\A,B)$ which is given by $\F(X, \tau_{B,X}, \tau_{F,X}, \nu)=(X, \tau_{B,X})$, it is 
obviously quasi-monoidal. 
The 2-cell $\crta\rho_{X,Y}$ from \equref{crta ro Martin} clearly satisfies the condition \equref{natur}  (recall module version of \equref{morphisms in comod}). 
For the condition \equref{vert comp M} with 
$\crta\rho_{X,Y}^{-1}=
\gbeg{5}{6}
\gvac{3} \got{1}{X} \got{1}{Y} \gnl
\gcn{1}{1}{2}{1} \gelt{\chi} \gcn{1}{1}{0}{1} \gcl{1} \gcl{2} \gnl 
\gcl{1} \gcl{1} \glmptb \gnot{\hspace{-0,34cm}\tau_{F,X}} \grmptb \gnl
\gcl{1} \glmptb \gnot{\hspace{-0,34cm}\tau_{F,X}} \grmptb \glmptb \gnot{\hspace{-0,34cm}\tau_{F,Y}} \grmptb \gnl
\glm \glm \gcl{1} \gnl
\gvac{1} \gob{1}{X} \gob{1}{} \gob{1}{Y} \gob{1}{B} 
\gend
$ we find: 
$$
\gbeg{8}{10}
\gvac{5} \got{1}{X} \got{1}{Y} \gnl
\gvac{2} \gcn{1}{1}{2}{1} \gelt{\Phi_{\lambda}} \gcn{1}{1}{0}{1} \gcl{1} \gcl{2} \gnl 
\gvac{2} \gcl{1} \gcl{1} \glmptb \gnot{\hspace{-0,34cm}\tau_{B,X}} \grmptb \gnl
\gvac{2} \gcl{1} \glmptb \gnot{\hspace{-0,34cm}\tau_{F,X}} \grmptb \glmptb \gnot{\hspace{-0,34cm}\tau_{B,Y}} \grmptb \gnl
\gvac{2} \glm \glm \gcl{4} \gnl
\gcn{1}{1}{2}{1} \gelt{\chi} \gcn{1}{1}{0}{1} \gcl{1} \gvac{1} \gcl{1} \gnl 
\gcl{1} \gcl{1} \glmptb \gnot{\hspace{-0,34cm}\tau_{B,X}} \grmptb \gcn{1}{1}{3}{1} \gnl
\gcl{1} \glmptb \gnot{\hspace{-0,34cm}\tau_{F,X}} \grmptb \glmptb \gnot{\hspace{-0,34cm}\tau_{B,Y}} \grmptb \gnl
\glm \glm \gwmu{3} \gnl
 \gvac{1} \gob{1}{X} \gob{1}{} \gob{1}{Y} \gob{3}{B} 
\gend\stackrel{3\times \equref{nat lm BFX}}{\stackrel{mod.}{=}}
\gbeg{8}{9}
\gvac{6} \got{1}{X} \got{1}{Y} \gnl
\gcn{1}{1}{2}{1} \gelt{\chi} \gcn{1}{1}{0}{1} \gcn{1}{1}{2}{1} \gelt{\Phi_{\lambda}} \gcn{1}{1}{0}{1} \gcl{1} \gcl{2} \gnl 
\gcl{2} \gcl{1} \glmptb \gnot{\hspace{-0,34cm}\tau_{B,F}} \grmptb \gcl{1} \glmptb \gnot{\hspace{-0,34cm}\tau_{B,X}} \grmptb \gnl
\gvac{1} \glmptb \gnot{\hspace{-0,34cm}\tau_{F,F}} \grmptb \glmptb \gnot{\hspace{-0,34cm}\tau_{B,F}} \grmptb \gcl{1} \glmptb \gnot{\hspace{-0,34cm}\tau_{B,Y}} \grmptb \gnl
\gmu \gmu \glmptb \gnot{\hspace{-0,34cm}\tau_{B,X}} \grmptb \gcl{1} \gcl{2} \gnl
\gcn{1}{1}{2}{2} \gvac{1} \gcn{1}{1}{2}{3} \gvac{1} \gcl{1} \glmptb \gnot{\hspace{-0,34cm}\tau_{B,Y}} \grmptb \gnl 
\gcn{2}{1}{2}{5} \gvac{1} \glmptb \gnot{\hspace{-0,34cm}\tau_{F,X}} \grmptb \gcl{1} \gmu \gnl
\gvac{2} \glm \glm \gcn{1}{1}{2}{2} \gnl
 \gvac{3} \gob{1}{X} \gob{1}{} \gob{1}{Y} \gob{2}{B} 
\gend\stackrel{2\times d.l.}{=}
\gbeg{8}{9}
\gvac{6} \got{1}{X} \got{1}{Y} \gnl
\gcn{1}{1}{2}{1} \gelt{\chi} \gcn{1}{1}{0}{1} \gcn{1}{1}{2}{1} \gelt{\Phi_{\lambda}} \gcn{1}{1}{0}{1} \gcl{3} \gcl{4} \gnl 
\gcl{2} \gcl{1} \glmptb \gnot{\hspace{-0,34cm}\tau_{B,F}} \grmptb \gcl{1} \gcl{2} \gnl
\gvac{1} \glmptb \gnot{\hspace{-0,34cm}\tau_{F,F}} \grmptb \glmptb \gnot{\hspace{-0,34cm}\tau_{B,F}} \grmptb \gnl
\gmu \gmu \gmu \gcn{1}{1}{1}{0} \gnl
\gcn{1}{1}{2}{2} \gvac{1} \gcn{1}{1}{2}{4} \gvac{2} \hspace{-0,32cm} \glmptb \gnot{\hspace{-0,34cm}\tau_{B,X}} \grmptb \gcn{1}{1}{2}{1} \gnl 
\gcn{2}{1}{3}{7} \gvac{2} \glmptb \gnot{\hspace{-0,34cm}\tau_{F,X}} \grmptb \glmptb \gnot{\hspace{-0,34cm}\tau_{B,Y}} \grmptb \gnl
\gvac{3} \glm \glm \gcl{1} \gnl
 \gvac{4} \gob{1}{X} \gob{1}{} \gob{1}{Y} \gob{1}{B} 
\gend=
\gbeg{3}{4}
\got{1}{X}  \got{1}{Y} \gnl
\gcl{2} \gcl{2} \gu{1} \gnl
\gvac{2} \gcl{1} \gnl
\gob{1}{X} \gob{1}{Y} \gob{1}{B}
\gend
$$
This together with the analogous computation with the reversed order of $\chi$ and $\Phi_{\lambda}$ means that $\chi$ is the inverse of the latter in the convolution algebra 
$\K(\A)(\Id_{\A},FFB)$ (take $X=Y=F$ and compose with two coppies of the unit of $F$). 

\medskip

To check \equref{2-cells EM^M one psi} we should prove: 
$$
\gbeg{7}{11}
\gvac{4} \got{1}{B} \got{1}{X} \got{1}{Y} \gnl
\gvac{3} \glcm \gcl{1} \gcl{2} \gnl 
\gvac{2} \gcn{2}{1}{3}{2} \glmptb \gnot{\hspace{-0,34cm}\tau_{B,X}} \grmptb \gnl
\gvac{2} \gcmu \gcl{1} \glmptb \gnot{\hspace{-0,34cm}\tau_{B,Y}} \grmptb \gnl
\gvac{2} \gcl{1} \glmptb \gnot{\hspace{-0,34cm}\tau_{F,X}} \grmptb \gcl{1} \gcl{5} \gnl
\gvac{2} \glm \glm \gnl
\gcn{1}{1}{2}{1} \gelt{\Phi_{\lambda}} \gcn{1}{1}{0}{1} \gcl{1} \gvac{1} \gcl{1} \gnl 
\gcl{1} \gcl{1} \glmptb \gnot{\hspace{-0,34cm}\tau_{B,X}} \grmptb \gcn{1}{1}{3}{1} \gnl
\gcl{1} \glmptb \gnot{\hspace{-0,34cm}\tau_{F,X}} \grmptb \glmptb \gnot{\hspace{-0,34cm}\tau_{B,Y}} \grmptb \gnl
\glm \glm \gwmu{3} \gnl
 \gvac{1} \gob{1}{X} \gob{1}{} \gob{1}{Y} \gob{3}{B} 
\gend=
\gbeg{7}{10}
\gvac{2} \got{1}{B} \gvac{3} \got{1}{X} \got{1}{Y} \gnl
\gvac{2} \gcl{3} \gcn{1}{1}{2}{1} \gelt{\Phi_{\lambda}} \gcn{1}{1}{0}{1} \gcl{1} \gcl{2} \gnl 
\gvac{3} \gcl{2} \gcl{1} \glmptb \gnot{\hspace{-0,34cm}\tau_{B,X}} \grmptb \gnl
\gvac{4} \glmptb \gnot{\hspace{-0,34cm}\tau_{F,X}} \grmptb \glmptb \gnot{\hspace{-0,34cm}\tau_{B,Y}} \grmptb  \gnl
\gvac{1} \glcm \glm \glm \gcl{2} \gnl
\gcn{2}{1}{3}{2} \gcl{1} \gcn{1}{1}{3}{1} \gcn{2}{2}{5}{1} \gnl 
\gcmu \glmptb \gnot{\hspace{-0,34cm}\tau_{B,X}} \grmptb \gcn{1}{2}{7}{5}\gnl
\gcl{1} \glmptb \gnot{\hspace{-0,34cm}\tau_{F,X}} \grmptb \glmptb \gnot{\hspace{-0,34cm}\tau_{B,Y}} \grmptb  \gnl
\glm \glm \gwmu{3} \gnl
 \gvac{1} \gob{1}{X} \gob{1}{} \gob{1}{Y} \gob{3}{B} 
\gend
$$
By exactly the same arguments as in the previous computation  this is equivalent to 
$$
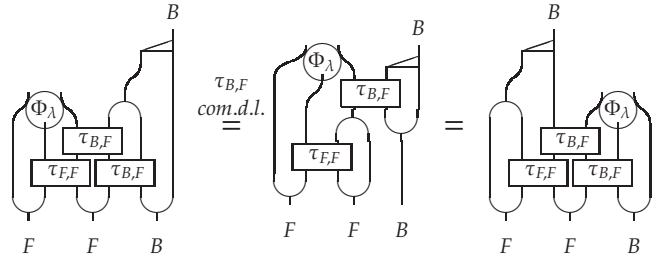
\gbeg{6}{8}
\gvac{5} \got{1}{B} \gnl
\gvac{4} \glcm \gnl
\gvac{3} \gcn{2}{1}{3}{2} \gcl{2} \gnl
\gcn{1}{1}{2}{1} \gelt{\Phi_{\lambda}} \gcn{1}{1}{0}{1} \gcmu \gnl 
\gcl{2} \gcl{1} \glmptb \gnot{\hspace{-0,34cm}\tau_{B,F}} \grmptb \gcl{1} \gcl{2} \gnl
\gvac{1} \glmptb \gnot{\hspace{-0,34cm}\tau_{F,F}} \grmptb \glmptb \gnot{\hspace{-0,34cm}\tau_{B,F}} \grmptb \gnl
\gmu \gmu \gmu \gnl
\gob{2}{F} \gob{2}{F} \gob{2}{B} 
\gend\stackrel{\tau_{B,F}}{\stackrel{com.d.l.}{=}}
\gbeg{5}{7}
\gvac{4} \got{1}{B} \gnl
\gcn{1}{1}{2}{0} \gelt{\Phi_{\lambda}} \gcn{1}{1}{0}{1} \glcm \gnl
\gcn{1}{3}{0}{0} \gcn{1}{1}{1}{0} \glmptb \gnot{\hspace{-0,34cm}\tau_{B,F}} \grmptb \gcl{1} \gnl
\gvac{1} \gcn{1}{1}{0}{0} \hspace{-0,24cm} \gcmu \hspace{-0,2cm} \gmu \gnl 
\gvac{2} \hspace{-0,2cm} \glmptb \gnot{\hspace{-0,34cm}\tau_{F,F}} \grmptb \gcl{1} \gcl{2} \gnl
\gvac{1} \gmu \gmu \gnl
\gvac{1} \gob{2}{F} \gob{2}{F} \gob{1}{B} 
\gend=
\gbeg{6}{8}
\gvac{2} \got{1}{B} \gnl
\gvac{1} \glcm \gnl
\gcn{2}{1}{3}{2} \gcl{2} \gnl
\gcmu \gvac{1} \gcn{1}{1}{2}{1} \gelt{\Phi_{\lambda}} \gcn{1}{1}{0}{1} \gnl 
\gcl{2} \gcl{1} \glmptb \gnot{\hspace{-0,34cm}\tau_{B,F}} \grmptb \gcl{1} \gcl{2} \gnl
\gvac{1} \glmptb \gnot{\hspace{-0,34cm}\tau_{F,F}} \grmptb \glmptb \gnot{\hspace{-0,34cm}\tau_{B,F}} \grmptb \gnl
\gmu \gmu \gmu \gnl
\gob{2}{F} \gob{2}{F} \gob{2}{B} 
\gend
$$
This is precisely \equref{quasi coaction} with $\Delta_M$ being as in \equref{HN datum}.

We now investigate when \equref{monad law ro new} is fulfilled. We find that this identity becomes: 
$$
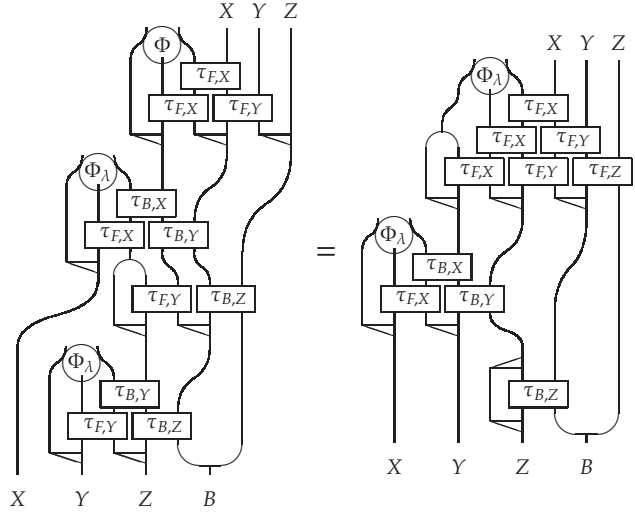
\gbeg{9}{16}
\gvac{6} \got{1}{X} \got{1}{Y} \got{1}{Z} \gnl
\gvac{3} \gcn{1}{1}{2}{1} \gelt{\Phi} \gcn{1}{1}{0}{1} \gcl{1} \gcl{2}  \gcl{3} \gnl 
\gvac{3} \gcl{1} \gcl{1} \glmptb \gnot{\hspace{-0,34cm}\tau_{F,X}} \grmptb \gnl
\gvac{3} \gcl{1} \glmptb \gnot{\hspace{-0,34cm}\tau_{F,X}} \grmptb \glmptb \gnot{\hspace{-0,34cm}\tau_{F,Y}} \grmptb \gnl
\gvac{3} \glm \glm \glm \gnl
\gvac{1} \gcn{1}{1}{2}{1} \gelt{\Phi_{\lambda}} \gcn{1}{1}{0}{1} \gcl{1} \gcn{1}{2}{3}{1} \gcn{1}{3}{5}{2} \gnl 
\gvac{1} \gcl{1} \gcl{1} \glmptb \gnot{\hspace{-0,34cm}\tau_{B,X}} \grmptb \gnl
\gvac{1} \gcl{1} \glmptb \gnot{\hspace{-0,34cm}\tau_{F,X}} \grmptb \glmptb \gnot{\hspace{-0,34cm}\tau_{B,Y}} \grmptb \gnl
\gvac{1} \glm \hspace{-0,22cm}\gcmu \gcn{1}{1}{0}{1} \gcn{1}{1}{0}{1} \gcl{1} \gnl
\gvac{1} \gcn{2}{2}{4}{-1} \gcl{1} \glmptb \gnot{\hspace{-0,34cm}\tau_{F,Y}} \grmptb  \glmptb \gnot{\hspace{-0,34cm}\tau_{B,Z}} \grmptb \gnl
\gvac{3} \glm \glm \gcl{4} \gnl
\gcl{4} \gcn{1}{1}{2}{1} \gelt{\Phi_{\lambda}} \gcn{1}{1}{0}{1} \gcl{1} \gvac{1} \gcn{1}{2}{1}{-1} \gnl
\gvac{1} \gcl{2} \gcl{1} \glmptb \gnot{\hspace{-0,34cm}\tau_{B,Y}} \grmptb \gnl
\gvac{2} \glmptb \gnot{\hspace{-0,34cm}\tau_{F,Y}} \grmptb \glmptb \gnot{\hspace{-0,34cm}\tau_{B,Z}} \grmptb \gnl
\gvac{1} \glm \glm \gwmu{3} \gnl
\gob{1}{X} \gob{1}{} \gob{1}{Y} \gob{3}{Z} \gob{1}{B} 
\gend=
\gbeg{9}{14}
\gvac{6} \got{1}{X} \got{1}{Y} \got{1}{Z} \gnl
\gvac{3} \gcn{1}{1}{2}{1} \gelt{\Phi_{\lambda}} \gcn{1}{1}{0}{1} \gcl{1} \gcl{2}  \gcl{3} \gnl 
\gvac{3} \gcn{1}{1}{1}{0} \gcl{1} \glmptb \gnot{\hspace{-0,34cm}\tau_{F,X}} \grmptb \gnl
\gvac{2} \gcmu \glmptb \gnot{\hspace{-0,34cm}\tau_{F,X}} \grmptb \glmptb \gnot{\hspace{-0,34cm}\tau_{F,Y}} \grmptb \gnl
\gvac{2} \gcl{1} \glmptb \gnot{\hspace{-0,34cm}\tau_{F,X}} \grmptb \glmptb \gnot{\hspace{-0,34cm}\tau_{F,Y}} \grmptb \glmptb \gnot{\hspace{-0,34cm}\tau_{F,Z}} \grmptb \gnl
\gvac{2} \glm \glm \glm \gcl{7} \gnl
\gcn{1}{1}{2}{1} \gelt{\Phi_{\lambda}} \gcn{1}{1}{0}{1} \gcl{1} \gcn{1}{2}{3}{1} \gcn{1}{3}{5}{3} \gnl 
\gcl{1} \gcl{1} \glmptb \gnot{\hspace{-0,34cm}\tau_{B,X}} \grmptb \gnl
\gcl{1} \glmptb \gnot{\hspace{-0,34cm}\tau_{F,X}} \grmptb \glmptb \gnot{\hspace{-0,34cm}\tau_{B,Y}} \grmptb \gnl
\glm \glm \gvac{1} \gcn{1}{1}{-1}{1} \gcl{2} \gnl
\gvac{1} \gcl{3} \gvac{1} \gcl{3} \glcm  \gnl
\gvac{4} \gcl{1} \glmptb \gnot{\hspace{-0,34cm}\tau_{B,Z}} \grmptb \gnl
\gvac{4} \glm \gwmu{3} \gnl
\gvac{1} \gob{1}{X} \gob{1}{} \gob{1}{Y} \gob{3}{Z} \gob{1}{B} 
\gend
$$
Applying the comonadic distributive law for $\tau_{F,X}$, monadic for $\tau_{B,X}$, naturality of both with respect to the left $F$-module action, 
naturality of $\tau_{B,X}$ with respect to the left $F$-comodule action and associativity of $F$, we see that this is equivalent to: 
$$
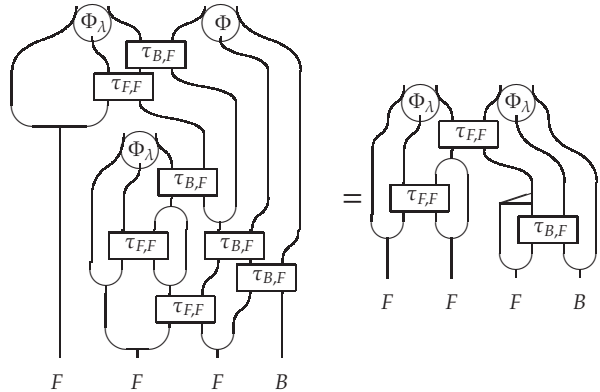
\gbeg{10}{12}
\gvac{1} \gcn{1}{2}{2}{-2} \gelt{\s\Phi_{\lambda}} \gcn{1}{1}{0}{2} \gvac{1} \gcn{1}{1}{2}{0} \gelt{\s\Phi} \gcn{2}{2}{0}{4}  \gnl
\gvac{2} \gcn{2}{1}{1}{2} \hspace{-0,22cm} \glmptb \gnot{\hspace{-0,34cm}\tau_{B,F}} \grmptb \gcn{1}{1}{2}{5} \gnl
\gcl{1} \gvac{2} \glmptb \gnot{\hspace{-0,34cm}\tau_{F,F}} \grmptb \gcn{1}{1}{1}{5} \gvac{2} \gcl{4} \gcl{5} \gnl
\gwmu{4} \gcn{1}{1}{1}{5} \gvac{2} \gcl{1} \gnl
\gvac{2} \gcn{1}{7}{0}{0} \gcn{1}{2}{2}{0} \gelt{\s\Phi_{\lambda}} \gcn{1}{1}{0}{1} \gcl{1} \gcl{1} \gnl
\gvac{4} \gcn{1}{1}{1}{0} \glmptb \gnot{\hspace{-0,34cm}\tau_{B,F}} \grmptb \gcl{1} \gnl
\gvac{2} \gcn{1}{3}{0}{0} \gcn{1}{2}{0}{0} \gcn{1}{1}{0}{0} \hspace{-0,22cm} \gcmu \hspace{-0,2cm} \gmu \gcn{1}{1}{1}{0} \gnl
\gvac{5} \hspace{-0,34cm} \glmptb \gnot{\hspace{-0,34cm}\tau_{F,F}} \grmptb \gcl{1} \glmptb \gnot{\hspace{-0,34cm}\tau_{B,F}} \grmptb \gcn{1}{1}{2}{1} \gnl
\gvac{4} \gmu \gmu \gcn{1}{1}{1}{0}  \glmptb \gnot{\hspace{-0,34cm}\tau_{B,F}} \grmptb\gnl
\gvac{5} \hspace{-0,22cm} \gcl{1} \gvac{1} \glmptb \gnot{\hspace{-0,34cm}\tau_{F,F}} \grmptb \gcn{1}{1}{2}{1} \gcn{1}{2}{2}{2} \gnl
\gvac{5} \gwmu{3} \gmu \gnl
\gvac{3} \gob{2}{F} \gob{3}{F} \gob{2}{F} \gob{2}{B} \gnl
\gend=
\gbeg{7}{7}
\gcn{1}{1}{2}{1} \gelt{\s\Phi_{\lambda}} \gcn{1}{1}{0}{1} \gcn{1}{1}{2}{1} \gelt{\s\Phi_{\lambda}} \gcn{1}{1}{0}{1} \gnl
\gcn{1}{1}{1}{0} \gcn{1}{1}{1}{0} \glmptb \gnot{\hspace{-0,34cm}\tau_{F,F}} \grmptb\gcn{2}{2}{1}{4} \gcn{1}{2}{-1}{2}  \gnl
\gcn{1}{2}{0}{0} \gcn{1}{1}{0}{0} \hspace{-0,22cm} \gcmu \gcn{1}{1}{0}{3} \gnl
\gvac{1} \glmptb \gnot{\hspace{-0,34cm}\tau_{F,F}} \grmptb \gcl{1} \glcm\gcl{1} \gcl{2} \gnl
\gmu \gmu \gcl{1} \glmptb \gnot{\hspace{-0,34cm}\tau_{B,F}} \grmptb \gnl
\gcn{1}{1}{2}{2} \gvac{1} \gcn{1}{1}{2}{2} \gvac{1} \gmu \gmu \gnl
\gob{2}{F} \gob{2}{F} \gob{2}{F} \gob{2}{B} \gnl
\gend
$$
Multiplying this from the right in the convolution algebra $\K(\Id_{A}, FFFB)$ by $\Phi^{-1}\times \eta_B$ we get precisely \equref{3-cocycle cond fi-lambda} with 
$\Delta_M$ and $\beta$ as in \equref{HN datum}. Observe that 
$\beta=
\gbeg{3}{5}
\got{1}{FFF} \gnl
\gcl{2} \gvac{1}  \gelt{\s\Phi^{-1}}  \gnl 
\gvac{2} \gcl{1} \gnl
\gwmu{3} \gnl
\gob{3}{FFF} 
\gend$. Similarly as in \equref{ab}, we have: 
$$
\gbeg{3}{5}
\got{1}{FFF} \got{3}{XYZ} \gnl
\gbmp{\beta} \gvac{1} \gbmp{\alpha} \gnl
\gcn{1}{1}{1}{3} \gvac{1} \gcl{1} \gnl
\gvac{1} \glmf \hspace{-0,42cm} \gcl{1}  \gnl
\gvac{1} \gob{5}{XYZ}  
\gend= 
\gbeg{3}{5}
\got{1}{FFF} \got{3}{XYZ} \gnl
\gcn{1}{1}{1}{3} \gvac{1} \gcl{3} \gnl
\gvac{1} \glmf \gnl
\gvac{1} \gob{3}{XYZ} 
\gend
$$
where $\alpha$ is the associativity constraint \equref{assoc. rcm}. 

Finally, for \equref{normalized in EM} we find: 
$$
\gbeg{4}{6}
\gvac{3} \got{1}{X} \gnl
\gcn{1}{1}{2}{1} \gelt{\Phi_{\lambda}} \gcn{1}{1}{0}{1} \gcl{1} \gnl 
\gcl{1} \gcl{1} \glmptb \gnot{\hspace{-0,34cm}\tau_{B,X}} \grmptb \gnl
\gcl{1} \glmptb \gnot{\hspace{-0,34cm}\tau_{F,X}} \grmptb \gcl{2} \gnl
\glm \gcu{1} \gnl
\gvac{1} \gob{1}{X} \gob{3}{B} 
\gend=
\gbeg{4}{6}
\gvac{3} \got{1}{X} \gnl
\gcn{1}{1}{2}{1} \gelt{\Phi_{\lambda}} \gcn{1}{1}{0}{1} \gcl{1} \gnl 
\gcl{1} \gcu{1} \glmptb \gnot{\hspace{-0,34cm}\tau_{B,X}} \grmptb \gnl
\gcn{2}{1}{1}{3}  \gcl{1} \gcl{2} \gnl
\gvac{1} \glm \gnl
\gvac{2} \gob{1}{X} \gob{1}{B} 
\gend\stackrel{*}{=}
\gbeg{2}{4}
\got{1}{X} \gnl
\gcl{1} \gu{1} \gnl
\gcl{1} \gcl{1} \gnl
\gob{1}{X} \gob{1}{B}
\gend\qquad\text{and}\qquad
\gbeg{4}{5}
\gvac{3} \got{1}{Y} \gnl
\gcn{1}{1}{2}{1} \gelt{\Phi_{\lambda}} \gcn{1}{1}{0}{1} \gcl{1} \gnl 
\gcu{1} \gcl{1} \glmptb \gnot{\hspace{-0,34cm}\tau_{B,Y}} \grmptb \gnl
\gvac{1} \glm \gcl{1} \gnl
\gvac{2} \gob{1}{Y} \gob{1}{B} 
\gend\stackrel{*}{=}
\gbeg{2}{4}
\got{1}{Y} \gnl
\gcl{1} \gu{1} \gnl
\gcl{1} \gcl{1} \gnl
\gob{1}{Y} \gob{1}{B}
\gend
$$
where the equalities at the places * hold if and only if \equref{normalized 3-cocycle fi-lambda} is fulfilled (set $X=Y=F$ and apply $\eta_F$). 
\qed\end{proof}

For a consequence of \thref{main}  and \prref{Martin case} we get:

\begin{cor} \colabel{Martin}
There is an action of categories ${}_{\tau_B(\A, F, \Phi)}\Tau\times {}_B\K\to {}_B\K$ given by $(X,M)\mapsto XM$ where $XM$ is a left $B$-module via \equref{B act XY psi pon2} 
and the action associativity isomorphism is given by \equref{ro Martin} 
\begin{center}
\begin{tabular} {p{7cm}p{0cm}p{7.6cm}} 
\begin{equation} \eqlabel{B act XY psi pon2}
\gbeg{3}{4}
\got{1}{B} \got{3}{XM} \gnl
\gcn{1}{1}{1}{3} \gvac{1} \gcl{1} \gnl
\gvac{1} \glm \gnl
\gvac{2} \gob{1}{XM} 
\gend=
\gbeg{3}{5}
\gvac{1} \got{1}{B} \got{1}{X} \got{1}{M} \gnl
\glcm \gcl{1} \gcl{3} \gnl
\gcl{1} \glmptb \gnot{\hspace{-0,34cm}\tau_{B,X}} \grmptb \gnl
\glm \glm \gnl
\gob{3}{X} \gob{1}{M}
\gend
\end{equation} & & \vspace{-0,6cm}
\begin{equation} \eqlabel{ro Martin}
\rho_{X,Y,M}=
\gbeg{6}{6}
\gvac{3} \got{1}{X} \got{1}{Y} \got{1}{M} \gnl
\gcn{1}{1}{2}{1} \gelt{\Phi_{\lambda}} \gcn{1}{1}{0}{1} \gcl{1} \gcl{2} \gcl{3} \gnl
\gcl{1} \gcl{1} \glmptb \gnot{\hspace{-0,34cm}\tau_{B,X}} \grmptb \gnl
\gcl{1} \glmptb \gnot{\hspace{-0,34cm}\tau_{F,X}} \grmptb \glmptb \gnot{\hspace{-0,34cm}\tau_{B,Y}} \grmptb \gnl
\glm \glm \glm \gnl 
\gob{3}{X} \gob{1}{Y} \gob{3}{M} 
\gend
\end{equation} 
\end{tabular} 
\end{center}
if and only if  $(B,F, \psi_{B,F}, \Delta_M, \Epsilon_M, \eta_F, \beta)$ is a Hausser-Nill datum where $\psi_{B,F},\Delta_M$ and $\beta$ are given as in \equref{HN datum}. 
\end{cor}

\begin{ex}
The above Corollary is a 2-categorical generalization of the result in \cite[Section 9]{HN}. 
\end{ex}

\medskip

The following result is proved directly, it justifies the name ``2-cocycle'' for $\Phi_{\lambda}$ in \equref{3-cocycle cond fi-lambda}.

\begin{lma} \lelabel{relation HN cocycle}
If \equref{ro fi} is a 2-cocycle in $\EM^M(\K)$ in the sense of \equref{monad law ro new}, where $r_{FF,F}$ and $r_{F,FF}$ are given as below, then $\Phi_{\lambda}$ satisfies \equref{3-cocycle cond fi-lambda}. 
\begin{center}
\begin{tabular} {p{7cm}p{0cm}p{6.8cm}} 
\begin{equation} \eqlabel{ro fi}
r_{\Phi_{\lambda}}=
\gbeg{5}{6}
\gvac{3} \got{1}{F} \got{1}{F} \gnl
\gcn{1}{1}{2}{1} \gelt{\Phi_{\lambda}} \gcn{1}{1}{0}{1} \gcl{1} \gcl{2} \gnl 
\gcl{1} \gcl{1} \glmptb \gnot{\hspace{-0,34cm}\tau_{B,F}} \grmptb \gnl
\gcl{1} \glmptb \gnot{\hspace{-0,34cm}\tau_{F,F}} \grmptb \glmptb \gnot{\hspace{-0,34cm}\tau_{B,F}} \grmptb \gnl
\gmu \gmu \gcl{1} \gnl
\gob{2}{F} \gob{2}{F} \gob{1}{B} 
\gend=
\gbeg{5}{5}
\gvac{3} \got{1}{FF} \gnl
\gcn{1}{1}{2}{1} \gelt{\Phi_{\lambda}} \gcn{1}{1}{0}{1} \gcl{1} \gnl 
\gcl{1} \gvac{1} \glmptb \gnot{\hspace{-0,34cm}\tau_{B,FF}} \grmptb \gnl
\gwmu{3} \gcl{1} \gnl
\gob{3}{FF} \gob{1}{B} 
\gend
\end{equation} & &
\begin{equation*}
r_{FF,F}=
\gbeg{4}{5}
\got{1}{F} \got{1}{F} \got{2}{F} \gnl
\gmu \gcn{1}{1}{2}{2} \gnl
\gvac{1} \hspace{-0,34cm} \glmptb \gnot{r_{\Phi_{\lambda}}} \gcmpb \grmptb \gnl
\gvac{1} \hspace{-0,2cm} \gcmu \gcn{1}{1}{0}{0} \gcn{1}{1}{0}{0} \gnl
\gvac{1} \gob{1}{F} \gob{1}{F} \gob{1}{\hspace{-0,32cm}F} \gob{1}{\hspace{-0,32cm}B} \gnl
\gend, \qquad 
r_{F,FF}=
\gbeg{3}{5}
\got{1}{F} \got{2}{F} \got{1}{\hspace{-0,32cm}F} \gnl
\gcl{1}  \gvac{1} \hspace{-0,32cm} \gmu \gnl
\gvac{1} \hspace{-0,2cm} \glmptb \gnot{r_{\Phi_{\lambda}}} \gcmpb \grmptb \gnl
\gvac{1} \gcl{1} \gcl{1} \hspace{-0,24cm} \gcmu \gnl
\gvac{1} \gob{1}{\hspace{0,32cm}F} \gob{1}{\hspace{0,32cm}F} \gob{1}{F} \gob{1}{B} \gnl
\gend
\end{equation*}
\end{tabular}
\end{center}
\end{lma}

\section{Yetter-Drinfel'd modules and relative Hopf modules} \selabel{Yetter}

When we introduced bimonads, the 2-category of bimonads and biwreaths in \cite{Femic5}, as a part of the data appeared what we suggested as a definition of 
one-sided Yetter-Drinfel`d modules in a 2-category. 
We recall here the definition of a (left) bimonad.

\begin{defn} \delabel{bimonad}
A {\em bimonad} in $\K$ is a quintuple $(\A, F, \mu, \eta, \Delta, \Epsilon, \lambda)$ where $(\A, F, \mu, \eta)$ is a monad and $(\A, F, \Delta, \Epsilon)$ is a comonad so that 
the following compatibility conditions hold: 
$$
\gbeg{3}{5}
\got{1}{F} \got{3}{F} \gnl
\gwmu{3} \gnl
\gvac{1} \gcl{1} \gnl
\gwcm{3} \gnl
\gob{1}{F}\gvac{1}\gob{1}{F}
\gend=
\gbeg{4}{5}
\got{1}{F} \got{3}{F} \gnl
\gcl{1} \gwcm{3} \gnl
\glmptb \gnot{\hspace{-0,34cm}\lambda} \grmptb \gvac{1} \gcl{1} \gnl
\gcl{1} \gwmu{3} \gnl
\gob{1}{F} \gvac{1} \gob{2}{\hspace{-0,34cm}F}
\gend, \quad
\gbeg{2}{3}
\got{1}{F} \got{1}{F} \gnl
\gcl{1} \gcl{1} \gnl
\gcu{1}  \gcu{1} \gnl
\gend=
\gbeg{2}{3}
\got{1}{F} \got{1}{F} \gnl
\gmu \gnl
\gvac{1} \hspace{-0,2cm} \gcu{1} \gnl
\gob{1}{}
\gend, \quad
\gbeg{2}{3}
\gu{1}  \gu{1} \gnl
\gcl{1} \gcl{1} \gnl
\gob{1}{F} \gob{1}{F}
\gend=
\gbeg{2}{3}
\gu{1} \gnl
\hspace{-0,34cm} \gcmu \gnl
\gob{1}{F} \gob{1}{F}
\gend, \quad
\gbeg{1}{2}
\gu{1} \gnl
\gcu{1} \gnl
\gob{1}{}
\gend=
\Id_{id_{\A}}
$$
and the 2-cell $\lambda: FF\to FF$ is such that $(F, \lambda)$ is a 1-cell both in $\Mnd(\K)$ and in $\Comnd(\K)$. 
\end{defn}

We rectify the definition of the 2-category of bimonads from \cite{Femic5} slightly changing the definition of 2-cells. We define the 2-category of bimonads $\Bimnd(\K)$ 
as follows. It has bimonads in $\K$ for 0-cells, 1-cells are triples $(X,\psi,\phi)$ 
where $(X, \psi)$ is a 1-cell in $\Mnd(\K)$, $(X, \phi)$ is a 1-cell in $\Comnd(\K)$ so that the compatibility: 
\begin{equation}  \eqlabel{psi-lambda-phi 2-bimonad}
\gbeg{3}{5}
\got{1}{F'} \got{1}{X} \got{1}{F} \gnl
\glmptb \gnot{\hspace{-0,34cm}\psi} \grmptb \gcl{1} \gnl
\gcl{1} \glmptb \gnot{\hspace{-0,34cm}\lambda} \grmptb \gnl
\glmptb \gnot{\hspace{-0,34cm}\phi} \grmptb \gcl{1} \gnl
\gob{1}{F'} \gob{1}{X} \gob{1}{F}
\gend=
\gbeg{3}{5}
\got{1}{F'} \got{1}{X} \got{1}{F} \gnl
\gcl{1} \glmptb \gnot{\hspace{-0,34cm}\phi} \grmptb \gnl
\glmptb \gnot{\hspace{-0,34cm}\lambda'} \grmptb \gcl{1} \gnl
\gcl{1} \glmptb \gnot{\hspace{-0,34cm}\psi} \grmptb \gnl
\gob{1}{F'} \gob{1}{X} \gob{1}{F}
\gend
\end{equation} 
holds, and 2-cells are 2-cells both in $\Mnd(\K)$ and $\Comnd(\K)$ simultaneously. The compositions of the latter and the identity 1- and 2-cells are defined in the obvious way.

\begin{rem}
With this rectified definition we still have the embedding 2-functor $\Bimnd(\K)\\ \hookrightarrow\bEM(\K)$ to what we conjectured to be the Eilenberg-Moore category for bimonads $\bEM(\K)$. 
We lose though the projection from the latter to the former 2-category, but the projection exists under an assumption that we used in all the examples that we treated in \cite{Femic5} 
(that the canonical restrictions of monadic and comonadic components of the 2-cells in $\bEM(\K)$ coincide). 
The definition of (2-cells in) $\Bimnd(\K)$ does not affect our results in \cite{Femic5}. 
\end{rem}

\medskip

Now a {\em strong Yetter-Drinfel`d module in $\K$} we define as an endomorphism 1-cell $(X,\psi,\phi)$ over a 0-cell $(\A, F, \lambda)$ in $\Bimnd(\K)$. 
We will differ Yetter-Drinfel`d modules and strong Yetter-Drinfel`d modules. In order to stress the difference between the two 
we give the next Definition.

\begin{defn} \delabel{left YD}
Let $(\A, F, \lambda)$ be a left bimonad in $\K$.  
A {\em Yetter-Drinfel`d module in $\K$} is a triple $(X, \psi, \phi)$, where $(X, \psi)$ is a 1-cell in $\Mnd(\K)$ and $(X, \phi)$ is a 1-cell in $\Comnd(\K)$, 
so that the compatibility condition \equref{YD condition} 
\begin{center} 
\begin{tabular}{p{5cm}p{2cm}p{5cm}}
\begin{equation} \eqlabel{YD condition}
\gbeg{3}{5}
\got{1}{F} \got{1}{X} \gnl
\glmptb \gnot{\hspace{-0,34cm}\psi} \grmptb \gu{1} \gnl
\gcl{1} \glmptb \gnot{\hspace{-0,34cm}\lambda} \grmptb \gnl
\glmptb \gnot{\hspace{-0,34cm}\phi} \grmptb \gcu{1} \gnl
\gob{1}{F} \gob{1}{X}
\gend=
\gbeg{3}{5}
\got{1}{F} \got{3}{X}  \gnl
\gcl{1} \glcm \gnl
\glmptb \gnot{\hspace{-0,34cm}\lambda} \grmptb \gcl{1} \gnl
\gcl{1} \glm \gnl
\gob{1}{F} \gob{3}{X}
\gend
\end{equation} & & 
\begin{equation}\eqlabel{psi-lambda-phi}
\gbeg{3}{5}
\got{1}{F} \got{1}{X} \got{1}{F} \gnl
\glmptb \gnot{\hspace{-0,34cm}\psi} \grmptb \gcl{1} \gnl
\gcl{1} \glmptb \gnot{\hspace{-0,34cm}\lambda} \grmptb \gnl
\glmptb \gnot{\hspace{-0,34cm}\phi} \grmptb \gcl{1} \gnl
\gob{1}{F} \gob{1}{X} \gob{1}{F}
\gend=
\gbeg{3}{5}
\got{1}{F} \got{1}{X} \got{1}{F} \gnl
\gcl{1} \glmptb \gnot{\hspace{-0,34cm}\phi} \grmptb \gnl
\glmptb \gnot{\hspace{-0,34cm}\lambda} \grmptb \gcl{1} \gnl
\gcl{1} \glmptb \gnot{\hspace{-0,34cm}\psi} \grmptb \gnl
\gob{1}{F} \gob{1}{X} \gob{1}{F}
\gend
\end{equation}
\end{tabular}
\end{center}
holds true, where 
$\gbeg{2}{3}
\got{1}{F} \got{1}{X} \gnl
\glm \gnl
\gob{3}{X}
\gend=
\gbeg{2}{4}
\got{1}{F} \got{1}{X} \gnl
\glmptb \gnot{\hspace{-0,34cm}\psi} \grmptb \gnl
\gcl{1} \gcu{1} \gnl
\gob{1}{X}
\gend$ and 
$\gbeg{2}{3}
\got{3}{X} \gnl
\glcm \gnl
\gob{1}{F} \gob{1}{X}
\gend=
\gbeg{2}{4}
\got{1}{X} \gnl
\gcl{1} \gu{1} \gnl
\glmptb \gnot{\hspace{-0,34cm}\phi} \grmptb \gnl
\gob{1}{F} \gob{1}{X}
\gend$. A Yetter-Drinfel`d module in $\K$ is called {\em strong} if \equref{psi-lambda-phi} holds. 
\end{defn}

Note that every strong Yetter-Drinfel`d module is a Yetter-Drinfel`d module (apply $\eta_F$ and $\Epsilon_F$ on the right 
hand-side lag of $F$ in \equref{psi-lambda-phi}). 
It is directly proved, as we showed in \cite{Femic5}, that the above 2-cells 
$\gbeg{2}{1}
\glm \gnl
\gend$ and 
$\gbeg{2}{1}
\glcm \gnl
\gend$ equip $X$ with structures of a proper left $F$-module and left $F$-comodule. For the same reason the 2-cells in $\Bimnd(\K)$ turn out to be 
left $F$-linear and left $F$-colinear. 




Let us investigate when Yetter-Drinfel`d modules in $\K$ together with (left $F$-linear and $F$-colinear) morphisms in $\Bimnd(\K)(F)$ form 
a monoidal category. Observe that the comodule version of the identity \equref{B act XY psi} is: 
\begin{equation} \eqlabel{F coact XM psi}
\gbeg{3}{4}
\got{5}{XM} \gnl
\gvac{1} \glcm \gnl
\gcn{1}{1}{3}{1} \gvac{1} \gcl{1} \gnl
\gob{1}{F} \gob{3}{XM} 
\gend=
\gbeg{3}{4}
\got{1}{X} \got{1}{} \got{1}{M} \gnl
\gcl{1} \glcm \gnl
\glmptb \gnot{\hspace{-0,34cm}\phi_{X,F}} \grmptb \gcl{1} \gnl
\gob{1}{F} \gob{1}{X} \gob{1}{M.}
\gend
\end{equation}
for $M:\A'\to\A$ a left $F$-comodule in $\K$ and $(X, \phi_{X,F})\in\Comnd(\K)(F)$. This applies then to two Yetter-Drinfel`d modules
$(X,\psi_X,\phi_X)$ and $(Y,\psi_Y,\phi_Y)$ in $\K$. We find: 
$$
\gbeg{3}{5}
\got{1}{F} \got{1}{XY} \gnl
\glmptb \gnot{\hspace{-0,34cm}\psi} \grmptb \gu{1} \gnl
\gcl{1} \glmptb \gnot{\hspace{-0,34cm}\lambda} \grmptb \gnl
\glmptb \gnot{\hspace{-0,34cm}\phi} \grmptb \gcu{1} \gnl
\gob{1}{F} \gob{1}{XY}
\gend\stackrel{\equref{tau B XY}}{=}
\gbeg{4}{7}
\got{1}{F} \got{1}{X} \got{1}{Y} \gnl
\glmptb \gnot{\hspace{-0,34cm}\psi_X} \grmptb \gcl{1} \gnl
\gcl{1} \glmptb \gnot{\hspace{-0,34cm}\psi_Y} \grmptb \gu{1} \gnl
\gcl{1} \gcl{1} \glmptb \gnot{\hspace{-0,34cm}\lambda} \grmptb \gnl
\gcl{1} \glmptb \gnot{\hspace{-0,34cm}\phi_Y} \grmptb \gcu{1} \gnl
\glmptb \gnot{\hspace{-0,34cm}\phi_X} \grmptb \gcl{1} \gnl
\gob{1}{F} \gob{1}{X} \gob{1}{Y} 
\gend\stackrel{YD_Y}{=}
\gbeg{4}{5}
\got{1}{F} \got{1}{X} \got{3}{Y} \gnl
\glmptb \gnot{\hspace{-0,34cm}\psi_X} \grmptb \glcm \gnl
\gcl{1} \glmptb \gnot{\hspace{-0,34cm}\lambda} \grmptb \gcl{1} \gnl
\glmptb \gnot{\hspace{-0,34cm}\phi_X} \grmptb \glm \gnl
\gob{1}{F} \gob{1}{X} \gob{3}{Y} 
\gend\stackrel{strong}{\stackrel{YD_X}{=}}
\gbeg{4}{7}
\got{1}{F} \got{1}{X} \got{3}{Y} \gnl
\gcl{1} \gcl{1} \glcm \gnl
\gcl{1} \glmptb \gnot{\hspace{-0,34cm}\phi_X} \grmptb \gcl{3} \gnl
\glmptb \gnot{\hspace{-0,34cm}\lambda} \grmptb \gcl{1} \gnl
\gcl{1} \glmptb \gnot{\hspace{-0,34cm}\psi_X} \grmptb \gnl
\gcl{1} \gcl{1} \glm \gnl
\gob{1}{F} \gob{1}{X} \gob{3}{Y} 
\gend\stackrel{\equref{B act XY psi}}{\stackrel{\equref{F coact XM psi}}{=}}
\gbeg{3}{5}
\got{1}{F} \got{3}{XY}  \gnl
\gcl{1} \glcm \gnl
\glmptb \gnot{\hspace{-0,34cm}\lambda} \grmptb \gcl{1} \gnl
\gcl{1} \glm \gnl
\gob{1}{F} \gob{3}{XY.}
\gend
$$
So, we find that Yetter-Drinfel`d modules in $\K$ should be strong in order to form a monoidal category. As endomorphism 1-cells in $\Bimnd(\K)$ 
strong Yetter-Drinfel`d modules, together with the respective 2-cells, form a monoidal category $\Bimnd(\K)(F)$. 
We already have seen that $\psi_{F,XY}$ is defined as in \equref{tau B XY}, similarly holds for $\phi_{XY,F}$. 
That these structures are appropriate for the monoidal structure of strong Yetter-Drinfel`d modules it is easy to show: 
$$
\gbeg{3}{5}
\got{1}{F} \got{1}{XY} \got{1}{F} \gnl
\glmptb \gnot{\hspace{-0,34cm}\psi} \grmptb \gcl{1} \gnl
\gcl{1} \glmptb \gnot{\hspace{-0,34cm}\lambda} \grmptb \gnl
\glmptb \gnot{\hspace{-0,34cm}\phi} \grmptb \gcl{1} \gnl
\gob{1}{F} \gob{1}{XY} \gob{1}{F} 
\gend\stackrel{\equref{tau B XY}}{=}
\gbeg{4}{7}
\got{1}{F} \got{1}{X} \got{1}{X} \got{1}{F} \gnl
\glmptb \gnot{\hspace{-0,34cm}\psi_X} \grmptb \gcl{1} \gcl{2} \gnl
\gcl{1} \glmptb \gnot{\hspace{-0,34cm}\psi_Y} \grmptb \gnl
\gcl{1} \gcl{1} \glmptb \gnot{\hspace{-0,34cm}\lambda} \grmptb \gnl
\gcl{1} \glmptb \gnot{\hspace{-0,34cm}\phi_Y} \grmptb \gcl{2} \gnl
\glmptb \gnot{\hspace{-0,34cm}\phi_X} \grmptb \gcl{1} \gnl
\gob{1}{F} \gob{1}{X} \gob{1}{Y} \gob{1}{F} 
\gend\stackrel{YD_Y}{=}
\gbeg{4}{7}
\got{1}{F} \got{1}{X} \got{1}{Y} \got{1}{F} \gnl
\glmptb \gnot{\hspace{-0,34cm}\psi_X} \grmptb \gcl{1} \gcl{1} \gnl
\gcl{1} \gcl{1} \glmptb \gnot{\hspace{-0,34cm}\phi_Y} \grmptb \gnl
\gcl{1} \glmptb \gnot{\hspace{-0,34cm}\lambda} \grmptb \gcl{1} \gnl
\gcl{1} \gcl{1} \glmptb \gnot{\hspace{-0,34cm}\psi_Y} \grmptb \gnl
\glmptb \gnot{\hspace{-0,34cm}\phi_X} \grmptb \gcl{1} \gcl{1} \gnl
\gob{1}{F} \gob{1}{X} \gob{1}{Y} \gob{1}{F} 
\gend\stackrel{strong}{\stackrel{YD_X}{=}}
\gbeg{4}{7}
\got{1}{F} \got{1}{X} \got{1}{Y}  \got{1}{F} \gnl
\gcl{1} \gcl{1} \glmptb \gnot{\hspace{-0,34cm}\phi_Y} \grmptb \gnl
\gcl{1} \glmptb \gnot{\hspace{-0,34cm}\phi_X} \grmptb \gcl{3} \gnl
\glmptb \gnot{\hspace{-0,34cm}\lambda} \grmptb \gcl{1} \gnl
\gcl{1} \glmptb \gnot{\hspace{-0,34cm}\psi_X} \grmptb \gnl
\gcl{1} \gcl{1} \glmptb \gnot{\hspace{-0,34cm}\psi_Y} \grmptb \gnl
\gob{1}{F} \gob{1}{X} \gob{1}{Y}  \gob{1}{F}
\gend\stackrel{\equref{tau B XY}}{=}
\gbeg{3}{5}
\got{1}{F} \got{1}{XY} \got{1}{F} \gnl
\gcl{1} \glmptb \gnot{\hspace{-0,34cm}\phi} \grmptb \gnl
\glmptb \gnot{\hspace{-0,34cm}\lambda} \grmptb \gcl{1} \gnl
\gcl{1} \glmptb \gnot{\hspace{-0,34cm}\psi} \grmptb \gnl
\gob{1}{F} \gob{1}{XY} \gob{1}{F}
\gend
$$

\begin{prop}
The category ${}^F_F\YD_s(\K,\A)=\Bimnd(\K)(F)$ of strong Yetter-Drinfel`d modules in $\K$ is monoidal. 
\end{prop}

\subsection{The action of the category of strong Yetter-Drinfel`d modules} \sslabel{action YD}

Now suppose that  $B$ is a 
monad and that there is a 1-cell $(F, \psi_{B,F})$ in $\Tau(\A, B)$. 
Let ${}_{\psi_B}{}^F_F\YD_s(\K,\A)$ denote the following category. 
Its objects are strong Yetter-Drinfel`d modules $(X, \psi_X, \phi_X)$ for which there is a 2-cell $\psi_{B,X}: BX\to XB$ so that 
$(X, \psi_{B,X})$ is a 1-cell in $\Tau(\A, B)$ and the 
following compatibility Yang-Baxter equations hold true: 
\begin{center} 
\begin{tabular} {p{6cm}p{0cm}p{6.8cm}} 
\begin{equation} \eqlabel{YBE psi BFX}
\gbeg{3}{5}
\got{1}{B} \got{1}{F} \got{1}{X} \gnl
\gcl{1} \glmptb \gnot{\hspace{-0,34cm}\psi_{F,X}} \grmptb \gnl
\glmptb \gnot{\hspace{-0,34cm}\psi_{B,X}} \grmptb \gcl{1} \gnl
\gcl{1} \glmptb \gnot{\hspace{-0,34cm}\psi_{B,F}} \grmptb \gnl
\gob{1}{X} \gob{1}{F} \gob{1}{B}
\gend=
\gbeg{3}{5}
\got{1}{B} \got{1}{F} \got{1}{X} \gnl
\glmptb \gnot{\hspace{-0,34cm}\psi_{B,F}} \grmptb \gcl{1} \gnl
\gcl{1} \glmptb \gnot{\hspace{-0,34cm}\psi_{B,X}} \grmptb \gnl
\glmptb \gnot{\hspace{-0,34cm}\psi_{F,X}} \grmptb \gcl{1} \gnl
\gob{1}{X} \gob{1}{F} \gob{1}{B}
\gend
\end{equation} & & 
\begin{equation}\eqlabel{YBE BXF}
\gbeg{3}{5}
\got{1}{B} \got{1}{X} \got{1}{F} \gnl
\gcl{1} \glmptb \gnot{\hspace{-0,34cm}\phi_{X,F}} \grmptb \gnl
\glmptb \gnot{\hspace{-0,34cm}\psi_{B,F}} \grmptb \gcl{1} \gnl
\gcl{1} \glmptb \gnot{\hspace{-0,34cm}\psi_{B,X}} \grmptb \gnl
\gob{1}{F} \gob{1}{X} \gob{1}{B}
\gend=
\gbeg{3}{5}
\got{1}{B} \got{1}{X} \got{1}{F} \gnl
\glmptb \gnot{\hspace{-0,34cm}\psi_{B,X}} \grmptb \gcl{1} \gnl
\gcl{1} \glmptb \gnot{\hspace{-0,34cm}\psi_{B,F}} \grmptb \gnl
\glmptb \gnot{\hspace{-0,34cm}\phi_{X,F}} \grmptb \gcl{1} \gnl
\gob{1}{F} \gob{1}{X} \gob{1}{B}
\gend
\end{equation} 
\end{tabular}
\end{center}
Note that these two identities can be restated so that $\psi_{B,FX}$ is natural with respect to $\psi_{F,X}$ and that $\psi_{B,XF}$ is natural with respect to $\phi_{X,F}$. 
Morphisms of ${}_{\psi_B}{}^F_F\YD_s(\K,\A)$ are (left $F$-linear and $F$-colinear) morphisms in $\Bimnd(\K)(F)$ and in $\Mnd(\K)(B)$. 
It is straightforwardly showed that ${}_{\psi_B}{}^F_F\YD_s(\K,\A)$ is a monoidal category. Observe that the identity \equref{YBE psi BFX} is the same one as in \equref{YBE BFX}. 

\bigskip

\begin{defn} \delabel{relative}
Let $B$ be a monad, $F$ a comonad, both on a 0-cell $\A$ in $\K$, and suppose there is a 1-cell $(F, \psi_{B,F})$ in $\Tau(\A, B)$. 
A {\em left relative $(F,B)$-module in $\K$} is a left $F$-comodule and a left $B$-module $M: \A'\to\A$ so that
$$
\gbeg{2}{4}
\got{1}{B} \got{1}{M}  \gnl
\glm \gnl
\glcm \gnl
\gob{1}{F} \gob{1}{M}
\gend=
\gbeg{3}{5}
\got{1}{B} \got{3}{M}  \gnl
\gcl{1} \glcm \gnl
\glmptb \gnot{\hspace{-0,34cm}\psi_{B,F}} \grmptb \gcl{1} \gnl
\gcl{1} \glm \gnl
\gob{1}{F} \gob{3}{M}
\gend
$$
holds. 

The category of left relative $(F,B)$-modules in $\K$ and left $F$-colinear and left $B$-linear 2-cells in $\K$ we will denote by ${}^F_B\K$. 
\end{defn}

Note that if $B$ is a left $F$-comodule monad (in the sense of \equref{F comod alg}), then it is an object of ${}^F_B\K$. Hence, if one wants 
to have $B$ inside, one should add the comodule monad property on $B$. 

On one hand, this definition generalizes relative/Hopf/entwined modules to 2-categories. On the other hand, it may be seen as a 2-categorical generalization of the result in 
\cite{Wolff}, recalled in \cite[Section 2.1]{Wisb}, that the category of algebras over a lifted monad $\hat B$ on the category of coalgebras over a comonad $F$ on an ordinary 
category $\C$ is isomorphic to the category of coalgebras over a lifted comonad $\hat F$ of the category of algebras over a monad $B$: $(\C^F)_{\hat B}\iso(\C_B)^{\hat F}$, 
provided a mixed distributed law $\psi: BF\to FB$ from a monad to a comonad. 
In our context we take this isomorphism of categories as a definition, so our $\psi_{B,F}$ is a distributive law with respect only to 
the monadic structure of $B$.

\begin{thm} \thlabel{YD action}
There is an action of categories ${}_{\psi_B}{}^F_F\YD_s(\K,\A)\times {}^F_B\K\to{}^F_B\K$ given by $(X,M)\mapsto XM$, where $XM$ is a left $F$-comodule by 
\equref{F coact XM psi} and a left $B$-module via \equref{B act XY psi}. 
\end{thm}

\begin{proof} 
We will consider the category action associativity isomorphism to be identity. Let us prove that a 1-cell $XM$ is indeed a left relative $(F,B)$-module 
with the indicated structures: 
$$
\gbeg{3}{6}
\got{1}{B} \got{3}{XM}  \gnl
\gcn{2}{1}{1}{3} \gcl{1} \gnl
\gvac{1} \glm \gnl
\gvac{1} \glcm \gnl
\gcn{2}{1}{3}{1} \gcl{1} \gnl
\gob{1}{F} \gob{3}{XM}
\gend=
\gbeg{3}{6}
\got{1}{B} \got{1}{X} \got{1}{M} \gnl
\glmptb \gnot{\hspace{-0,34cm}\psi_{B,X}} \grmptb \gcl{1} \gnl
\gcl{1} \glm \gnl
\gcl{1} \glcm \gnl
\glmptb \gnot{\hspace{-0,34cm}\phi_{X,F}} \grmptb \gcl{1} \gnl
\gob{1}{F} \gob{1}{X} \gob{1}{M}
\gend\stackrel{M \s rel.}{=}
\gbeg{4}{5}
\got{1}{B} \got{1}{X} \got{3}{M} \gnl
\glmptb \gnot{\hspace{-0,34cm}\psi_{B,X}} \grmptb \glcm \gnl
\gcl{1} \glmptb \gnot{\hspace{-0,34cm}\psi_{B,F}} \grmptb \gcl{1} \gnl
\glmptb \gnot{\hspace{-0,34cm}\phi_{X,F}} \grmptb \glm \gnl
\gob{1}{F} \gob{1}{X} \gob{3}{M}
\gend\stackrel{\equref{YBE BXF}}{=}
\gbeg{4}{7}
\got{1}{B} \got{1}{X} \got{3}{M}  \gnl
\gcl{1} \gcl{1} \glcm \gnl
\gcl{1} \glmptb \gnot{\hspace{-0,34cm}\phi_{X,F}} \grmptb \gcl{2} \gnl
\glmptb \gnot{\hspace{-0,34cm}\psi_{B,F}} \grmptb \gcl{1} \gnl
\gcl{1}  \glmptb \gnot{\hspace{-0,34cm}\psi_{B,X}} \grmptb \gcl{1} \gnl
\gcl{1} \gcl{1} \glm \gnl
\gob{1}{F} \gob{1}{X} \gob{3}{M}
\gend =
\gbeg{3}{5}
\got{1}{B} \got{3}{XM}  \gnl
\gcl{1} \glcm \gnl
\glmptb \gnot{\hspace{-0,34cm}\psi_{B,F}} \grmptb \gcl{1} \gnl
\gcl{1} \glm \gnl
\gob{1}{F} \gob{3}{XM.}
\gend 
$$
The proof that given a morphism $F:M\to N$ in ${}^F_B\K$, the newly obtained morphism $X\times M: XM\to XN$ is in ${}^F_B\K$, is direct, 
as well as that the defined action of categories is compatible with composition of morphisms and identity morphisms. 
\qed\end{proof}

\begin{rem}
We could have considered the category ${}^F_F\YD_s(\K,\A)$ and require that for every $X\in {}^F_F\YD_s(\K,\A)$ there exists $\psi_{B,X}$ so that 
$(X, \psi_{B,X})$ is a 1-cell in $\Tau(\A, B)$ and the relations \equref{YBE psi BFX} and \equref{YBE BXF} hold true. Then there is an action of 
${}^F_F\YD_s(\K,\A)$ on ${}^F_B\K$ (precisely because \equref{YBE BXF} holds). 
\end{rem}

\bigskip

In the right hand-side version of \deref{left YD} the right strong Yetter-Drinfel`d modules are triples $(X, \psi'_{X, F}, \phi'_{F,X})$ where 
$\psi'_{X, F}: XF\to FX$ and $\phi'_{F,X}: XF\to FX$, so that $(X, \psi'_{X, F})\in\Mnd(\K^{op})$ and $(X, \phi'_{F,X})\in\Comnd(\K^{op})$. The corresponding 
category we denote by $\YD_s(\K,\A)^F_F$. Moreover, we consider the category ${}_{\psi_B}\YD_s(\K,\A)^F_F$ analogous to the category ${}_{\psi_B}{}^F_F\YD_s(\K,\A)$ 
from the beginning of this Subsection. 
The only difference in the definition is that the 2-cells $\psi_{F,X}$ and $\phi_{X,F}$ are now substituted by 
$\phi_{F,X}'$ and $\psi_{X,F}'$\hspace{0,04cm}, respectively. Observe that this change, although formally minor, implies changing from $\Mnd(\K)$ to $\Comnd(\K^{op})$, and 
from $\Comnd(\K)$ to $\Mnd(\K^{op})$. Also in this setting we have a 1-cell $(F, \psi_{B,F})$ in $\Tau(\A, B)$ and 2-cells $\psi_{B,X}: BX\to XB$ so that 
$(X, \psi_{B,X})\in\Tau(\A, B)$ for every $(X, \psi'_{X, F}, \phi'_{F,X})\in \YD_s(\K,\A)^F_F$ and so that the  
following compatibility Yang-Baxter equations hold: 
\begin{center} 
\begin{tabular} {p{6cm}p{0cm}p{6.8cm}} 
\begin{equation} \eqlabel{YBE phi' BFX}
\gbeg{3}{5}
\got{1}{B} \got{1}{F} \got{1}{X} \gnl
\gcl{1} \glmptb \gnot{\hspace{-0,34cm}\phi'_{F,X}} \grmptb \gnl
\glmptb \gnot{\hspace{-0,34cm}\psi_{B,X}} \grmptb \gcl{1} \gnl
\gcl{1} \glmptb \gnot{\hspace{-0,34cm}\psi_{B,F}} \grmptb \gnl
\gob{1}{X} \gob{1}{F} \gob{1}{B}
\gend=
\gbeg{3}{5}
\got{1}{B} \got{1}{F} \got{1}{X} \gnl
\glmptb \gnot{\hspace{-0,34cm}\psi_{B,F}} \grmptb \gcl{1} \gnl
\gcl{1} \glmptb \gnot{\hspace{-0,34cm}\psi_{B,X}} \grmptb \gnl
\glmptb \gnot{\hspace{-0,34cm}\phi'_{F,X}} \grmptb \gcl{1} \gnl
\gob{1}{X} \gob{1}{F} \gob{1}{B}
\gend  
\end{equation} & & 
\begin{equation}\eqlabel{YBE BXF'}
\gbeg{3}{5}
\got{1}{B} \got{1}{X} \got{1}{F} \gnl
\gcl{1} \glmptb \gnot{\hspace{-0,34cm}\psi'_{X,F}} \grmptb \gnl
\glmptb \gnot{\hspace{-0,34cm}\psi_{B,F}} \grmptb \gcl{1} \gnl
\gcl{1} \glmptb \gnot{\hspace{-0,34cm}\psi_{B,X}} \grmptb \gnl
\gob{1}{F} \gob{1}{X} \gob{1}{B}
\gend=
\gbeg{3}{5}
\got{1}{B} \got{1}{X} \got{1}{F} \gnl
\glmptb \gnot{\hspace{-0,34cm}\psi_{B,X}} \grmptb \gcl{1} \gnl
\gcl{1} \glmptb \gnot{\hspace{-0,34cm}\psi_{B,F}} \grmptb \gnl
\glmptb \gnot{\hspace{-0,34cm}\psi'_{X,F}} \grmptb \gcl{1} \gnl
\gob{1}{F} \gob{1}{X} \gob{1}{B}
\gend
\end{equation} 
\end{tabular}
\end{center}
This is a monoidal category and we have:

\begin{thm} \thlabel{YD-right action}
There is an action of categories ${}_{\psi_B}\YD_s(\K,\A)^F_F\times {}^F_B\K\to{}^F_B\K$ given by $(X,M)\mapsto XM$, where $XM$ is a left $B$-module via \equref{B act XY psi}
and a left $F$-comodule by: 
$$
\gbeg{3}{4}
\got{5}{XM} \gnl
\gvac{1} \glcm \gnl
\gcn{1}{1}{3}{1} \gvac{1} \gcl{1} \gnl
\gob{1}{F} \gob{3}{XM} 
\gend=
\gbeg{3}{4}
\got{1}{X} \got{1}{} \got{1}{M} \gnl
\gcl{1} \glcm \gnl
\glmptb \gnot{\hspace{-0,34cm}\psi'_{X,F}} \grmptb \gcl{1} \gnl
\gob{1}{F} \gob{1}{X} \gob{1}{M.}
\gend
$$
\end{thm}

The proof of this result is analogous as in \thref{YD action}.

\subsection{Relation to classical Yetter-Drinfel`d modules}

We are first going to prove some results and these will lead to the introduction of another 2-category which resembles more (and generalizes) 
the classical Yetter-Drinfel`d modules. 

\begin{prop} \prlabel{prepairing psi-phi YD}
Assume the following: 
\begin{enumerate}
\item let $F$ be a bimonad with $\lambda$ indicated as below, where $\tau_{F,F}$ is both left monadic and comonadic and right monadic and 
comonadic distributive law; 
\item let $(X, \nu)$ be a left $F$-module so that $\nu: (FX, \tau_{F,FX})\to (X, \tau_{F,X})$ is a morphism in $\Mnd(\K)(F)$ 
(in other words, $\tau_{F,X}$ is a left monadic distributive law and natural with respect to 
$\gbeg{2}{1}
\glm \gnl
\gend$);
\item let $(X, l)$ be a left $F$-comodule so that $l:(X, \tau_{X,F})\to (FX, \tau_{FX,F})$ is a morphism in $\Comnd(\K)(F)$ (in other words, 
$\tau_{X,F}$ is left comonadic and natural with respect to 
$\gbeg{2}{1}
\glcm \gnl
\gend$). 
\end{enumerate}
Then $(X, \psi_{F,X})$ is a 1-cell in $\Mnd(\K)(F)$ and $(X, \phi_{X,F})$ is a 1-cell in $\Comnd(\K)(F)$, where:
\begin{equation} \eqlabel{YD concrete}
\lambda=
\gbeg{4}{5}
\got{2}{F} \got{1}{F} \gnl
\gcmu \gcl{1} \gnl
\gcl{1} \glmptb \gnot{\hspace{-0,34cm}\tau_{F,F}} \grmptb \gnl
\gmu \gcl{1} \gnl
\gob{2}{F} \gob{1}{F} 
\gend,\quad
\psi_{F,X}=
\gbeg{4}{5}
\got{2}{F} \got{1}{X} \gnl
\gcmu \gcl{1} \gnl
\gcl{1} \glmptb \gnot{\hspace{-0,34cm}\tau_{F,X}} \grmptb \gnl
\glm \gcl{1} \gnl
\gvac{1} \gob{1}{X} \gob{1}{F} 
\gend, \quad
\phi_{X,F}=
\gbeg{4}{5}
\gvac{1} \got{1}{X} \got{1}{F} \gnl
\glcm \gcl{1} \gnl
\gcl{1} \glmptb \gnot{\hspace{-0,34cm}\tau_{X,F}} \grmptb \gnl
\gmu \gcl{1} \gnl
\gob{2}{F} \gob{1}{X.} 
\gend
\end{equation}
\end{prop}

\begin{proof}
The proof that $\lambda$ is a desired distributive law making $F$ a left bimonad is direct. Then the claim for $(X, \psi_{F,X})$ holds by 
\rmref{psi versus tau}, and the one for $(X, \psi_{F,X})$ holds by duality. 
\qed\end{proof}


\begin{prop} \prlabel{YD are strong}
Let $(X, \psi_{F,X}, \phi_{X,F})$ be a Yetter-Drinfel`d module determined by \equref{YD concrete}. 
Under the hypotheses as in the above Proposition and moreover assuming: 
\begin{itemize}
\item  $l: (X, \tau_{F,X})\to (FX, \tau_{F,FX})$ is a morphism in $\Mnd(\K)(F)$ (in other words, 
$\tau_{F,X}$ is a right comonadic distributive law and natural with respect to  
$\gbeg{2}{1}
\glcm \gnl
\gend$);
\item $\nu:(FX, \tau_{FX,F})\to (X, \tau_{X,F})$ is a morphism in $\Comnd(\K)(F)$ (in other words, 
$\tau_{X,F}$ is right monadic and natural with respect to 
$\gbeg{2}{1}
\glm \gnl
\gend$); 
\item the Yang-Baxter equation
\begin{equation} \eqlabel{YBE tau FXF}
\gbeg{3}{5}
\got{1}{F} \got{1}{X} \got{1}{F} \gnl
\gcl{1} \glmptb \gnot{\hspace{-0,34cm}\tau_{X,F}} \grmptb \gnl
\glmptb \gnot{\hspace{-0,34cm}\tau_{F,F}} \grmptb \gcl{1} \gnl
\gcl{1} \glmptb \gnot{\hspace{-0,34cm}\tau_{F,X}} \grmptb \gnl
\gob{1}{F} \gob{1}{X} \gob{1}{F}
\gend=
\gbeg{3}{5}
\got{1}{F} \got{1}{X} \got{1}{F} \gnl
\glmptb \gnot{\hspace{-0,34cm}\tau_{F,X}} \grmptb \gcl{1} \gnl
\gcl{1} \glmptb \gnot{\hspace{-0,34cm}\tau_{F,F}} \grmptb \gnl
\glmptb \gnot{\hspace{-0,34cm}\tau_{X,F}} \grmptb \gcl{1} \gnl
\gob{1}{F} \gob{1}{X} \gob{1}{F}
\gend
\end{equation} 
holds;
\end{itemize}
the Yetter-Drinfel`d module $(X, \psi_{F,X}, \phi_{X,F})$ is a strong Yetter-Drinfel`d module.  
\end{prop}

\begin{proof}
We prove the relation \equref{psi-lambda-phi} 
for a Yetter-Drinfel`d module $X$: 
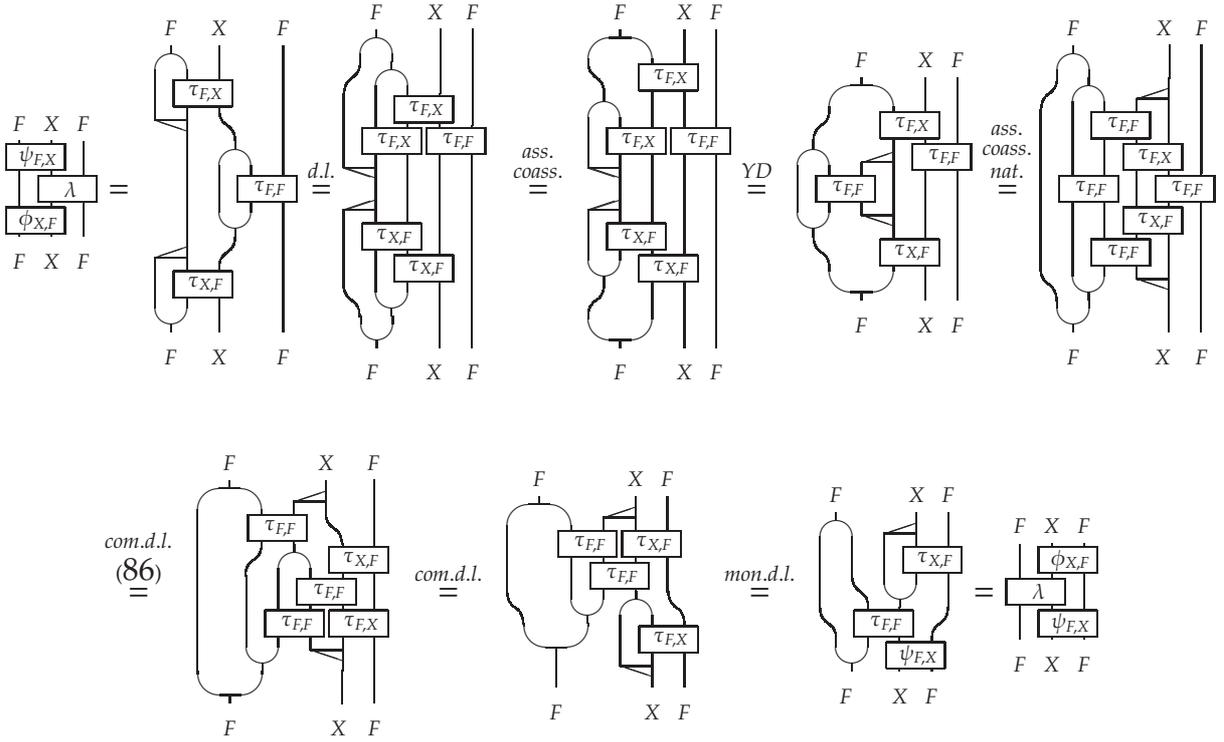
$$\gbeg{3}{5}
\got{1}{F} \got{1}{X} \got{1}{F} \gnl
\glmptb \gnot{\hspace{-0,34cm}\psi_{F,X}} \grmptb \gcl{1} \gnl
\gcl{1} \glmptb \gnot{\hspace{-0,34cm}\lambda} \grmptb \gnl
\glmptb \gnot{\hspace{-0,34cm}\phi_{X,F}} \grmptb \gcl{1} \gnl
\gob{1}{F} \gob{1}{X} \gob{1}{F}
\gend=
\gbeg{5}{11}
\got{2}{F} \got{1}{X} \got{3}{F} \gnl
\gcmu \gcl{1} \gvac{1} \gcl{3} \gnl
\gcl{1} \glmptb \gnot{\hspace{-0,34cm}\tau_{F,X}} \grmptb \gnl
\glm \gcn{1}{1}{1}{2} \gnl
\gvac{1} \gcl{1}  \gcmu \gcl{1} \gnl
\gvac{1} \gcl{1} \gcl{1} \glmptb \gnot{\hspace{-0,34cm}\tau_{F,F}} \grmptb \gnl
\gvac{1} \gcl{1} \gmu \gcl{1} \gnl
\glcm \gcn{2}{1}{2}{1} \gcl{3} \gnl
\gcl{1} \glmptb \gnot{\hspace{-0,34cm}\tau_{X,F}} \grmptb \gnl
\gmu \gcl{1} \gnl
\gob{2}{F} \gob{1}{X} \gob{3}{F}
\gend\stackrel{d.l.}{=}
\gbeg{5}{12}
\got{2}{F} \got{1}{\hspace{0,28cm}X} \got{2}{F} \gnl
\gcmu \gcn{1}{2}{2}{2} \gcn{1}{3}{2}{2} \gnl
\gcn{1}{1}{1}{0} \hspace{-0,22cm} \gcmu \gnl
\gcl{2} \gcl{1} \glmptb \gnot{\hspace{-0,34cm}\tau_{F,X}} \grmptb \gnl
\gvac{1} \glmptb \gnot{\hspace{-0,34cm}\tau_{F,X}} \grmptb \glmptb \gnot{\hspace{-0,34cm}\tau_{F,F}} \grmptb \gnl
\glm \gcl{2} \gcl{3} \gcl{6} \gnl 
\glcm \gnl
\gcl{2} \glmptb \gnot{\hspace{-0,34cm}\tau_{X,F}} \grmptb \gnl
\gvac{1} \gcl{1} \glmptb \gnot{\hspace{-0,34cm}\tau_{X,F}} \grmptb \gnl
\gcn{1}{1}{1}{2} \gmu \gcl{2} \gnl 
\gvac{1} \hspace{-0,3cm} \gmu \gnl
\gob{3}{\hspace{0,3cm}F} \gob{1}{\hspace{0,24cm}X} \gob{2}{F}
\gend\stackrel{ass.}{\stackrel{coass.}{=}}
\gbeg{5}{12}
\gvac{1} \got{1}{F} \gvac{1} \got{1}{X} \got{1}{F} \gnl
\gwcm{3} \gcl{1} \gcl{3} \gnl
\gcn{1}{1}{1}{2} \gvac{1} \glmptb \gnot{\hspace{-0,34cm}\tau_{F,X}} \grmptb \gnl
\gcmu \gcl{1} \gcl{1} \gnl
\gcl{1} \glmptb \gnot{\hspace{-0,34cm}\tau_{F,X}} \grmptb \glmptb \gnot{\hspace{-0,34cm}\tau_{F,F}} \grmptb \gnl
\glm \gcl{2} \gcl{3} \gcl{6} \gnl 
\glcm \gnl
\gcl{1} \glmptb \gnot{\hspace{-0,34cm}\tau_{X,F}} \grmptb \gnl
\gmu \glmptb \gnot{\hspace{-0,34cm}\tau_{X,F}} \grmptb \gnl
\gcn{2}{1}{2}{1} \gcl{1} \gcl{2} \gnl
\gwmu{3} \gnl
\gob{3}{F} \gob{1}{X} \gob{1}{F}
\gend\stackrel{YD}{=}
\gbeg{6}{9}
\gvac{2} \got{1}{F} \gvac{1} \got{1}{X} \got{1}{F} \gnl
\gvac{1} \gwcm{3} \gcl{1} \gcl{2} \gnl
\gcn{3}{1}{3}{2} \glmptb \gnot{\hspace{-0,34cm}\tau_{F,X}} \grmptb \gnl
\gcmu \glcm  \glmptb \gnot{\hspace{-0,34cm}\tau_{F,F}} \grmptb \gnl
\gcl{1} \glmptb \gnot{\hspace{-0,34cm}\tau_{F,F}} \grmptb \gcl{1} \gcl{2} \gcl{4} \gnl
\gmu \glm \gnl 
\gcn{3}{1}{2}{3} \glmptb \gnot{\hspace{-0,34cm}\tau_{X,F}} \grmptb \gnl
\gvac{1} \gwmu{3} \gcl{1} \gnl
\gvac{2} \gob{1}{F} \gvac{1} \gob{1}{X} \gob{1}{F}
\gend\stackrel{ass.}{\stackrel{coass.}{\stackrel{nat.}{=}}}
\gbeg{5}{11}
\got{2}{F} \gvac{1} \got{1}{\hspace{0,28cm}X} \got{2}{F} \gnl
\gcmu \gvac{1} \gcn{1}{1}{2}{2} \gcn{1}{4}{2}{2} \gnl
\gcn{1}{1}{1}{0} \hspace{-0,22cm} \gcmu \glcm \gnl
\gcl{5} \gcl{2} \glmptb \gnot{\hspace{-0,34cm}\tau_{F,F}} \grmptb \gcl{1} \gnl
\gvac{2} \gcl{1} \glmptb \gnot{\hspace{-0,34cm}\tau_{F,X}} \grmptb \gnl
\gcl{1} \glmptb \gnot{\hspace{-0,34cm}\tau_{F,F}} \grmptb \gcl{1} \glmptb \gnot{\hspace{-0,34cm}\tau_{F,F}} \grmptb \gnl 
\gvac{1} \gcl{2} \gcl{1} \glmptb \gnot{\hspace{-0,34cm}\tau_{X,F}} \grmptb \gcl{4} \gnl
\gvac{2} \glmptb \gnot{\hspace{-0,34cm}\tau_{F,F}} \grmptb \gcl{3} \gnl
\gcn{1}{1}{1}{2} \gmu \glm \gnl 
\gvac{1} \hspace{-0,2cm} \gmu \gnl
\gob{4}{F} \gob{1}{\hspace{0,24cm}X} \gob{2}{F}
\gend
$$

$$\stackrel{com. d.l.}{\stackrel{\equref{YBE tau FXF}}{=}}
\gbeg{7}{9}
\got{3}{F} \gvac{1} \got{1}{X} \got{2}{F} \gnl
\gwcm{3} \glcm \gcn{1}{2}{2}{2} \gnl
\gcl{5} \gvac{1} \glmptb \gnot{\hspace{-0,34cm}\tau_{F,F}} \grmptb \gcn{1}{1}{1}{2}  \gnl
\gvac{2} \gcn{1}{1}{1}{0} \hspace{-0,2cm} \gcmu \glmptb \gnot{\hspace{-0,34cm}\tau_{X,F}} \grmptb  \gnl
\gvac{2} \gcl{1} \gcl{1} \glmptb \gnot{\hspace{-0,34cm}\tau_{F,F}} \grmptb \gcl{1} \gnl
\gvac{2} \gcl{1} \glmptb \gnot{\hspace{-0,34cm}\tau_{F,F}} \grmptb \glmptb \gnot{\hspace{-0,34cm}\tau_{F,X}} \grmptb \gnl
\gvac{2} \gmu \glm \gcl{2} \gnl
\gvac{1} \hspace{-0,32cm} \gwmu{3} \gvac{1} \gcn{1}{1}{2}{2} \gnl
\gob{5}{F} \gob{1}{\hspace{0,32cm}X} \gob{2}{F}
\gend\stackrel{com. d.l.}{=}
\gbeg{7}{8}
\got{3}{F} \gvac{1} \got{1}{X} \got{1}{F} \gnl
\gwcm{3} \glcm \gcl{1} \gnl
\gcl{2} \gvac{1} \glmptb \gnot{\hspace{-0,34cm}\tau_{F,F}} \grmptb \glmptb \gnot{\hspace{-0,34cm}\tau_{X,F}} \grmptb \gnl
\gvac{2} \gcl{1} \glmptb \gnot{\hspace{-0,34cm}\tau_{F,F}} \grmptb \gcl{1} \gnl
\gcn{2}{1}{1}{2} \gmu \hspace{-0,2cm} \gcmu \gcn{1}{1}{0}{1} \gnl
\gvac{1} \gwmu{3} \gcl{1} \glmptb \gnot{\hspace{-0,34cm}\tau_{F,X}} \grmptb \gnl
\gvac{2} \gcl{1} \gvac{1} \glm \gcl{1} \gnl
\gob{5}{F} \gob{1}{X} \gob{1}{F}
\gend\stackrel{mon. d.l.}{=}
\gbeg{5}{7}
\got{2}{F} \gvac{1} \got{1}{X} \got{1}{F} \gnl
\gcmu \glcm \gcl{1} \gnl
\gcl{2} \gcl{1} \gcl{1} \glmptb \gnot{\hspace{-0,34cm}\tau_{X,F}} \grmptb \gnl
\gvac{1} \gcn{1}{1}{1}{2} \gmu \gcl{1} \gnl
\gcn{2}{1}{1}{2} \hspace{-0,22cm} \glmptb \gnot{\hspace{-0,34cm}\tau_{F,F}} \grmptb \gcn{1}{1}{2}{1} \gnl
\gvac{1} \gmu \glmptb \gnot{\hspace{-0,34cm}\psi_{F,X}} \grmptb \gnl
\gob{3}{\hspace{0,28cm} F} \gob{1}{X} \gob{1}{F}
\gend=
\gbeg{3}{5}
\got{1}{F} \got{1}{X} \got{1}{F} \gnl
\gcl{1} \glmptb \gnot{\hspace{-0,34cm}\phi_{X,F}} \grmptb \gnl
\glmptb \gnot{\hspace{-0,34cm}\lambda} \grmptb \gcl{1} \gnl
\gcl{1} \glmptb \gnot{\hspace{-0,34cm}\psi_{F,X}} \grmptb \gnl
\gob{1}{F} \gob{1}{X} \gob{1}{F}
\gend
$$
\qed\end{proof}

\begin{cor}
In any braided monoidal category $\C$ Yetter-Drinfel`d modules over a bialgebra $H$ in $\C$ are strong in the 2-category induced by $\C$. 
\end{cor}

The classically studied Yetter-Drinfel`d modules over a bialgebra $H$ over a field all turn out to be strong in 
the sense of \equref{psi-lambda-phi}.

\medskip

We now introduce the following 2-category, which we will call {\em the 2-category of left $\tau$-bimonads in $\K$} and denote by $\tau\x\Bimnd^l(\K)$. 
A left $\tau$-bimonad in $\K$ is a left bimonad $(F, \lambda)$ in $\K$ such that $\lambda$ is given as in \equref{YD concrete}. We 
dealt with $\tau$-bimonads in \cite{Femic6}.

\underline{0-cells:} are $\tau$-bimonads $(\A, F, \tau_{F,F})$ in $\K$. 

\underline{1-cells:} are quintuples $(X,\tau_{F,X}, \tau_{X,F}, \nu, l): (\A, F, \tau_{F,F})\to(\A', F', \tau_{F',F'})$ where: 

$(X, \tau_{F,X})$ are 1-cells in $\Mnd(\K)$ and in $\Mnd(\K^{op})$ and $(X, \tau_{X,F})$ are 1-cells in $\Comnd(\K)$ and in $\Comnd(\K^{op})$ 
(observe the notation: $\tau_{F,X}: F'X\to XF$ and $\tau_{X,F}: XF\to F'X$), 

$(X, \nu)$ is a left $F'$-module and a left $F'$-comodule, so that 

$\nu: (F'X, \tau_{F,F'X})\to (X, \tau_{F,X})$ and $l: (X, \tau_{F,X})\to (F'X, \tau_{F,F'X})$ are morphisms in $\Mnd(\K)$, 
$\nu:(F'X, \tau_{F'X,F})\to (X, \tau_{X,F})$ and $l:(X, \tau_{X,F})\to (F'X, \tau_{F'X,F})$ are morphisms in $\Comnd(\K)$, recall: 
$$\tau_{F,F'X}=
\gbeg{3}{4}
\got{1}{F'} \got{1}{F'} \got{1}{X} \gnl
\glmptb \gnot{\hspace{-0,34cm}\tau_{F', F'}} \grmptb \gcl{1} \gnl
\gcl{1} \glmptb \gnot{\hspace{-0,34cm}\tau_{F,X}} \grmptb \gnl
\gob{1}{F'} \gob{1}{X} \gob{1}{F}
\gend\qquad\text{and}\qquad
\tau_{F'X,F}=
\gbeg{3}{4}
\got{1}{F'} \got{1}{X} \got{1}{F} \gnl
\gcl{1} \glmptb \gnot{\hspace{-0,34cm}\tau_{X,F}} \grmptb \gnl
\glmptb \gnot{\hspace{-0,34cm}\tau_{F', F'}} \grmptb \gcl{1} \gnl
\gob{1}{F'} \gob{1}{F'} \gob{1}{X}
\gend
$$
(in other words, $\tau_{F,X}$ is left monadic and right comonadic and natural with respect to 
$\gbeg{2}{1}
\glm \gnl
\gend$ and 
$\gbeg{2}{1}
\glcm \gnl
\gend$, and 
$\tau_{X,F}$ is left comonadic and right monadic and natural with respect to 
$\gbeg{2}{1}
\glm \gnl
\gend$ and 
$\gbeg{2}{1}
\glcm \gnl
\gend$), 

and the following compatibility conditions among the data in the quintuple and $\tau_{F,F}$ and $\tau_{F',F'}$ hold: 
$$
\gbeg{3}{5}
\got{1}{F'} \got{1}{X} \got{1}{F} \gnl
\gcl{1} \glmptb \gnot{\hspace{-0,34cm}\tau_{X,F}} \grmptb \gnl
\glmptb \gnot{\hspace{-0,34cm}\tau_{F',F'}} \grmptb \gcl{1} \gnl
\gcl{1} \glmptb \gnot{\hspace{-0,34cm}\tau_{F,X}} \grmptb \gnl
\gob{1}{F'} \gob{1}{X} \gob{1}{F}
\gend=
\gbeg{3}{5}
\got{1}{F'} \got{1}{X} \got{1}{F} \gnl
\glmptb \gnot{\hspace{-0,34cm}\tau_{F,X}} \grmptb \gcl{1} \gnl
\gcl{1} \glmptb \gnot{\hspace{-0,34cm}\tau_{F',F'}} \grmptb \gnl
\glmptb \gnot{\hspace{-0,34cm}\tau_{X,F}} \grmptb \gcl{1} \gnl
\gob{1}{F'} \gob{1}{X} \gob{1}{F}
\gend\hspace{3cm}
\gbeg{3}{8}
\got{2}{F'}  \got{1}{X}  \gnl
\gcmu \gcl{1} \gnl
\gcl{1} \glmptb \gnot{\hspace{-0,34cm}\tau_{F,X}} \grmptb \gnl
\glm \gcl{2} \gnl 
\glcm \gnl
\gcl{1} \glmptb \gnot{\hspace{-0,34cm}\tau_{X,F}} \grmptb \gnl
\gmu \gcl{1} \gnl
\gob{2}{F'} \gob{1}{X} 
\gend=
\gbeg{3}{5}
\got{2}{F'} \gvac{1} \got{1}{X} \gnl
\gcmu \glcm  \gnl
\gcl{1} \glmptb \gnot{\hspace{-0,34cm}\tau_{F,F}} \grmptb \gcl{1} \gnl
\gmu \glm \gnl 
\gob{2}{F'} \gvac{1} \gob{1}{X} 
\gend
$$

\underline{2-cells:} are the same as in $\Bimnd(\K)$, they are 2-cells in $\Mnd(\K)$ and $\Comnd(\K)$ simultaneously (recall that they are left $F$-linear and left $F$-colinear). 

\underline{The composition of 1-cells:} given $(X,\tau_{F,X}, \tau_{X,F}, \nu_X, l_X): (\A, F, \tau_{F,F})\to(\A', F', \tau_{F',F'})$ and 
$(Y,\tau_{F,Y}, \tau_{Y,F}, \nu_Y, l_Y): (\A', F', \tau_{F',F'})\to(\A'', F'', \tau_{F'',F''})$, their composition is a quintuple 
$(YX,\tau_{F,YX}, \tau_{YX,F}, \nu_{YX}, l_{YX}): (\A, F, \tau_{F,F})\to(\A'', F'', \tau_{F'',F''})$ where 
$$
\tau_{F,YX}
\gbeg{4}{4}
\got{1}{F''} \got{1}{Y} \got{1}{X} \gnl
\glmptb \gnot{\hspace{-0,34cm}\tau_{F', Y}} \grmptb \gcl{1} \gnl
\gcl{1} \glmptb \gnot{\hspace{-0,34cm}\tau_{F,X}} \grmptb \gnl
\gob{1}{Y} \gob{1}{X} \gob{1}{F}
\gend, \quad
\tau_{YX,F}=
\gbeg{3}{4}
\got{1}{Y} \got{1}{X} \got{1}{F} \gnl
\gcl{1} \glmptb \gnot{\hspace{-0,34cm}\tau_{X,F}} \grmptb \gnl
\glmptb \gnot{\hspace{-0,34cm}\tau_{Y, F'}} \grmptb \gcl{1} \gnl
\gob{1}{F''} \gob{1}{Y} \gob{1}{X}
\gend, \quad
\nu_{YX}=
\gbeg{4}{5}
\got{2}{F''} \got{1}{Y} \got{1}{X} \gnl
\gcmu \gcl{1} \gcl{3} \gnl
\gcl{1} \glmptb \gnot{\hspace{-0,34cm}\tau_{F', Y}} \grmptb \gnl
\glm \glmo \gnl
\gvac{1} \gob{1}{Y} \gob{3}{X}
\gend, \quad
l_{YX}=
\gbeg{3}{5}
\gvac{1} \got{1}{Y} \got{3}{X} \gnl
\glcm \grmo \gvac{1} \gcl{1} \gnl
\gcl{1} \glmptb \gnot{\hspace{-0,34cm}\tau_{Y, F'}} \grmptb \gcl{1} \gnl
\gmu \gcl{1} \gcl{1} \gnl
\gob{2}{F''} \gob{1}{Y} \gob{1}{X.} 
\gend
$$

\bigskip

The above two Propositions allow to define the embedding 2-functor 
\begin{equation} \eqlabel{left embedding}
\tau\x\Bimnd^l(\K)\to\Bimnd^l(\K)
\end{equation}  
(here $\Bimnd^l(\K)$ is the 2-category that so far we denoted by $\Bimnd(\K)$, we now stress that this is the 2-category of {\em left} bimonads in $\K$). 
It sends a $\tau$-bimonad $(F, \tau_{F,F})$ to the left bimonad $(F, \lambda)$ and a 1-cell $(X,\tau_{F,X}, \tau_{X,F}, \nu_X, l_X)$ from $\tau\x\Bimnd^l(\K)$ 
to the left strong Yetter-Drinfel`d module $(X, \psi_{F,X}, \phi_{X,F})$ where $\lambda, \psi_{F,X}$ and $\phi_{X,F}$ are given by \equref{YD concrete}.

\medskip

On the other hand, classical Yetter-Drinfel`d modules (over a filed $k$) are endomorphism 1-cells in $\tau\x\Bimnd(\K)$ when $\K$ is induced by the braided monoidal category of vector spaces. 
This inspires:

\begin{defn}
A {\em classical Yetter-Drinfel'd module in $\K$} is an endomorphism 1-cell in $\tau\x\Bimnd(\K)$ (left and right versions). 
\end{defn}

Now it is clear that we have:

\begin{cor}
The category of classical left Yetter-Drinfel'd modules over a bimonad $F$ in $\K$ is ${}^F_F\YD_{cl}(\K,\A):=\tau\x\Bimnd^l(\K)(F)$ and it is monoidal. 
Moreover, there is a monoidal embedding of categories: 
$$\tau\x\Bimnd^l(\K)(F)\to\Bimnd^l(\K)(F)$$ 
{\em i.e.}
$${}^F_F\YD_{cl}(\K,\A)\to{}^F_F\YD_s(\K,\A).$$
\end{cor}

\bigskip


We next want to apply the results of \thref{YD action} to the above setting. As before, assume that  $B$ is a monad and that there is a 1-cell $(F, \tau_{B,F})$ in $\Tau(\A, B)$. 
Let ${}_{\tau_B}{}^F_F\YD_{cl}(\K,\A)$ denote the following category. 
Its objects are classical Yetter-Drinfel`d modules $(X, \psi_X, \phi_X)$ for which there is a 2-cell $\tau_{B,X}: BX\to XB$ so that 
$(X, \tau_{B,X})$ is a 1-cell in $\Tau(\A, B)$ and the following Yang-Baxter equations hold true:
\begin{center} 
\begin{tabular} {p{6cm}p{0cm}p{6.8cm}} 
\begin{equation} \eqlabel{YBE tau BFX}
\gbeg{3}{5}
\got{1}{B} \got{1}{F} \got{1}{X} \gnl
\gcl{1} \glmptb \gnot{\hspace{-0,34cm}\tau_{F,X}} \grmptb \gnl
\glmptb \gnot{\hspace{-0,34cm}\tau_{B,X}} \grmptb \gcl{1} \gnl
\gcl{1} \glmptb \gnot{\hspace{-0,34cm}\tau_{B,F}} \grmptb \gnl
\gob{1}{X} \gob{1}{F} \gob{1}{B}
\gend=
\gbeg{3}{5}
\got{1}{B} \got{1}{F} \got{1}{X} \gnl
\glmptb \gnot{\hspace{-0,34cm}\tau_{B,F}} \grmptb \gcl{1} \gnl
\gcl{1} \glmptb \gnot{\hspace{-0,34cm}\tau_{B,X}} \grmptb \gnl
\glmptb \gnot{\hspace{-0,34cm}\tau_{F,X}} \grmptb \gcl{1} \gnl
\gob{1}{X} \gob{1}{F} \gob{1}{B}
\gend
\end{equation} & & 
\begin{equation}\eqlabel{YBE tau BXF}
\gbeg{3}{5}
\got{1}{B} \got{1}{X} \got{1}{F} \gnl
\gcl{1} \glmptb \gnot{\hspace{-0,34cm}\tau_{X,F}} \grmptb \gnl
\glmptb \gnot{\hspace{-0,34cm}\tau_{B,F}} \grmptb \gcl{1} \gnl
\gcl{1} \glmptb \gnot{\hspace{-0,34cm}\tau_{B,X}} \grmptb \gnl
\gob{1}{F} \gob{1}{X} \gob{1}{B}
\gend=
\gbeg{3}{5}
\got{1}{B} \got{1}{X} \got{1}{F} \gnl
\glmptb \gnot{\hspace{-0,34cm}\tau_{B,X}} \grmptb \gcl{1} \gnl
\gcl{1} \glmptb \gnot{\hspace{-0,34cm}\tau_{B,F}} \grmptb \gnl
\glmptb \gnot{\hspace{-0,34cm}\tau_{X,F}} \grmptb \gcl{1} \gnl
\gob{1}{F} \gob{1}{X} \gob{1}{B.}
\gend
\end{equation} 
\end{tabular}
\end{center}
Here $\nu_X:(FX,\tau_{B,FX})\to (X,\tau_{B,X}),l_X: (X,\tau_{B,X})\to (FX,\tau_{B,FX}), l_B:(B, \tau_{F,B})\to(FB, \tau_{F,FB})\in\Tau(\A, B)$ 
($\tau_{B,X}$ is natural with respect to the left $F$-action and left $F$-coaction on $X$, and $\tau_{B,F}$ with respect to the left $F$-coaction on $B$).   
The morphisms of ${}_{\tau_B}{}^F_F\YD_{cl}(\K,\A)$ are the morphisms in $\Mnd(\K)(F)$ and in $\Comnd(\K)(F)$. 
Then ${}_{\tau_B}{}^F_F\YD_{cl}(\K,\A)$ is a monoidal category.

\begin{prop}
Let $F$ be a left bimonad in $\K$, $B$ a left $F$-comodule monad with $\psi_{B,F}$ given by \equref{psi_3 for BF}. 
Let $(X, \psi_{F,X}, \phi_{X,F})$ be an object in ${}_{\tau_B}{}^F_F\YD_{cl}(\K,\A)$ and let $\psi_{B,X}$ be as in \equref{psi_3 for X}. 
Then the Yang-Baxter equations \equref{YBE psi BFX} and \equref{YBE BXF} are fulfilled. 

Consequently, 
there is a monoidal embedding of categories ${}_{\tau_B}{}^F_F\YD_{cl}(\K,\A)\to {}_{\psi_B}{}^F_F\YD_s(\K,\A)$. 
\end{prop}

\begin{proof}
The proof that the identities \equref{YBE psi BFX} and \equref{YBE BXF} hold is analogous to the proof in \prref{YD are strong} that the relation \equref{psi-lambda-phi} holds. 
For the proof of \equref{YBE psi BFX} apply the following changes in the arguments in the proof: 
coassociativity of $F$ $\leadsto$ $F$-comodule law, comonadic distributive law property $\leadsto$ naturality with respect to the $F$-comodule structure, \equref{YBE tau FXF} $\leadsto$ 
\equref{YBE tau BFX}, whereas to prove \equref{YBE BXF} change additionally: associativity of $F$ $\leadsto$ $F$-module law, monadic distributive law property $\leadsto$ naturality with 
respect to the $F$-module structure, and change \equref{YBE tau FXF} $\leadsto$ \equref{YBE tau BXF}. 
\qed\end{proof}

Let us denote by ${}_{\psi_B}{}^F_F\u{\YD_s(\K,\A)}$ the image of ${}_{\tau_B}{}^F_F\YD_{cl}(\K,\A)$ by the latter embedding functor.

\begin{cor} \colabel{known action Rob-Mar}
Under the hypotheses of the previous Proposition there is an action of categories ${}_{\psi_B}{}^F_F\u{\YD_s(\K,\A)}\times {}^F_B\K\to{}^F_B\K$ as described in \thref{YD action}. 
\end{cor}

\begin{ex}
When $\K$ is induced by a braided monoidal category $\C$, our \coref{known action Rob-Mar} recovers \cite[Theorem 2.3]{RL}. 
\end{ex}

If we forget the left $F$-comodule action both on Yetter-Drinfel`d modules and on $B$, the 2-cells $\psi_{B,X}, \psi_{F,X}$ and $\phi_{X,F}$ pass to be the 2-cells 
$\tau_{B,X}, \tau_{F,X}$ and $\tau_{X,F}$, respectively. Moreover, the category ${}_{\psi_B}{}^F_F\u{\YD_s(\K,\A)}$ becomes ${}_{\tau_B}{}_{(\A,F)}\Tau$. 
On the other hand, if we consider a left bimonad $F$ in $\K$ as a quasi-bimonad (with $\Phi=\eta_F\times\eta_F\times\eta_F$) and take $\sigma$ to be trivial in \coref{known action Rob-Mar}, 
then we see that the latter result and \coref{Martin} are generalizations in two different directions of one and the same result. 

\bigskip

In the right hand-side version of \equref{left embedding} we have an embedding of 2-categories $\tau\x\Bimnd^r(\K)\to\Bimnd^r(\K)$ determined by:  
\begin{equation} \eqlabel{YD-right concrete}
\lambda=
\gbeg{4}{5}
\got{1}{F} \got{2}{F} \gnl
\gcl{1} \gcmu \gnl
\glmptb \gnot{\hspace{-0,34cm}\tau_{F,F}} \grmptb \gcl{1} \gnl
\gcl{1} \gmu \gnl
\gob{1}{F} \gob{2}{F} 
\gend,\quad
\psi'_{X,F}=
\gbeg{4}{5}
\got{1}{X} \got{2}{F} \gnl
\gcl{1} \gcmu \gnl
\glmptb \gnot{\hspace{-0,34cm}\tau_{X,F}} \grmptb \gcl{1} \gnl
\gcl{1} \grm \gnl
\gob{1}{F} \gob{1}{X} 
\gend, \quad
\phi'_{F,X}=
\gbeg{4}{5}
\got{1}{F} \got{1}{X} \gnl
\gcl{1} \grcm \gnl
\glmptb \gnot{\hspace{-0,34cm}\tau_{F,X}} \grmptb \gcl{1} \gnl
\gcl{1} \gmu \gnl
\gob{1}{X} \gob{2}{F} 
\gend
\end{equation}
which induces an embedding of monoidal categories $\tau\x\Bimnd^r(\K)(F)\to\Bimnd^r(\K)(F)$, that is: 
$$\YD_{cl}(\K,\A)^F_F\to\YD_s(\K,\A)^F_F.$$
For a monad $B$ and supposing that there is a 1-cell $(F, \tau_{B,F})$ in $\Tau(\A, B)$, the category  ${}_{\tau_B}\YD_{cl}(\K,\A)^F_F$ is defined analogously as 
${}_{\tau_B}{}^F_F\YD_{cl}(\K,\A)$. Its objects are classical right Yetter-Drinfel`d modules $(X, \psi'_{X, F}, \phi'_{F,X})$ for which there is a 2-cell 
$\tau_{B,X}: BX\to XB$ so that $(X, \tau_{B,X})$ is a 1-cell in $\Tau(\A, B)$ and the same Yang-Baxter equations \equref{YBE tau BFX} and \equref{YBE tau BXF} hold true.

\begin{thm} \thlabel{known action Rob-Sch}
Let $F$ be a right bimonad in $\K$, $B$ a right $F$-module monad (in the sense of \equref{F mod alg}--\equref{F mod alg unit}). We have the following: 
\begin{enumerate}
\item Let $(X, \psi'_{X, F}, \phi'_{F,X})$ be an object in ${}_{\tau_B}\YD_{cl}(\K,\A)^F_F$ and let 
$\psi_{B,X}$ and $\psi_{B,F}$ be as in \equref{psi_2 for X} and \equref{Sw datum}. 
Then the Yang-Baxter equations \equref{YBE phi' BFX} and \equref{YBE BXF'} are fulfilled. 
\item There is a monoidal embedding of categories ${}_{\tau_B}\YD_{cl}(\K,\A)^F_F\to {}_{\psi_B}\YD_s(\K,\A)^F_F$. 
\item There is an action of categories: 
$${}_{\psi_B}\u{\YD_s(\K,\A)^F_F}\times {}^F_B\K\to{}^F_B\K$$ 
as described in \thref{YD-right action}, where ${}_{\psi_B}\u{\YD_s(\K,\A)^F_F}$ denotes the image of ${}_{\tau_B}\YD_{cl}(\K,\A)^F_F$ by the above embedding functor. 
\end{enumerate}
\end{thm}

Forgetting the right $F$-module action both on Yetter-Drinfel`d modules and on $B$ 
we obtain the forgetful functor ${}_{\psi_B}\u{\YD_s(\K,\A)}^F_F\to {}_{\tau_B}\Tau^{(\A,F)}$. Taking $\Phi_{\lambda}$ in \coref{Sch-Balan} to be trivial and 
considering a right bimonad $F$ in $\K$ as a coquasi-bimonad (with $\omega=\Epsilon_F\times\Epsilon_F\times\Epsilon_F$), we see that \coref{Sch-Balan} and 
\thref{known action Rob-Sch} are generalization of the same result in two different directions. 

\section{Actions of monoidal categories as pseudofunctors} \selabel{pseudo}

In this last section we want to show how the action theorems \thref{main} and \thref{YD action} come out from a more general situation 
that occurs in 2-categories. Let $\Ll$ be a bicategory and $\Cat$ the 2-category of categories. Given a pseudofunctor $\Tau: \Ll\to\Cat$ there is a 
``pre-monoidal'' functor $\Tau_{\A, \B}: \Ll(\A, \B)\to\Fun(\Tau(\A), \Tau(\B))$ for every pair of 0-cells $\A, \B$ in $\Ll$, where 
$\Fun(\Tau(\A), \Tau(\B))$ is the category of functors from $\Tau(\A)$ to $\Tau(\B)$. The adjective 
``pre-monoidal'' here means that the functors $\Tau_{-, -}$ are compatible with the composition of composable 1-cells. 
For $\A=\B$ we have a monoidal category $\Ll(\A)=\Ll(\A, \A)$, a strict monoidal category $\Fun(\Tau(\A), \Tau(\A))$ and a monoidal functor 
$\Tau_{\A}=\Tau_{\A, \A}$ between them. 
This in turn yields an action of the monoidal category $\Ll(\A)$ on the category $\Tau(\A)$. Resuming we have:  
$$ \qquad \Tau: \Ll\to\Cat \quad pseudofunctor \quad \Rightarrow$$
\begin{equation} \eqlabel{pseudo ->action}
\Tau_{\A}: \Ll(\A)\to\Fun(\Tau(\A), \Tau(\A)) \quad \Leftrightarrow \quad \Ll(\A)\times\Tau(\A)\to\Tau(\A). \quad
\end{equation} 
$$ \textnormal{ \footnotesize \hspace{-0,8cm} monoidal functor}  \hspace{6cm}  \textnormal{\footnotesize action} $$ 
We are going to show that the action of categories in \thref{main} and \thref{YD action} come from two pseudofunctors $\Tau: \Ll\to\Cat$ for appropriate 
bi/2-categories $\Ll$.

\subsection{A 2-functor from the 2-category of Tambara bimonads}

We define the 2-category of Tambara bimonads ${}_{\psi}\Bimnd(\K)$ as follows. 

\underline{0-cells:} are quadriples $(\A, F, B, \psi_{B,F})$ where $(\A, F)$ is a bimonad, $(\A, B)$ a monad in $\K$ and $(F, \psi_{B,F})$ is a 1-cell in $\Mnd(\K)$. 

\underline{1-cells:} are quadriples $(X,\psi_{F,X},\phi_{X,F}, \psi_{B,X}): (\A, F, B, \psi_{B,F})\to(\A', F', B', \psi_{B',F'})$ where: 
$(X,\psi_{F,X},\phi_{X,F})$ is a 1-cell in $\Bimnd(\K)$, recall \equref{psi-lambda-phi 2-bimonad}, and $(X, \psi_{B,X})$ a 1-cell in $\Mnd(\K)$ such that the conditions: 
\begin{center} 
\begin{tabular} {p{6cm}p{0cm}p{6.8cm}} 
\begin{equation} \eqlabel{YBE psi BFX gen}
\gbeg{3}{5}
\got{1}{B'} \got{1}{F'} \got{1}{X} \gnl
\gcl{1} \glmptb \gnot{\hspace{-0,34cm}\psi_{F,X}} \grmptb \gnl
\glmptb \gnot{\hspace{-0,34cm}\psi_{B,X}} \grmptb \gcl{1} \gnl
\gcl{1} \glmptb \gnot{\hspace{-0,34cm}\psi_{B,F}} \grmptb \gnl
\gob{1}{X} \gob{1}{F} \gob{1}{B}
\gend=
\gbeg{3}{5}
\got{1}{B'} \got{1}{F'} \got{1}{X} \gnl
\glmptb \gnot{\hspace{-0,34cm}\psi_{B',F'}} \grmptb \gcl{1} \gnl
\gcl{1} \glmptb \gnot{\hspace{-0,34cm}\psi_{B,X}} \grmptb \gnl
\glmptb \gnot{\hspace{-0,34cm}\psi_{F,X}} \grmptb \gcl{1} \gnl
\gob{1}{X} \gob{1}{F} \gob{1}{B}
\gend
\end{equation} & & 
\begin{equation}\eqlabel{YBE BXF gen}
\gbeg{3}{5}
\got{1}{B'} \got{1}{X} \got{1}{F} \gnl
\gcl{1} \glmptb \gnot{\hspace{-0,34cm}\phi_{X,F}} \grmptb \gnl
\glmptb \gnot{\hspace{-0,34cm}\psi_{B',F'}} \grmptb \gcl{1} \gnl
\gcl{1} \glmptb \gnot{\hspace{-0,34cm}\psi_{B,X}} \grmptb \gnl
\gob{1}{F'} \gob{1}{X} \gob{1}{B}
\gend=
\gbeg{3}{5}
\got{1}{B'} \got{1}{X} \got{1}{F} \gnl
\glmptb \gnot{\hspace{-0,34cm}\psi_{B,X}} \grmptb \gcl{1} \gnl
\gcl{1} \glmptb \gnot{\hspace{-0,34cm}\psi_{B,F}} \grmptb \gnl
\glmptb \gnot{\hspace{-0,34cm}\phi_{X,F}} \grmptb \gcl{1} \gnl
\gob{1}{F'} \gob{1}{X} \gob{1}{B}
\gend
\end{equation} 
\end{tabular}
\end{center}
hold. Observe that the 2-cells involved are: $\psi_{F,X}: F\s'X\to XF,\phi_{X,F}: XF\to F\s'X, \psi_{B,X}: B'X\to XB, \psi_{B,F}: BF\to FB, \psi_{B',F\s'}: B'F\s'\to F\s'B'$. 

\medskip

\underline{2-cells:} $\zeta: (X,\psi_{F,X},\phi_{X,F}, \psi_{B,X})\to (Y,\psi_{F,Y},\phi_{Y,F}, \psi_{B,Y})$ where 
$\zeta: (X,\psi_{F,X},\phi_{X,F})\to (Y,\psi_{F,Y},\phi_{Y,F})$ is a 2-cell in $\Bimnd(\K)$ and $\zeta: (X, \psi_{B,X})\to (Y, \psi_{B,Y})$ a 2-cell in $\Mnd(\K)$.

\bigskip

We now define a 2-functor 
$$\Tau: {}_{\psi}\Bimnd(\K)\to\Cat$$ 
as follows. It sends a 0-cell $(\A, F, B, \psi_{B,F})$ to the category ${}^F_B\K$ (from \deref{relative}). 

Given two 0-cells $(\A, F, B, \psi_{B,F})$ and $(\A', F\s', B', \psi_{B',F\s'})$ in ${}_{\psi}\Bimnd(\K)$ we define a functor 
\begin{equation} \eqlabel{T F,F'}
\Tau_{F,F\s'}: {}_{\psi}\Bimnd(\K)(F, F\s')\to\Fun({}^F_B\K, {}^{F\s'}_{B'}\K)
\end{equation}
$$\hspace{2cm} X\mapsto\Tau_X=X\x$$
$$\hspace{2cm} \zeta\mapsto\Tau_{\zeta}=\zeta\times\x$$
like this. An object $(X,\psi_{F,X},\phi_{X,F}, \psi_{B,X})$ the functor $\Tau_{F, F\s'}$ sends to the functor $\Tau_X: {}^F_B\K \to {}^{F\s'}_{B'}\K$, where $\Tau_X(M)=XM$, for $M\in {}^F_B\K$, 
so that $XM$ is endowed with the following structures of a left $B'$-module and $F\s'$-comodule: 
\begin{center} 
\begin{tabular}{p{5cm}p{0,8cm}p{5cm}}
\begin{equation} \eqlabel{B act XY psi gen}
\gbeg{3}{4}
\got{1}{B'} \got{3}{XM} \gnl
\gcn{1}{1}{1}{3} \gvac{1} \gcl{1} \gnl
\gvac{1} \glm \gnl
\gvac{2} \gob{1}{XM} 
\gend=
\gbeg{3}{4}
\got{1}{B'} \got{1}{X} \got{1}{M} \gnl
\glmptb \gnot{\hspace{-0,34cm}\psi_{B,X}} \grmptb \gcl{1} \gnl
\gcl{1} \glm \gnl
\gob{1}{X} \gob{3}{M}
\gend 
\end{equation} & & \vspace{0,1cm}
\begin{equation*} 
\gbeg{3}{4}
\got{5}{XM} \gnl
\gvac{1} \glcm \gnl
\gcn{1}{1}{3}{1} \gvac{1} \gcl{1} \gnl
\gob{1}{F'} \gob{3}{XM} 
\gend=
\gbeg{3}{4}
\got{1}{X} \got{1}{} \got{1}{M} \gnl
\gcl{1} \glcm \gnl
\glmptb \gnot{\hspace{-0,34cm}\phi_{X,F}} \grmptb \gcl{1} \gnl
\gob{1}{F'} \gob{1}{X} \gob{1}{M.}
\gend
\end{equation*}
\end{tabular}
\end{center}
That $XM\in {}^{F\s'}_{B'}\K$ is proved analogously as in \thref{YD action}: 
$$
\gbeg{3}{6}
\got{1}{B'} \got{3}{XM}  \gnl
\gcn{2}{1}{1}{3} \gcl{1} \gnl
\gvac{1} \glm \gnl
\gvac{1} \glcm \gnl
\gcn{2}{1}{3}{1} \gcl{1} \gnl
\gob{1}{F'} \gob{3}{XM}
\gend=
\gbeg{3}{6}
\got{1}{B'} \got{1}{X} \got{1}{M} \gnl
\glmptb \gnot{\hspace{-0,34cm}\psi_{B,X}} \grmptb \gcl{1} \gnl
\gcl{1} \glm \gnl
\gcl{1} \glcm \gnl
\glmptb \gnot{\hspace{-0,34cm}\phi_{X,F}} \grmptb \gcl{1} \gnl
\gob{1}{F'} \gob{1}{X} \gob{1}{M}
\gend\stackrel{M \s rel.}{=}
\gbeg{4}{5}
\got{1}{B'} \got{1}{X} \got{3}{M} \gnl
\glmptb \gnot{\hspace{-0,34cm}\psi_{B,X}} \grmptb \glcm \gnl
\gcl{1} \glmptb \gnot{\hspace{-0,34cm}\psi_{B,F}} \grmptb \gcl{1} \gnl
\glmptb \gnot{\hspace{-0,34cm}\phi_{X,F}} \grmptb \glm \gnl
\gob{1}{F'} \gob{1}{X} \gob{3}{M}
\gend\stackrel{\equref{YBE BXF gen}}{=}
\gbeg{4}{7}
\got{1}{B'} \got{1}{X} \got{3}{M}  \gnl
\gcl{1} \gcl{1} \glcm \gnl
\gcl{1} \glmptb \gnot{\hspace{-0,34cm}\phi_{X,F}} \grmptb \gcl{2} \gnl
\glmptb \gnot{\hspace{-0,34cm}\psi_{B',F\s'}} \grmptb \gcl{1} \gnl
\gcl{1}  \glmptb \gnot{\hspace{-0,34cm}\psi_{B,X}} \grmptb \gcl{1} \gnl
\gcl{1} \gcl{1} \glm \gnl
\gob{1}{F'} \gob{1}{X} \gob{3}{M}
\gend =
\gbeg{3}{5}
\got{1}{B'} \got{3}{XM}  \gnl
\gcl{1} \glcm \gnl
\glmptb \gnot{\hspace{-0,34cm}\psi_{B,F}} \grmptb \gcl{1} \gnl
\gcl{1} \glm \gnl
\gob{1}{F'} \gob{3}{XM.}
\gend 
$$
A morphism $F:M\to N$ in ${}^F_B\K$ the functor $\Tau_X$ sends to the morphism $X\times F:XM\to XN$, the proof that it is a morphism in 
${}^{F\s'}_{B'}\K$ is straightforward. 

A morphism $\zeta: (X,\psi_{F,X},\phi_{X,F}, \psi_{B,X})\to (Y,\psi_{F,Y},\phi_{Y,F}, \psi_{B,Y})$ in ${}_{\psi}\Bimnd(\K)(F, F\s')$ the functor 
$\Tau_{F,F\s'}$ sends to the natural transformation $\Tau_\zeta: \Tau_X\to\Tau_Y: {}^F_B\K \to {}^{F\s'}_{B'}\K$. Given $M\in {}^F_B\K$ we have 
a morphism $\Tau_\zeta(M)=\zeta\times M: XM\to YM$. It is $B$-linear, since $(X,\psi_{B,X})$ is a 1-cell in $\Mnd(\K)$, and it is $F$-colinear, 
since $(X,\phi_{X,F})$ is a 1-cell in $\Comnd(\K)$. The morphism $\zeta\times M$ is trivially natural in $M$. We also clearly have: $\Tau_{F, F\s'}(\xi\zeta)=
\Tau_{F, F\s'}(\xi)\Tau_{F, F\s'}(\zeta)$ and $\Tau_{F, F\s'}(id_X)=\Tau_{id_X}=\Id_{\Tau_X}$. 

\medskip

Given two composable 1-cells in ${}_{\psi}\Bimnd(\K)$: 
$$(\A, F, B, \psi_{B,F}) \stackrel{ (X,\psi_{F,X},\phi_{X,F}, \psi_{B,X}) } {\longrightarrow} (\A', F\s', B', \psi_{B',F'}) 
\stackrel{ (Y,\psi_{F\s',Y},\phi_{Y,F\s'}, \psi_{B',Y}) } {\longrightarrow} (\A'', F\s'', B'', \psi_{B'',F\s''})$$ 
we take the identities for the natural isomorphisms 
$$s_{Y,X}: \Tau_{F\s', F\s''}(Y)\comp\Tau_{F, F\s'}(X) \to \Tau_{F, F\s''}(YX), \qquad 
s_0: \Id_{\Tau(F)} \to \Tau_{F, F}(id_F)$$
which actually map 
$$s_{Y,X}: \Tau_Y\comp\Tau_X\to\Tau_{YX}, \qquad s_0: \Id_{{}^F_B\K}\to\Tau_{id_F}.$$

\begin{thm} \thlabel{pseudo bimonad}
The above defines a 2-functor $\Tau: {}_{\psi}\Bimnd(\K)\to\Cat$. Fixing a 0-cell $(\A, F, B, \psi_{B,F})$ in ${}_{\psi}\Bimnd(\K)$ by \equref{pseudo ->action} 
we recover the action of categories ${}_{\psi_B}{}^F_F\YD_s(\K,\A)\times {}^F_B\K\to{}^F_B\K$ from \thref{YD action}. 
\end{thm}

\subsection{A pseudofunctor from a bicategory over the 2-category of monads} \sslabel{pseudo}

Let $\K$ be a 2-category and $\Ll$ a bicategory. We say that $\Ll$ is a {\em bicategory over the 2-category of monads}, if there is a 
{\em faithful quasi 2-functor} $\u{\U}: \Ll\to\Mnd(\K)$ and a {quasi 2-functor} $\u{\F}: \Ll\to\Mnd(\K)$ that factorizes through $\u{\U}$. 
The adjective ``quasi'' means that the 2-functor does not satisfy the monoidal coherence axiom, while ``faithful'' refers to the fact that each functor on 
hom-categories is faithful. We fix the following notation: 
\begin{center} 
\begin{tabular}{p{0cm}p{0cm}p{6cm}}  
\begin{equation} \eqlabel{F-U-L triangle} \hspace{-6cm}
\qtriangle<1`1`1;500>[\Ll`\Mnd(\K)`\Mnd(\K); \u{\U}`\u{\F}`\u{\iota}]  
\end{equation} & & 
\begin{equation*}
\qtriangle<2`2`2;500>[\crta B`B`B; ``]
\hspace{1,4cm} \qtriangle<2`2`2;500>[\crta X`(X, \tau_{B,X})`(X, \psi_{B,X}); ``] \hspace{1,4cm} \qtriangle<2`2`2;500>[\crta\zeta`\zeta`\zeta; ``]
\put(-2220,300){\fbox{{\tiny 0-cells}}}
\put(-1310,300){\fbox{{\tiny 1-cells}}}
\put(-270,300){\fbox{{\tiny 2-cells}}}
\end{equation*}
\end{tabular}
\end{center}
where $B:\A\to\A, X:\A\to\A', \tau_{B,X}:B'X\to XB$, and define a pseudofunctor 
$$\Tau: \Ll\to \Cat$$ 
on 0-cells:
$$\crta B\mapsto {}_B\K$$ 
and for 0-cells $\crta B, \crta B'$ in $\Ll$ we define the functor 
$$ \Tau_{\crta B, \crta B'}: \Ll(\crta B, \crta B')\to \Fun({}_B\K, {}_{B'}\K)$$
on objects: 
$$ \hspace{2cm}\crta X \quad \mapsto \quad \Tau_X: {}_B\K\to {}_{B'}\K$$
$$\hspace{4,7cm} M\mapsto XM$$
$$\hspace{6,8cm} (f:M\to N)\mapsto (X\times f: X\times M\to X\times N).$$
and on morphisms: 
$$ \hspace{2cm}\crta\zeta:\crta X\to \crta Y \quad \mapsto \quad \Tau_\zeta: \Tau_X\to \Tau_Y: {}_B\K\to {}_{B'}\K$$
$$\hspace{4,7cm} \Tau_\zeta(M)=\zeta\times M: XM\to YM$$  
analogously as in \equref{T F,F'}. Here $XM$ is a left $B'$-module as in \equref{B act XY psi gen}, with $\psi_{B,X}$ from $\u{\F}(\crta X)=(X, \psi_{B,X})$, 
as depicted in the third diagram above. As in the previous subsection we have that the functor $\Tau_{\crta B, \crta B'}$ is well-defined.

For two composable 1-cells in $\Ll$: $\crta B \stackrel{ \crta X } {\longrightarrow} \crta B' \stackrel{ \crta Y } {\longrightarrow} \crta B'' $ 
we consider the natural isomorphisms 
$$s_{Y,X}: \Tau_{\crta B', \crta B''}(Y)\comp\Tau_{\crta B, \crta B'}(X) \to \Tau_{\crta B, \crta B''}(YX), \qquad 
s_0: \Id_{\Tau(\crta B)} \to \Tau_{\crta B,\crta B}(id_{\crta B})$$
which actually map 
$$s_{Y,X}: \Tau_Y\comp\Tau_X\to\Tau_{YX}, \qquad s_0: \Id_{{}_B\K}\to\Tau_{id_{\crta B}}$$
and should satisfy a commutative hexagon and two triangles. Observe that $\Tau_X(M)=XM$ and that $(\Tau_Y\comp\Tau_X)(M)=Y(XM)$, then 
the arrow $\Tau_Z\comp(\Tau_Y\comp\Tau_X)\to(\Tau_Z\comp \Tau_Y)\comp\Tau_X$ in the hexagon is identity and we get a commutative pentagon. Setting $r_{Y,X,M}:=s_{Y,X}(M)$ 
this pentagon amounts to the identity \equref{pentagon new} between morphisms in ${}_{B'}\K$ (with the difference that $X,Y,Z$ there are 1-endocells, and here we have: 
$\A \stackrel{ X } {\longrightarrow} \A' \stackrel{ Y } {\longrightarrow} \A'' \stackrel{ Z } {\longrightarrow} \A'''$).  
We consider $s_0$ to be identity, 
thus the two triangles become: $r_{I,X,M}=id_{XM}=r_{X,I,M}$. Having a natural isomorphism $s_{Y,X}$ such that the above-mentioned conditions are fulfilled, we obtain 
a pseudofunctor $\Tau: \Ll\to \Cat$. 
As we proved in \thref{main}, considering relations \equref{ro asoc isom novo} and \equref{ro asoc isom}, the conditions that $s_{Y,X}$, {\em i.e.} $r_{Y,X,M}$ should 
satisfy express the existence of an invertible 2-cocycle in $\EM^M(\K)$. 

Fixing a 0-cell $\crta B$ in $\Ll$, we have a quasi-monoidal faithful functor $\U$ from the monoidal category $\C=\Ll(\crta B, \crta B)$ to 
the strict monoidal category $\Mnd(B,B)=\Tau(\A, B)$, and another quasi-monoidal functor $\F:\C\to \Tau(\A, B)$ factorizing through $\U$. 

\begin{thm}  \thlabel{pseudo bicat}
Let $\u{\U}, \u{\F}, \Ll$ be as in \equref{F-U-L triangle}. The above defines a pseudofunctor $\Tau: \Ll\to \Cat$. Fix $\crta B\in\Ll$, then 
\equref{pseudo ->action} yields the action of categories $\C\times {}_B\K\to {}_B\K$ characterized in \thref{main}. 
\end{thm}

\subsection{A pseudofunctor from the bicategory of Tambara modules over a quasi-bimonad}

In this subsection we want to see how the action of categories in \coref{Martin} fits the framework of a pseudofunctor $\Tau:\Ll\to\Cat$ from the previous subsection. 
For that purpose we introduce a new bicategory. For the action of categories in \coref{Sch-Balan} the construction is similar. 

We call the bicategory of Tambara modules over quasi-bimonads in $\K$, and denote by ${}_\tau\QB\x\Mod(\K)$, the following bicategory. 

\medskip

\underline{0-cells:} are quintuples $(\A, F, \Phi, B, \tau_{B,F})$ where $(\A, F, \Phi)$ is a quasi-bimonad in $\K$ such that \equref{Phi nat new} holds, 
$(\A, B)$ is a monad in $\K$, $(F, \tau_{B,F})$ is a 1-cell in $\Mnd(\K)$ so that $B$ is a left $F$-comodule monad in $\K$ by \equref{psi_3 for BF}, and \equref{Phi B nat new} holds.

\underline{1-cells:} are quartuples $(X,\tau_{F,X}, \tau_{B,X}, \nu): (\A, F, \Phi, B, \psi_{B,F})\to(\A', F\s', \Phi', B', \psi_{B',F\s'})$ where 
$(X, \tau_{F,X}): (\A, F, \Phi)\to(\A', F', \Phi')$ is a 1-cell in $\QB(\K)$, satisfying identities \equref{monadic d.l.} -- \equref{YBE BBX}, 
$(X, \nu)$ is a left $F\s'$-module so that the action $\nu$ seen as 2-cells: $\nu: (F\s'X, \tau_{F, F\s'X})\to (X, \tau_{F, X})$ and  
$\nu: (F\s'X, \tau_{B, F\s'X})\to (X, \tau_{B, X})$ are 2-cells in $\Mnd(\K)$, meaning that \equref{nat lm -FX} holds both when the unlabeled 1-cell stands for $B$ and $F$, 
and moreover the condition \equref{nat lcm'} holds true:
\begin{center} 
\begin{tabular}{p{5cm}p{0cm}p{5cm}} 
\begin{equation} \eqlabel{nat lm -FX}
\gbeg{3}{5}
\got{1}{} \got{1}{F} \got{1}{X} \gnl
\gcl{1} \glm \gnl
\glmptb \gnot{\tau_{-,X}} \gcmp \grmptb \gnl
\gcl{1} \gvac{1} \gcl{1} \gnl
\gob{1}{X} \gob{3}{}
\gend=
\gbeg{3}{5}
\got{1}{} \got{1}{F} \got{1}{X} \gnl
\glmptb \gnot{\hspace{-0,34cm}\tau_{-,F}} \grmptb \gcl{1} \gnl
\gcl{1} \glmptb \gnot{\hspace{-0,34cm}\tau_{-,X}} \grmptb \gnl
\glm \gcl{1} \gnl
\gvac{1} \gob{1}{X} \gob{1}{}
\gend
\end{equation} & &
\begin{equation} \eqlabel{nat lcm'}
\gbeg{3}{5}
\got{1}{B'} \got{3}{X} \gnl
\gcl{1} \gvac{1} \gcl{1} \gnl
\glmptb \gnot{\tau_{B,X}} \gcmp \grmptb \gnl
\gcl{1} \glcm \gnl
\gob{1}{X} \gob{1}{F} \gob{1}{B} 
\gend=
\gbeg{3}{5}
\got{1}{} \got{1}{B'} \got{1}{X} \gnl
\glcm \gcl{1}\gnl 
\gcl{1} \glmptb \gnot{\hspace{-0,34cm}\tau_{B,X}} \grmptb \gnl
\glmptb \gnot{\hspace{-0,34cm}\tau_{F,X}} \grmptb \gcl{1} \gnl
\gob{1}{X} \gob{1}{F} \gob{1}{B} 
\gend
\end{equation} 
\end{tabular}
\end{center}

\medskip

\underline{2-cells:} $\zeta: (X,\tau_{F,X}, \tau_{B,X}, \nu) \to (Y,\tau_{F,Y}, \tau_{B,Y}, \nu')$ are left $F'$-linear 2-cells 
$\zeta: (X,\tau_{F,X})\to(Y,\tau_{F,Y})$ and $\zeta: (X,\tau_{B,X})\to(Y,\tau_{B,Y})$ in $\Mnd(\K)$. 

\medskip

In the above defined bicategory the identity 1- and 2-cells, as well as the horizontal and vertical composition of 2-cells, are defined in the obvious way (as in $\Mnd(\K)$). 
Let us show how the composition of 1-cells is defined. This will show that ${}_\tau\QB\x\Mod(\K)$ is indeed a bicategory and not a 2-category. Given two composable 1-cells 
$$(\A, F, \Phi, B, \tau_{B,F}) \stackrel{ (X,\tau_{F,X}, \tau_{B,X}, \nu) } {\longrightarrow} (\A', F\s', \Phi', B', \tau_{B',F\s'}) 
\stackrel{ (Y,\tau_{F\s',Y}, \tau_{B',Y}) } {\longrightarrow} (\A'', F\s'', \Phi'', B'', \tau_{B'',F\s''})$$ 
their composition 
is induced by the one in the 2-category $\Mnd(\K)$, where the $F\s''$-action on $YX$ is given by  
\begin{equation} \eqlabel{F act YX tau}
\gbeg{3}{5}
\got{1}{F\s''} \got{3}{YX} \gnl
\gcl{1} \gvac{1} \gcl{2} \gnl
\gcn{1}{1}{1}{3}  \gnl
\gvac{1} \glm \gnl
\gvac{2} \gob{1}{YX} 
\gend=
\gbeg{3}{5}
\got{2}{F\s''} \got{1}{Y} \got{1}{X} \gnl
\gcmu \gcl{1} \gcl{2} \gnl
\gcl{1} \glmptb \gnot{\hspace{-0,34cm}\tau_{F\s',Y}} \grmptb \gnl
\glm \glm \gnl
\gvac{1} \gob{1}{Y} \gob{3}{X.}
\gend
\end{equation}
If $(\A'', F\s'', \Phi'', B'', \tau_{B'',F\s''}) \stackrel{ (Z,\tau_{F\s'',Z}, \tau_{B'',Z}) } {\longrightarrow} (\A''', F\s''', \Phi''', B''', \tau_{B''',F\s'''})$ is 
a third with the previous two composable 1-cell, we define a natural isomorphism $\alpha: (ZY)X\to Z(YX)$ by 
$$
\alpha_{Z,Y,X}=
\gbeg{6}{6}
\gvac{3} \got{1}{Z} \got{1}{Y} \got{1}{X} \gnl
\gcn{1}{1}{2}{1} \gelt{\Phi'''} \gcn{1}{1}{0}{1} \gcl{1} \gcl{2} \gcl{3} \gnl 
\gcl{1} \gcl{1} \glmptb \gnot{\hspace{-0,36cm}\tau_{F\s'',Z}} \grmptb \gnl
\gcl{1} \glmptb \gnot{\hspace{-0,34cm}\tau_{F\s'',Z}} \grmptb \glmptb \gnot{\hspace{-0,34cm}\tau_{F\s',Y}} \grmptb \gcl{1} \gnl
\glm \glm \glm \gnl
\gvac{1} \gob{1}{Z} \gob{1}{} \gob{1}{Y} \gob{1}{} \gob{1}{X.} 
\gend
$$
It is clear that $\alpha_{Z,Y,X}$ is $F\s'''$-linear. The proof that it is an isomorphism 2-cell in  ${}_\tau\QB\x\Mod(\K)$ satisfying the pentagon axiom, and that 
the above composition of 1-cells is well defined, is analogous as in the proof of \thref{quasi-bim monoidal}. 

\medskip

Taking $\Ll={}_\tau\QB\x\Mod(\K)$ in the setting of \ssref{pseudo} it is clear that we recover \coref{Martin}. Observe that $\u{\U}: {}_\tau\QB\x\Mod(\K)\to\Mnd(\K)$ 
given by  $\u{\U}(\A, F, \Phi, B, \tau_{B,F})=(\A, B)$ on 0-cells and by $(X,\tau_{F,X}, \tau_{B,X}, \nu)\mapsto(X, \tau_{B,X}), \zeta\mapsto\zeta$ on 1- and 2-cells 
is a faithful quasi 2-functor. Moreover, that $\u{\F}: {}_\tau\QB\x\Mod(\K)\to\Mnd(\K)$, differing from $\u{\U}$ in that it maps a 1-cell 
$(X,\tau_{F,X}, \tau_{B,X}, \nu)$ into $(X, \psi_{B,X})$, where $\psi_{B,X}$ is given by \equref{psi_3 for BF}, is a quasi 2-functor that factors 
through $\u{\U}$.

\begin{prop} \prlabel{pseudo Martin}
The pseudofunctor $\Tau: \Ll\to \Cat$ from \thref{pseudo bicat} with $\Ll={}_\tau\QB\x\Mod(\K)$ yields the action of categories characterized in \coref{Martin}. 
\end{prop}

\bigskip

\bigskip

{\bf Acknowledgements.}
The author wishes to thank to Ignacio L\'opez for the observation that actions of monoidal categories may be seen as coming from a pseudofunctor between 2-categories. 
This lead the author to develop and add \seref{pseudo}. This research was supported by PEDECIBA and ANII Uruguay.

\end{document}